\documentclass[a4paper,11pt]{amsart}
\usepackage{amsmath, graphicx, amsthm, amssymb}
\usepackage[left=2.5cm,right=2cm, bottom= 3cm]{geometry}
\usepackage{pb-diagram, srcltx}
\usepackage{amsfonts}
\usepackage{amscd}
\usepackage{euscript}
\usepackage{fancyhdr}
\usepackage{float}
\usepackage{subfigure}

\author{Lara Simone Su\'arez}
\title{Exact Lagrangian cobordism and pseudo-isotopy}
\address{D\'epartement de Math\'ematques et de statistique\\ Universit\'e de Montr\'eal, Montr\'eal, QC, Canada.}
\email{suarez@dms.umontreal.ca}
\address{Present address: Institut Camille Jordan
\\ Universit\'e Lyon 1, Villeurbanne cedex, France.}
\email{suarez-lopez@math.univ-lyon1.fr}

\newtheorem{lemma}{Lemma}[section]
\newtheorem{theorem}{Theorem}[section]
\newtheorem{definition}{Definition}[section]
\newtheorem{remark}{Remark}[section]
\newtheorem{proposition}{Proposition}[section]
\newtheorem{corollary}{Corollary}[section]

\newtheorem*{theorem*}{Main Theorem}
\newtheorem*{conjectures}{Symplectic s-cobordism Conjecture}
\newtheorem*{conjecturess}{Conjecture}
\newtheorem*{corollary*}{Corollary}

\begin{document}

\begin{abstract}
We show that under some topological assumptions, an exact Lagrangian cobordism $(W; L_{0}, L_{1})$ of dimension $dim(W) >5$ is a Lagrangian pseudo-isotopy. This result is a weaker form of a conjecture proposed by Biran and Cornea, which states that any exact Lagrangian cobordism is Hamiltonian isotopic to a Lagrangian suspension.
\end{abstract}
\maketitle
\tableofcontents

\section{Introduction}
Let $(M, \omega)$ be a tame symplectic manifold  \cite[Definition 4.1.1]{AuLa}. Consider a pair of closed Lagrangian submanifolds $L_{0}, L_{1}\subset (M, \omega)$. 
In this paper we will concentrate on the equivalence relation of Lagrangian cobordism between Lagrangian submanifolds $L_{0}, L_{1}$.
\begin{definition}\cite[Definition 2.1.1]{BiC}
An (elementary) Lagrangian cobordism $(W; L_{0}, L_{1})$ is a smooth cobordism $W$ between $L_{0}$ and $L_{1}$ admitting a Lagrangian embedding $$W \hookrightarrow  (\widetilde{M}=([0,1]\times \mathbb{R})\times M, (dx \wedge dy) \oplus \omega)$$ such that for some $\epsilon > 0$ we have: 
$$W \cap ([0,\epsilon)\times \mathbb{R}\times M) = [0,\epsilon)\times \{1\} \times L_{0},$$ 
$$W \cap ((1-\epsilon, 1] \times \mathbb{R}\times M) = (1-\epsilon, 1] \times \{1\} \times L_{1}.$$

\end{definition}
The relation of Lagrangian cobordism was introduced by Arnold \cite{Ar1}, \cite{Ar2}. Since then, many authors have studied this relation \cite{Eli}, \cite{Aud},\cite{Che},\cite{NaTa},\cite{BiC}
,\cite{Ta}.
A recent result of Biran and Cornea \cite[Theorem 2.2.2]{BiC}, states that if  a pair of monotone Lagrangian submanifolds $L_{0}, L_{1}$, are related by a monotone Lagrangian cobordism $(W; L_{0}, L_{1})$, then their associated Floer homologies (when defined) are isomorphic, that is $HF(L_{0}; \mathbb{Z}_{2}) \cong HF(L_{1};\mathbb{Z}_{2})$.

A choice of orientation and spin structure on the cobordism $(W; L_{0}, L_{1})$ allows us to extend this result to $\mathbb{Z}(\pi_{1}W)$-coefficients.
Thus, we can conclude that the $\mathbb{Z}(\pi_{1}W)$-Floer homologies of the Lagrangians $L_{0}, L_{1}$ coincide (see the precise statement as Theorem \ref{Teorema}).

Consider an orientable and spin exact Lagrangian cobordism $(W, L_{0}, L_{1})$ satisfying the topological condition:

\textit{\label{hyp} (1)
The morphisms $(j_{0})_{\#}:\pi_{1}(L_{0})\rightarrow \pi_{1}(W)$ and $(j_{1})_{\#}:\pi_{1}(L_{1})\rightarrow \pi_{1}(W)$ induced by the inclusion maps $L_{0}\hookrightarrow W$ and $L_{1}\hookrightarrow W$ are isomorphisms.}

For such a cobordism, Theorem \ref{Teorema} implies that $(W; L_{0}, L_{1})$ is an h-cobordism.
By an h-cobordism we understand a cobordism $(W; L_{0}, L_{1})$ where the inclusion map $L_{i}\hookrightarrow M$ is a homotopy equivalence for $i= 0, 1$.

In summary, an orientable-exact-spin-Lagrangian cobordism satisfying condition \ref{hyp} is an h-cobor-\\dism.
In particular, for a simply connected exact Lagrangian cobordism, the $\mathbb{Z}$-version of \cite[Theorem 2.2.2]{BiC} and the h-cobordism Thoerem of Smale, imply that the cobordism is a pseudo-isotopy i.e. it is diffeomorphic to the trivial cobordism $([0,1]\times L_{0}; L_{0}, L_{0})$. This result appears in \cite{Ta}, for an exact Lagrangian cobordism in a simply-connected cotangent bundle and the proof uses different techniques.

In higher dimensions $(dim(W) > 5)$, the h-cobordisms are classified (up to simple homotopy equivalence) by their Whitehead torsion. For these dimensions, the $s$-cobordism Theorem states that an h-cobordism with trivial Whitehead torsion is a pseudo-isotopy. 

It is natural to ask whether an exact Lagrangian cobordism $(W; L_{0}, L_{1})$ is a Lagrangian pseudo-isotopy (by a Lagrangian pseudo-isotopy we understand a Lagrangian cobordism diffeomorphic to the trivial cobordism $([0,1]\times L_{0}; L_{0}, L_{0})$).  

Given the rigidity of exact Lagrangian submanifolds, and therefore of exact Lagrangian cobordisms, Biran and Cornea proposed the following conjecture:
\begin{conjecturess}[Biran and Cornea, 2012]
Any exact Lagrangian cobordism is Hamiltonian isotopic to a Lagrangian suspension.
\end{conjecturess}

A Lagrangian suspension \cite{AuLa} is a Lagrangian pseudo-isotopy defined by the image of the map \begin{align*}
[0,1]\times L &\rightarrow [0,1] \times \mathbb{R} \times M\\ (t,x)&\rightarrow(t, -H(t,\psi^{t}(x)),\psi^{t}(x))
\end{align*} where $\{\psi^{t}\}_{t\in [0,1]}$ is a Hamiltonian isotopy generated by some Hamiltonian $H$ and $L \subset (M, \omega)$ is some Lagrangian submanifold.

A weaker form of the conjecture of Biran-Cornea is proven in the main theorem of this paper:
\begin{theorem*}
Let $(W; L_{0}, L_{1})$ be an exact, orientable and spin Lagrangian cobordism equipped with a choice of spin structure. Assume $dim(W)>5$.
  
If the map $(j_{i})_{\#}:\pi_{1}(L_{i})\rightarrow \pi_{1}(W)$ induced by the inclusion $L_{i} \hookrightarrow W$ is an isomorphism for $i =0,1$, then $(W; L_{0}, L_{1})$ is a Lagrangian pseudo-isotopy.
\end{theorem*}

The paper is organized as follows. In section \ref{Chapter Floer homology} we study the homology of the $\mathbb{Z}(\pi_{1}W)$-Floer complex of a monotone Lagrangian cobordism. 
The technical aspect of the section concerns the definition of the Floer complex for Lagrangian cobordism, and is treated following the work of Biran-Cornea \cite{BiC}.
As a result, we obtain that a cobordism which satisfies the hypotheses of the Main Theorem is automatically an h-cobordism, and as such it has a well-defined Whitehead torsion.  

In section \ref{Chapter Simple h-theory} we briefly recall some definitions and known results on the Whitehead torsion. In particular, we state the s-cobordism theorem \cite{Ba}, \cite{Ma}, \cite{St}, which will constitute a fundamental ingredient in the proof of the Main Theorem. 

In section \ref{chapitre whitehead} we study the invariance (under horizontal isotopies) of the Whitehead torsion of the $\mathbb{Z}(\pi_{1}W)$-Floer complex associated to an exact Lagrangian cobordism (see Theorem \ref{theorem invariance}). This is, perhaps, the most technical part in the proof of the Main Theorem due to the bifurcation analysis involved. The proof of Theorem \ref{theorem invariance} is based on the work of M. Sullivan \cite{Su} and Lee \cite{Le2}. We apply the technique of stabilization of Sullivan, and a gluing theorem for a degenerate Floer strip, proved by Lee for the Floer theory of Hamiltonian orbits. It follows that if a cobordism satisfies the hypotheses of the Main Theorem, then its torsion agrees with the torsion of the Floer complex. We then apply the displaceability of the cobordism, which yields that the Whitehead torsion of the cobordism vanishes. Hence, we conclude with an application of the s-cobordism theorem, that the cobordism in question, is a Lagrangian pseudo-isotopy. This concludes the proof of the Main Theorem, modulo orientation issues.

Finally, in section \ref{Chapter orientations} we address the orientation issues which arose in sections \ref{Chapter Floer homology} and \ref{chapitre whitehead}. This technical part is based on the work of FOOO \cite{FOOO} and Lee \cite{Le2}.

\begin{center}
\textbf{Acknowledgments}
\end{center}
This paper is based on my Ph.D thesis being carried out under the guidance of Prof. Octav Cornea in Universit\'e de Montr\'eal. This work would have not been possible without all the ideas and support of my great adviser, I am deeply grateful to him. I would like to thank Jean-Fran\c{c}ois Barraud, Baptiste Chantraine and Fran\c{c}ois Lalonde for their interest in this work and helpful conversations. I am grateful to Michael Sullivan for his comments and corrections on my thesis, from which this paper has been extracted. 

\section{Floer homology with $\mathbb{Z}(\pi_{1}W)$-coefficients}
\label{Chapter Floer homology} 
One of the main tools used to study Lagrangian submanifolds is Floer homology. In this section we study Floer homology with twisted coefficients, associated to a monotone Lagrangian cobordism $(W; L_{0}, L_{1})$.

In this section we show that the Floer homology with twisted coefficients is invariant under monotone Lagrangian cobordism, this is the twisted version of \cite[Theorem 2.2.2]{BiC}.

Moreover, when the Lagrangian cobordism is exact, under an additional topological condition the cobordism itself is an h-cobordism.

To a Lagrangian cobordism $(W; L_{0}, L_{1})$ we can associate a Lagrangian with cylindrical ends $\overline{W}$ (see definition 2.2). Let $HF(\overline{W}, \overline{W};\mathbb{Z}(\pi_{1}W),[f_{S}])$ denote the Floer homology with $\mathbb{Z}(\pi_{1}W)$-coefficients of $\overline{W}$, where $[f_{S}]$ denotes the path component of a function $f_{S}$ defining a vertical perturbation on the ends of $\overline{W}$ for $S \in \{\emptyset, L_{0}, L_{1}, L_{0} \cup L_{1}\}$.

The main theorem of this section is the following:
\begin{theorem}
\label{Teorema}
Let $(W; L_{0}, L_{1})$ be an orientable and spin\footnote{See section 5 for definition of spin.}, monotone Lagrangian cobordism equipped with a choice of spin structure. 
Then, $$HF(\overline{W}, \overline{W};\mathbb{Z}(\pi_{1}W),[f_{L_{0}}]) \cong 0 \cong HF(\overline{W},\overline{W};\mathbb{Z}(\pi_{1}W),[f_{L_{1}}]) \text{ and,}$$
$$HF(L_{0};\mathbb{Z}(\pi_{1}W))\cong HF(\overline{W},\overline{W};\mathbb{Z}(\pi_{1}W),[f_{\emptyset/ L_{0}\cup L_{1}}])\cong HF(L_{1};\mathbb{Z}(\pi_{1}W)).$$ 
\end{theorem}
As main consequences we will show the following corollaries:
\begin{corollary}\label{Col1}
If $(W, L_{0}, L_{1})$ is an orientable and spin, exact Lagrangian cobordism and $(j_{i})_{\#}:\pi_{1}(L_{i})\hookrightarrow \pi_{1}(W)$ is an isomorphism for $i= 0,1$, then $(W; L_{0}, L_{1})$ is an h-cobordism.
\end{corollary}

\begin{corollary}\label{Col2}
If $(W; L_{0}, L_{1})$ is an orientable and spin, exact Lagrangian cobordism with $dim(W)> 5$ and $L_{i}$ ($i= 0,1$), $W$ are simply connected, then $(W; L_{0}, L_{1})$ is a Lagrangian pseudo-isotopy.  
\end{corollary}
\begin{remark}
Corollary \ref{Col2} has been shown in \cite{Ta} for the cotangent bundle using different techniques involving previous work of Abouzaid \cite{Ab} and Kragh \cite{Kr}.

To extend the result of corollary \ref{Col2} to the non-simply connected setting, it is necessary the study of the Whitehead torsion, in order to apply the s-cobordism theorem. 
\end{remark}
\subsection{The Floer complex with $\mathbb{Z}(\pi_{1}L_{0})$-coefficients for Lagrangian intersections}
\subsubsection{Setting} 

Given a Lagrangian $L \subset M$ there are two canonical homomorphisms, the symplectic area 
\begin{align*}
\omega:\pi_{2}(M,L)&\longrightarrow \mathbb{R},\\
[u] &\mapsto \omega(u) := \int_{D^{2}} u^{*}\omega
\end{align*}
and the Maslov index:
$$ \mu:\pi_{2}(M,L)\longrightarrow
\mathbb Z.$$
The \textit{minimal Maslov number} of $L$ is the integer $N_{L}$ defined by
$$N_{L} := \text{min} \{\mu(A) > 0 \hspace*{0.2cm} | \hspace*{0.2cm} A\in \pi_{2}(M,L)\}.$$
\begin{definition}A Lagrangian submanifold $L\subset M$ is monotone if there exist a constant $\rho > 0$, such that,$$\omega(A) = \rho \mu(A),$$ for all $A\in \pi_{2}(M,L)$ and $N_{L} \geq 2$. We call $\rho$ the monotonicity constant.
\end{definition}
The Lagrangian cobordisms $(W, L_{0}, L_{1})$, are orientable, monotone ($W \subset (\tilde{M}, \tilde{\omega})$ is a monotone Lagrangian) and spin and we assume that the choice of spin structure on $W$, is such that when we restricted it to $L_{i}$ it coincides with the chosen spin structures on $L_{i}$ for $i = 0,1$.

A fundamental assumption to have a well defined Floer complex is that the minimal Maslov number of any Lagrangian $L$, $N_{L}$ is strictly larger than two \textbf{$N_{L} > 2$}.
We will explain this latter.

\subsubsection{The Floer complex with $\mathbb{Z}(\pi_{1}L_{0})$-coefficients for Lagrangian intersections}
Let $(L_{0}, L_{1}) \subset (M, \omega)$ be a pair of Lagrangian submanifolds. In this subsection we recall the construction of the Floer complex with $\mathbb{Z}(\pi_{1}L_{0})$-coefficients associated to the pair $(L_{0}, L_{1})$. We also show the invariance under Hamiltonian perturbation of the given homology. We follow the exposition in \cite{BiC}.

The Floer complex with twisted coefficients has been studied before in \cite{Su} and a similar version called lifted Floer complex is defined in \cite{Da}.

\textit{Ring of coefficients and set of generators.} Consider the following rings:

\textit{The integral group ring} $\mathbb{Z}(G)$ of the group $G$ (in our case $G =\pi_{1}L_{0}$), is the set of finite formal sums:
$$\mathbb{Z}(G) :=\{ \sum_{i} a_{i}g_{i} \mid n_{i} \in \mathbb{Z}, \hspace{0.2cm} g_{i}\in G\},$$  
with the natural addition and multiplication.

\textit{The universal Novikov ring} is denoted by $\mathcal{A}$; $$\mathcal{A}:=\{ \sum_{k=0}^{\infty}a_{k}T^{\lambda_{k}}\hspace{0.2cm} |\hspace{0.2cm}
a_{k}  \in \mathbb{Z}, \lim_{k\rightarrow \infty}\lambda_{k} = \infty \}$$ with the addition and multiplication of power series.

The set of generators is a subset of the space of paths in $M$ connecting $L_{0}$ with $L_{1}$,
$$\mathcal{P}(L_{0}, L_{1}) = \{ \gamma \in C^{0}([0,1],M)
\hspace{0.2cm}|\hspace{0.2cm} \gamma(0) \in L_{0}, \gamma(1)\in L_{1} \}.$$
In particular $L_{0}\cap L_{1} \subset \mathcal{P}(L_{0}, L_{1})$.
For $\eta \in \pi_{0}(\mathcal{P}(L_{0}, L_{1}))$, let $\mathcal{P}_{\eta}(L_{0},
L_{1})$ denote the path connected component of $\eta$.

If $L_{0}\pitchfork L_{1}$ the set of generators of the complex is $L_{0}\cap L_{1}$. If not, let $H:M\times [0,1]\rightarrow \mathbb{R}$ be a Hamiltonian function with
Hamiltonian flow $\psi^{H}_{t}$ such that $\psi^{H}_{1}(L_{0})$ is transverse to $L_{1}$.

Denote $\mathcal{O}_{\eta}(H) = \{ \gamma \in \mathcal{P}_{\eta}(L_{0} ,L_{1}) \hspace{0.2cm}|\hspace{0.2cm} \exists x \in L_{0}, \gamma(t) =
\psi^{H}_{t}(x)\},$ the space of paths connecting $L_{0}$ with $L_{1}$ that are orbits of the Hamiltonian flow $\psi^{H}_{t}$. 

Fixing $\eta$ and $H$, the $\mathbb{Z}(\pi_{1}L_{0})$-Floer complex is the complex generated by the subset of orbits of the Hamiltonian flow $\psi^{H}_{t}$, $ \mathcal{O}_{\eta}(H)$.
$$CF_{\eta}((L_{0} ,L_{1}; \mathbb{Z}(\pi_{1}L_{0}), H, \mathbf{J}) = (\mathbb{Z}(\pi_{1}L_{0})\otimes\mathcal{A} \langle\mathcal{O}_{\eta}(H)\rangle, d).$$ 

\textit{The differential.} The differential is defined as usual, by counting elements in the moduli spaces of (perturbed) pseudo-holomorphic Floer strips with a weight given by an element in the integral group ring $\mathbb{Z}(\pi_{1}L_{0})$, following the same construction that \cite{Su},\cite{BaC}, and a similar version in \cite{Da}. 

Denote by $\mathcal{J}_{\omega}$ the space of $\omega$-compatible almost complex structures on $M$. We choose a generic, time dependent almost complex structure $\mathbf{J}= \{ J_{t} \}_{t \in [0,1]}$ (with $J_{t} \in \mathcal{J}_{\omega}$ for all $t$).

For any pair $\gamma_{-}, \gamma_{+} \in \mathcal{O}_{\eta}(H)$, denote by $\mathcal{M}(\gamma_{-}, \gamma_{+};H,\mathbf{J})$
the space consisting of maps $u \in
C^{\infty}(\mathbb{R}\times[0,1],M)$ such that:
\begin{enumerate}
\item The map $u$ satisfies the equation:      
\begin{equation}
\frac{\partial u (s,t)}{\partial s} + J_{t}(u(s,t))\frac{\partial u (s,t)}{\partial t} + \nabla^{t}H_{t}(u(s,t)) =0
\end{equation}
\item 
$u(s,i) \in L_{i}$ for $i = 0,1$,
\item 
 $\lim\limits_{s \to -\infty}u(s,t) = \gamma_{-}(t), \lim\limits_{s\to \infty }u(s,t) = \gamma_{+}(t)$,
\item The energy of $E(u)$, is bounded:
$$ E(u) = \int_{\mathbb{R} \times [0,1]} \vert \frac{\partial u (s,t)}{\partial s} \vert dsdt < \infty.$$
\end{enumerate}

The space $\hat{\mathcal{M}}(\gamma_{-}, \gamma_{+};H,\mathbf{J}) = \mathcal{M}(\gamma_{-}, \gamma_{+};H,\mathbf{J})/\mathbb{R}$ denotes the quotient by the $\mathbb{R}$-action on $\mathcal{M}(\gamma_{-}, \gamma_{+};H,\mathbf{J})$ given by reparametrization: $\tau \in \mathbb{R}$ acts by $u \mapsto u(s - \tau, t)$.

Denote by $\mathcal{M}^{n}(\gamma_{-}, \gamma_{+};H,\mathbf{J})$ and $\hat{\mathcal{M}}^{n}(\gamma_{-}, \gamma_{+};H,\mathbf{J})$ their $n$-dimensional component. For a generic choice $(H,\mathbf{J})$ the spaces $\hat{\mathcal{M}}^{0}(\gamma_{-}, \gamma_{+};H,\mathbf{J})$ and $\hat{\mathcal{M}}^{1}(\gamma_{-}, \gamma_{+};H,\mathbf{J})$ are manifolds \cite{Oh},\cite{Le3}. For
$u\in \hat{\mathcal{M}}^{0}(\gamma_{-}, \gamma_{+};H,\mathbf{J})$  
the space \begin{center}
$\hat{\mathcal{M}}^{0}(\gamma_{-}, \gamma_{+};H,\mathbf{J}, [u]) = \{ v \in \hat{\mathcal{M}}^{0}(\gamma_{-}, \gamma_{+};H,\mathbf{J}) \mid [v] = [u] \in \pi_{2}(M, L_{0} \cup \gamma_{-} \cup L_{1} \cup \gamma_{+})\}$
\end{center} 
is a compact manifold.

To involve the integral group ring $\mathbb{Z}(\pi_{1}L_{0})$ in the construction, consider the space obtained from $L_{0}$ by contracting to a point an embedded path $w(t):[0, 1] \rightarrow L_{0}$, passing through each point in $\{ \gamma(0) \hspace*{0.2cm} |  \hspace*{0.2cm} \gamma
\in \mathcal{O}_{\eta}(H)\}$. Denote the resulting space by $L^{*}_{0} = L_{0}/ w \sim *$, and note that it has the same homotopy type as $L_{0}$.

There is a natural map defined as follows:
\begin{equation}\label{def-teta}
 \Theta:\mathcal{M}(\gamma_{-}, \gamma_{+};H,\mathbf{J})\longrightarrow
\Omega(L^{*}_{0}), \hspace*{0.7cm} u\mapsto u(s,0),
\end{equation}
where $\Omega(L^{*}_{0})$ is the loop space of $L^{*}_{0}$.

The differential of the $\mathbb{Z}(\pi_{1}L_{0})$-Floer complex is given by:
\begin{equation}\label{differential}
d(\gamma_{-}) =
\sum_{\gamma_{+} \in \mathcal{O}_{\eta}(H)} \sum_{\hat{u} \in \hat{\mathcal{M}}^{0}(\gamma_{-}, \gamma_{+};H,\mathbf{J})}
\text{sign}(u)[\Theta(u)]T^{\omega(u)}\gamma_{+}.
\end{equation}
Where $u \in \mathcal{M}^{1}(\gamma_{-}, \gamma_{+};H,\mathbf{J})$ is a representative of $\hat{u} \in \hat{\mathcal{M}}^{0}(\gamma_{-}, \gamma_{+};H,\mathbf{J})$. Note that this is well-defined since for a fixed homotopy class $[u]$, the space $\hat{\mathcal{M}}^{0}(\gamma_{-}, \gamma_{+};H,\mathbf{J}, [u])$ is an oriented and compact $0$-dimensional manifold so the expression in the right side of equation (\ref{differential}) makes sense (see section 4 for definition of sign$(u)$). Moreover, the homotopy class of the loop defined by the map $\Theta$ is invariant under the $\mathbb{R}$-action.

We do not discuss the way we can assign a grading to the Floer complex in this setting, but we will do it in the next section when working with exact Lagrangians. 

\begin{remark}
\begin{enumerate}
\item When the pair $(L_{0}, L_{1})$ intersect transversely, we set $H = 0$. In this case the moduli spaces are composed by $J_{t}$-holomorphic strips connecting intersection points.

\item A proof of the regularity of the moduli spaces $\mathcal{M}^{n}(\gamma_{-}, \gamma_{+};H,\mathbf{J})$ in the relative case (Lagrangian boundary condition), can be found in \cite[Section 3.2.2]{Le3}. The proof follows the arguments in \cite{FlHoDi}, adapted to the monotone Lagrangian setting.
\end{enumerate}
\end{remark}
  
\subsubsection{The square of the differential is zero}
There is no additional difficulty to see $d^{2} =0$ in this setting. The proof follows the same argument than in the monotone setting with coefficients ring $\mathcal{A}$. 

The only difference is that if the minimal Maslov number $N_{L_{i}} = 2$ for $i = 0,1$, then we can have $d^{2} \neq 0$.  This can be illustrated by the following example from \cite{Da}. 
\begin{figure}[H]
\includegraphics[scale=0.5]{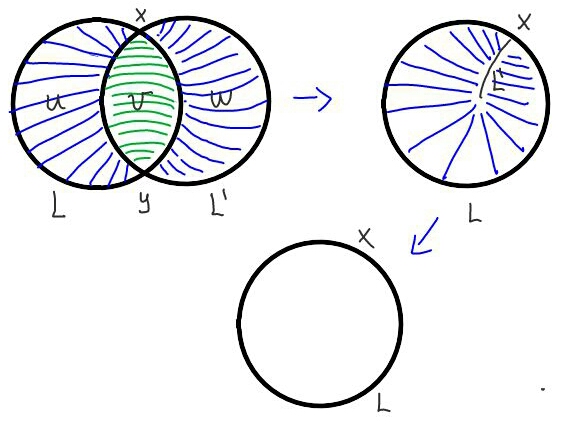}
\caption{Example with $d^{2} \neq 0$.}
\end{figure}
Considering $w(t)$ (the path defining the base point) to be the arc connecting $x$ to $y$ following the counterclockwise orientation. We have $dx = \text{sign}(w)[\Theta(w)]T^{\omega(w)}\neq 0$ and $dy = \text{sign}(v)[\Theta(v)]T^{\omega(v)}\neq 0$, so $d^{2} \neq 0$.

\vspace{0.2cm}
To see $d^{2} =0$ we show the following equation:
\smaller
\begin{equation}\label{d^2=0}
\langle \gamma_{+},d^{2}(\gamma_{-})\rangle =
\sum_{\substack{ \gamma'\\(\hat{u'}, \hat{u''})\in  \hat{\mathcal{M}}^{0}(\gamma_{-}, \gamma')\times
\hat{\mathcal{M}}^{0}(\gamma', \gamma_{+})}}
\text{sign}(u')
\text{sign}(u'')[\Theta(u')][\Theta(
u'')] T^{\omega(u') + \omega(u'')}=0.
\end{equation}
\normalsize
Since the pair $(H, \mathbf{J})$ is assumed to be generic and $N_{L_{i}} > 2$ for $i = 0,1$, the Gromov compactness, the gluing theorem and a choice of compatible orientations (with gluing) on the 1-dimensional unparametrized moduli spaces imply the following decomposition of the moduli spaces for a fixed homotopy class $[u]$: $$\partial\overline{\mathcal{\hat{M}}}^{1}(\gamma_{-}, \gamma_{+};H,\mathbf{J}, [u])=
\bigcup_{\gamma', [u'] + [u''] = [u]} \hat{\mathcal{M}}^{0}(\gamma_{-}, \gamma';H,\mathbf{J}, [u'])\times \hat{\mathcal{M}}^{0}(\gamma',
\gamma_{+};H,\mathbf{J}, [u'']).$$ 
In addition, $[ - ]\circ\Theta$ ( the homotopy class off a the loop given by $\Theta$), is constant on a fixed class $[u]$ and observe that $[\Theta(u')][\Theta(
u'')] = [\Theta(u'\#u'')]$:

For a fixed $\gamma_{0} \in \mathcal{O}_{\eta}(H)$, let $\gamma_{0}(0)$  be the base point of $\Omega(L_{0})$. The map $w(t)$ defines a path connecting $\gamma(0)$ with $\gamma_{0}$, for any $\gamma \in \mathcal{O}_{\eta}(H)$. Denote by $w_{\gamma}(t)$ this path.
Then we have that 
\begin{align*}
[\Theta(u')][\Theta(
u'')] =& [(w_{\gamma_{-}}(t))^{-1}u'(s,0)w_{\gamma'}(t)][(w_{\gamma'}(t))^{-1}u''(s,0)w_{\gamma_{+}}(t)]\\
 =&  [(w_{\gamma_{-}}(t))^{-1}u'(s,0)\#u''(s,0)w_{\gamma_{+}}(t)]\\
 =& [\Theta(u'\#u'')].
\end{align*}

Since the signed sum of all boundary components of a 1-dimensional manifold is zero, we have that (\ref{d^2=0}) is zero for each $\gamma_{+}$ and therefore $d^{2} = 0$.

The homology of this complex is denoted by $HF_{\eta}(L_{0}, L_{1};\mathbb{Z}(\pi L_{0}),(H, \mathbf{J}))$.

The complex $CF(L_{0}, L_{1};\mathbb{Z}(\pi L_{0}),(H;\mathbf{J}))$ is defined by considering the sum over all the components $\eta
\in \pi_{0}(\mathcal{P}(L_{0}, L_{1}))$, 
$$CF(L_{0}, L_{1};\mathbb{Z}(\pi L_{0}),(H,\mathbf{J})) = \oplus_{\eta}CF_{\eta}(L_{0}, L_{1};\mathbb{Z}(\pi L_{0}),(H,\mathbf{J}))$$ 
and the corresponding homology is denoted $HF(L_{0}, L_{1};\mathbb{Z}(\pi L_{0}),(H,\mathbf{J}))$.

The resulting homology is independent on the choice of regular pairs $(H, \mathbf{J})$. The proof of this is completely analogous to the monotone case.

We will recall the proof of the invariance when changing the Hamiltonian $H$, that turns out to be  equivalent (by naturality of Floer equation) to the invariance of the homology under Hamiltonian perturbations of a the Lagrangian $L_{1}$, assuming $L_{0} \pitchfork L_{1}$.

\subsubsection{Invariance under Hamiltonian perturbations}
We follow the exposition in Section 3.2 of \cite{BiC}.

Let $\{\varphi_{t}\}_{t \in [0,1]}$ be a Hamiltonian isotopy with $\varphi_{0}= 1$ associated to a Hamiltonian function $G$.
The isotopy $\varphi_{t}$ induces a map:
\begin{equation}
\varphi_{*}: \pi_{0}(\mathcal{P}(L_{0}, L_{1}))\longrightarrow \pi_{0}(\mathcal{P}(L_{0}, \varphi_{1}(L_{1}))), \hspace*{0.2cm} \eta = [\gamma] \mapsto [\varphi_{t}(\gamma(t))].
\end{equation}

Assume that $L_{0} \pitchfork L_{1}$ and $L_{0} \pitchfork \varphi_{1}(L_{1})$. Fix the data $(\eta, \mathbf{J})$. To compare the homology of the two complexes $CF_{\eta}(L_{0}, L_{1};\mathbb{Z}(\pi_{1}L_{0}),\mathbf{J})$ and $CF_{\phi_{*}\eta}(L_{0}, \varphi_{1}(L_{1}); \mathbb{Z}(\pi_{1}L_{0}),\mathbf{J})$, a chain map
\begin{equation}
\widetilde{c}_{\varphi}:CF_{\eta}(L_{0}, L_{1};\mathbb{Z}(\pi_{1}L_{0}),\mathbf{J})\rightarrow CF_{\varphi_{*}\eta}(L_{0}, \varphi_{1}(L_{1}); \mathbb{Z}(\pi_{1}L_{0}),\mathbf{J})
\end{equation}
is defined using moving boundary conditions. The moduli spaces with moving boundary conditions were introduced by Oh in \cite{Oh}. 

In order to define this map we need to introduce some notation. 

Let $\beta: \mathbb{R}\longrightarrow [0,1]$, be a smooth function such that $\beta(s) = 0$ for $s \leq 0$, $\beta(s) = 1$ for $s\geq 1$ and $\beta$ is strictly increasing in $(0, 1)$.

For $\gamma_{-} \in L_{0}\cap L_{1}$ and  $\gamma_{+} \in L_{0}\cap \varphi_{1}(L_{1})$, consider $\mathcal{M}_{\varphi}(\gamma_{-}, \gamma_{+})$ to be the moduli space of maps $u \in
C^{\infty}(\mathbb{R}\times[0,1],M)$ such that:
\begin{enumerate}
\item For all $s$, $u(s,0) \in L_{0}$ and $u(s,1) \in \varphi_{\beta(s)}(L_{1})$,
\item $\lim\limits_{s\to -\infty}u(s,t) = \gamma_{-}(t), \lim\limits_{s\to \infty}u(s,t) = \gamma_{+}(t)$,
\item $\bar{\partial}_{\mathbf{J}}u = 0$.
\item $E(u) < \infty$.
\end{enumerate} 

Let $\psi_{t}= (\varphi_{t})^{-1}$ and fix $\gamma_{0}$ a path in the connected component $\eta$. Define the functional:
\begin{equation}
 \Phi_{G}:\mathcal{P}_{\varphi_{*}\eta}(L_{0}, \varphi_{1}(L_{1}))\longrightarrow \mathbb{R}, \hspace*{0.3cm} \gamma\mapsto \int_{0}^{1}
G(\psi_{t}(\gamma(t)))dt - \int_{0}^{1} G(\gamma_{0}(t))dt.
\end{equation}
For each map $u \in \mathcal{M}_{\varphi}(\gamma_{-}, \gamma_{+})$ let $v_{u}:\mathbb{R}\times [0, 1]\longrightarrow M$ denote the map defined by $v_{u}(s, t) = \psi_{t\beta(s)}(u(s, t))$.

The chain map is defined as follows:
\begin{equation}
\widetilde{c}_{\varphi}(\gamma_{-}) = \sum_{\gamma_{+} \in L_{0}\cap \varphi_{1}(L_{1})} \big(\sum_{u \in
\mathcal{M}_{\varphi}^{0}(\gamma_{-}, \gamma_{+})}\text{sign}(u)[\Theta(u)]T^{\omega(v_{u})-
\Phi_{G}(\gamma_{+})} \big)\gamma_{+}.
\end{equation}
In order to obtain a loop $\Theta(u) \in \Omega(L_{0})$, we consider the space $L_{0}/w\sim *$ as in the previous section, where the embedded path $w(t):[0, 1] \rightarrow L_{0}$ passes through each point in $(L_{0}\cap L_{1})\cup (L_{0}\cap\varphi_{1}(L_{1}))$.

To see that the map $\widetilde{c}_{\varphi}(\gamma_{-})$ is a chain map we compare the following equations: 
\begin{align*}
&\langle d(\widetilde{c}_{\varphi}(\gamma_{-})),\gamma_{+}\rangle =\\
& \sum_{\substack{ \gamma' \in L_{0}\cap \varphi_{1}(L_{1})\\(u,\hat{w})\in
\mathcal{M}_{\varphi}^{0}(\gamma_{-}, \gamma')\times \hat{\mathcal{M}}^{0}(\gamma', \gamma_{+})}}\text{sign}(u)\text{sign}(w)[\Theta(u)][\Theta(w)] T^{\omega(v_{u})+\omega(w)-\Phi_{G}(\gamma')}\\
&\langle\widetilde{c}_{\varphi}(d(\gamma_{-})),\gamma_{+}\rangle =\\
& \sum_{\substack{ \gamma'' \in L_{0}\cap L_{1}\\(\hat{u'},w')\in
\hat{\mathcal{M}}^{0}(\gamma_{-}, \gamma'')\times \mathcal{M}_{\varphi}^{0}(\gamma'', \gamma_{+})}}
\text{sign}(u')\text{sign}(w')[\Theta(u')][\Theta(w')]
T^{\omega(v_{w'})+\omega(u')-\Phi_{G}(\gamma_{+})}.
\end{align*}
For generic $\mathbf{J}$ and fixed homotopy class $A \in \pi_{2}(M, L_{0}\cup L_{1}\cup \gamma_{-}\cup \gamma_{+})$, the space $\mathcal{M}_{\varphi}^{1}(\gamma_{-},
\gamma_{+};\mathbf{J}, A)$ admits a compactification into a $1$-dimensional manifold $\mathcal{\overline{M}}_{\varphi}^{1}(\gamma_{-},
\gamma_{+};\mathbf{J}, [u])$, with boundary given by:
\begin{center}
$\bigcup\limits_{\gamma', B + C = A}
\mathcal{M}_{\varphi}^{0}(\gamma_{-}, \gamma';\mathbf{J}, B)\times
\hat{\mathcal{M}}^{0}(\gamma',\gamma_{+};\mathbf{J}, C)$\\
$\cup$\\
$\bigcup\limits_{\gamma'', B + C = A}
\hat{\mathcal{M}}^{0}(\gamma_{-}, \gamma'';\mathbf{J},B)\times
\mathcal{M}_{\varphi}^{0}(\gamma'',\gamma_{+};\mathbf{J}, C ).$ 
\end{center}
The choices of orientations and spin structures on $L_{0}, L_{1}$ induce canonical orientations on this spaces, compatible with the gluing map. 
Moreover, any two boundary components $(u,\hat{w}), (\hat{u'}, w')$ of $\mathcal{\overline{M}}_{\varphi}^{1}(\gamma_{-}, \gamma_{+};\mathbf{J}, [u])$ in the same connected component satisfy $[\Theta(u)][\Theta(w)] = [\Theta(u')][\Theta(w')]$ and $\omega(v_{w'})+\omega(u')-\Phi_{G}(\gamma_{+})=
\omega(v_{u})+\omega((\psi_{t}(w))-\Phi_{G}(\gamma_{+})= \omega(v_{u})+
E((\psi_{t}(w))-\Phi_{G}(\gamma') = \omega(v_{u})+\omega(w)-\Phi_{G}(\gamma').$
Therefor, the sum of the two expressions above vanishes.
In a similar way we define a chain map \begin{equation}
\widetilde{c'}_{\varphi^{-1}}:CF_{\eta}(L_{0}, \varphi_{1}(L_{1}); \mathbb{Z}(\pi_{1}L_{0}),\mathbf{J}), \rightarrow CF_{\varphi^{-1}_{*}\eta}(L_{0}, L_{1};\mathbb{Z}(\pi_{1}L_{0}),\mathbf{J}) 
\end{equation}
such that $\widetilde{c'}_{\varphi^{-1}}$ is a quasi-inverse of $\widetilde{c}_{\varphi}$, concluding that the map;
$$c_{\varphi}:HF_{\eta}(L_{0}, L_{1}; \mathbb{Z}(\pi_{1}L_{0}),\mathbf{J})\rightarrow HF_{\varphi_{*}\eta}(L_{0}, \varphi_{1}(L_{1}); \mathbb{Z}(\pi_{1}L_{0}),\mathbf{J}),$$
is an isomorphism.

We summarize this section in the following proposition: 
\begin{proposition}\label{PropHF}
Let $(L_{0}, L_{1})$ be a pair of oriented and spin  monotone Lagrangians, equipped with a choice of spin structure. Assume that the minimal Maslov number satisfies $N_{L_{i}}> 2$ for $i = 0,1$. Let $(H,\mathbf{J})$ be generic, then the Floer complex $CF(L_{0}, L_{1}; \mathbb{Z}(\pi_{1}L_{0}),(H, \mathbf{J}))$ is a chain complex.\\
Moreover, the homology $H(CF(L_{0}, L_{1}; \mathbb{Z}(\pi_{1}L_{0}),(H, \mathbf{J})))$ does not depend on the choice of generic pair $(H,\mathbf{J})$ and is invariant under Hamiltonian isotopies of $L_{0}$ or $L_{1}$.
\end{proposition}

\begin{remark}
Proposition \ref{PropHF} is a $\mathbb{Z}(\pi_{1}L_{0})$-version of the same result due to Oh \cite{Oh} in the monotone setting. 

For twisted coefficients, the $\mathbb{Z}_{2}(\pi_{1}L_{0})$-version of proposition \ref{PropHF} appears in \cite{Su} for a pair $(L_{0}, L_{1})$ of non-compact Lagrangians with $L_{1} = \phi^{H}_{1}(L_{0})$ and $\omega\mid_{ \pi_{2}(M, L_{0})} = 0$.

In the monotone setting, a similar result for the lifted Floer homology appears in \cite{Da}.
\end{remark}

\subsection{Lagrangians with cylindrical ends}

In this subsection we adapt the chain complex defined in the last section in order to associate such a complex to a
Lagrangian cobordism $(W; L_{0}, L_{1})$. Instead of working with a Lagrangian cobordism we use the corresponding non-compact Lagrangian with cylindrical ends (obtained when extending the ends of the cobordism). The Floer theory for Lagrangian with cylindrical ends was developed in \cite{BiC}.

Given a symplectic manifold $(M, \omega)$, denote by $(\widetilde{M}, \widetilde{\omega})$ the symplectic manifold composed by  $\widetilde{M} = \mathbb{R}^{2}\times M$ and $\widetilde{\omega} = \omega_{st}\oplus \omega$. Denote by $\pi$ the projection $\pi: \mathbb{R}^{2}\times M \rightarrow \mathbb{R}^{2}$. 

\begin{definition}\cite[Section 4.1]{BiC}{
An (elementary) Lagrangian submanifold with cylindrical ends, $\overline{W} \subset (\widetilde{M}, \tilde{\omega})$ is a Lagrangian submanifold without boundary and with the following properties:
\begin{enumerate}
 \item For every $a < b$ the subset $\overline{W}\cap ([a, b]\times \mathbb{R})\times M$ is compact.
\item There exists $R_{-}, R_{+} \in \mathbb{R}$ with $R_{-} \leq R_{+}$ such that:
$$\overline{W}\cap ([R_{+}, \infty)\times \mathbb{R})\times M = ([R_{+}, \infty)\times \{a_{+}\})
\times L_{1}$$ 
$$\overline{W}\cap ((-\infty, R_{-}]\times \mathbb{R})\times M = ((-\infty ,R_{-}]\times
\{a_{-}\})\times L_{0}$$
for some pair of Lagrangian $L_{0 }, L_{1} \subset
M$ and $a_{-}, a_{+} \in \mathbb{R}$.
\end{enumerate}
}\end{definition}

We use the term ``elementary`` since the general definition appearing in \cite{BiC}, considers Lagrangians with more than one positive or negative end.

For every $R \geq R_{+}$, let $E^{+}_{R}(\overline{W}) = \overline{W}\cap ([R, \infty)\times \mathbb{R})\times M$ denote the positive cylindrical end of $\overline{W}$, and similarly for $R \leq R_{-}$, $E^{-}_{R}(\overline{W})$ will denote the negative
cylindrical end of $\overline{W}.$

\begin{definition}
Let $\overline{W}_{0}, \overline{W}_{1} \subset \widetilde{M}$ be two Lagrangians with cylindrical ends. We say that they are
cylindrically distinct at infinity if there exists $R > 0$ such that $\pi(E^{-}_{-R}(\overline{W}_{0}))\cap
\pi(E^{-}_{-R}(\overline{W}_{1})) = \emptyset$ and $\pi(E^{+}_{R}(\overline{W}_{0}))\cap \pi(E^{+}_{R}(\overline{W}_{1})) =
\emptyset.$ 
\end{definition}

Any Lagrangian cobordism $(W; L_{0}, L_{1})$ extends to a Lagrangian with cylindrical ends $\overline{W}$ in the following way;$$\overline{W} = ((-\infty ,0]\times \{1\}\times L_{0}) \cup W \cup ([1, \infty)\times \{1\}\times L_{1}).$$ In the class of Lagrangians with cylindrical ends we consider Hamiltonian isotopies that "fix" the cylindrical ends in the sense of the following definition:

\begin{definition}\cite[Definition 4.1.2]{BiC} An isotopy $\{\overline{W}_{t}\}_{t\in[0,1]}$ of Lagrangian submanifolds with cylindrical ends of
$\widetilde{M}$ is called Horizontal isotopy if there exists a Hamiltonian isotopy $\{\psi_{t}\}_{t\in [0,1]}$ of
$\widetilde{M}$, with $ \psi_{0}= \mathbb{I}$ and the following properties:
\begin{itemize}
 \item $\overline{W}_{t} = \psi_{t}(\overline{W}_{0})$ for all $t\in [0,1]$.
\item There exist real numbers $R_{-} < R_{+}$ and a constant $K > 0$, such that for all $t\in [0,1]$ and $x \in E^{\pm}_{R_{\pm}}(\overline{W}_{0})$,
we have  $\psi_{t}(x) \in E^{\pm}_{R_{\pm}\mp K}(\overline{W}_{0})$.
\item For all $x \in E^{\pm}_{R_{\pm}}(\overline{W}_{0})$, $|d\pi_{x}(X_{t}(x))|< K.$
Here, $X_{t}$ is the time dependent vector field of the flow $\{\psi_{t}\}_{t\in [0,1]}$.
\end{itemize}
\end{definition}

We now proceed to define the $\mathbb{Z}(\pi_{1}W_{1})$-Floer complex associated to a pair of monotone Lagrangians with cylindrical ends $(\overline{W}_{0}, \overline{W}_{1})$.

\subsubsection{The $\mathbb{Z}(\pi_{1}W_{1})$-Floer complex for Lagrangian with cylindrical ends}

Consider $(\overline{W}_{0}, \overline{W}_{1})$ a pair of monotone Lagrangians with cylindrical ends with $N_{\overline{W}_{i}} > 2$ for $i = 0,1$.
In \cite{BiC} it is shown that the Floer homology with $\mathbb{Z}_{2}$-coefficients, $HF(\overline{W}_{0},\overline{W}_{1}, [f])$ is well defined and depends on an additional data $[f]$, coming from the choice of a Hamiltonian perturbation of the ends of
$\overline{W}_{1}$, which makes the image of $\overline{W}_{1}$ under this perturbation cylindrically distinct from $\overline{W}_{0}$ at infinity.

When working with $\mathbb{Z}(\pi_{1}W_{0})$-coefficients, the construction is completely analogous under additional choices of spin structures on the pair  $(\overline{W}_{0}, \overline{W}_{1})$. The construction of the Floer complex follows the schema presented in the previous section. 

Given the non-compactness of the Lagrangian pair, additional choices have to be made to ensure the  compactness of the moduli spaces. 

The Floer complex is defined using the following data:

\begin{itemize}
 \item A fixed component $\eta \in \pi_{0}(\mathcal{P}(\overline{W}_{0}, \overline{W}_{1}))$.
\item A perturbation $(H,f)$ where $H:[0,1]\times\widetilde{M}\rightarrow \mathbb{R}$ and
$f:\mathbb{R}^{2}\rightarrow \mathbb{R}$ are two Hamiltonians satisfying:
\begin{enumerate}
\item $H$ has compact support.
\item For $R_{\pm}$ big enough so that $\overline{W}_{1}$ is cylindrical on $E^{\pm}_{R_{\pm}}(\overline{W}_{1})$, the function $f$ is such that the support of $f$ is contained in a neighborhoods $U_{\pm}$ of $\pi(E^{\pm}_{R_{\pm}}(\overline{W}_{1}))$ where $f(x,y) = \alpha_{\pm}x + \beta_{\pm}$ with $\alpha_{\pm}\in \mathbb{R}$. 
\end{enumerate}
We denote the space of pairs $(H,f)$ as above by
$\mathcal{H}(\overline{W}_{0},\overline{W}_{1})$.
\item An almost complex structure $\mathbf{\widetilde{J}}$ on $(\widetilde{M}, \widetilde{\omega})$. We will restrict to a family of time dependent complex structure $\mathcal{\widetilde{J}}_{B}$ where $B \subset \mathbb{R}^{2}$ is a compact set, and $\mathbf{\widetilde{J}}=\{
\widetilde{J}_{t} \}_{t \in [0,1]}\in \mathcal{\widetilde{J}}_{B}$ satisfies the following properties:
\begin{enumerate}
\item For every $t$, $\widetilde{J}_{t}$ is an $\widetilde{\omega}$-tamed almost complex structure on $\widetilde{M}$.
\item For every $t$, the projection $\pi$ is $(\widetilde{J}_{t}, i)$-holomorphic on $(\mathbb{R}^{2}\setminus
B)\times M$. If $B = \emptyset $, we write $\mathcal{\widetilde{J}}$. 
\end{enumerate}
\end{itemize} 

The additional choice of perturbation $(H, f) \in \mathcal{H}(\overline{W}_{0},\overline{W}_{1})$ guaranties that $\overline{W}_{0}$ and $\phi^{f\circ \pi}_{1}(\overline{W}_{1})$ are cylindrically distinct at infinity.
If the ends of $\overline{W}_{0}$ coincide with those of $\overline{W}_{1}$, the space $\mathcal{H}(\overline{W}_{0},\overline{W}_{1})$ has four connected components fixed by the four possible choices of the function $f$ (depending on the signs of $\alpha_{\pm}$). We denote by $[f]$ the path component of $\mathcal{H}(\overline{W}_{0},\overline{W}_{1})$ associated to a perturbation $(H,f)$. 

The choice of almost complex structure $\mathbf{\widetilde{J}} \in
\mathcal{\widetilde{J}}_{B}$ implies the compactness of the moduli spaces of $J_{t}$-holomorphic strips (in a fixed relative to the boundary homotopy class), since an application of the open mapping theorem shows that any pseudo-holomorphic curve with finite energy has its image in a fixed compact set. For the space of perturbed $J_{t}$-holomorphic strips, with boundary on a Lagrangian pair $(\overline{W}_{0},\overline{W}_{1})$ cylindrically distinct at infinity, we apply a Hamiltonian perturbation $\psi_{1}^{H}$ to the second Lagrangian $\overline{W}_{1}$, such that the pair $(\overline{W}_{0},\psi_{1}^{H}({\overline{W}_{1}}))$ intersect transversely. After a naturality argument we can conclude that the moduli spaces of perturbed $J_{t}$-holomorphic strips is compact. For a detailed proof of compactness in the cobordism setting, see \cite[Lemma 4.2.1]{BiC}.

For generic data $((H, f), \mathbf{\widetilde{J}})$, the Floer complex 
$$CF_{\eta}(\overline{W}_{0},\overline{W}_{1}; \mathbb{Z}(\pi_{1}\overline{W}_{0}),(H, f), \mathbf{\widetilde{J}}) :=
CF_{(\phi^{f\circ \pi})_{*}\eta}(\overline{W}_{0}, \phi_{1}^{f\circ \pi}(\overline{W}_{1}); \mathbb{Z}(\pi_{1}\overline{W}_{0}),H, \mathbf{\widetilde{J}})$$ is well-defined \cite[Section 4.3]{BiC}. As in previous section, $CF(\overline{W}_{0},\overline{W}_{1}; \mathbb{Z}(\pi_{1}\overline{W}_{0}),(H, f), \mathbf{\widetilde{J}})$ denotes the sum along all the class of Hamiltonian chords $\eta \in \pi_{0}(\mathcal{P}(\overline{W}_{0},\overline{W}_{1}))$. From now on, we make no distinction in the notation of a cobordism $(W; L_{0}, L_{1})$ and its associated Lagrangian with cylindrical ends $\overline{W}$, unless it is necessary.

\begin{proposition}\label{Prop}
Let $(W_{0}, W_{1})$ be a pair of oriented and spin monotone Lagrangians with cylindrical ends equipped with a choice of spin structure. Assume that the minimal Maslov number satisfies $N_{W_{i}}> 2$ for $i = 0,1$.
 Let $((H, f), \mathbf{\widetilde{J}})$ be generic.  
 
Then the Floer complex $CF(W_{0},W_{1}; \mathbb{Z}(\pi_{1}W_{0}),(H, f), \mathbf{\widetilde{J}})$ is a chain complex. Moreover, its homology $H(CF(W_{0},W_{1}; \mathbb{Z}(\pi_{1}W_{0}),(H, f), \mathbf{\widetilde{J}}))$ does not depend on the choice of generic pair $(H,\mathbf{\widetilde{J}})$ but depends on the path connected component $[f] \in \pi_{0}(\mathcal{H}(W_{0}, W_{1}))$ and is invariant under horizontal isotopies of $W_{0}$ or $W_{1}$, up to an isomorphism.
\end{proposition}
Once the compactness of the moduli spaces is guarantied, the proof of this proposition follows from last section. The only additional argument concerns the invariance of the homology for different perturbations $f,f' \in \pi_{0}(\mathcal{H}(W_{0}, W_{1}))$ with $[f] = [f']$. This issue is addressed in the proof of the analogous statement for $\mathbb{Z}_{2}$-coefficients, \cite[Proposition 4.3.1]{BiC}. The idea is to study a chain map defined using moving boundary conditions induced by an homotopy $f_{\tau} = \beta(\tau)f +(1-\beta(\tau))f'$, connecting $f$, with $f'$. 

The homology of the complex $CF(W_{0},W_{1}; \mathbb{Z}(\pi_{1}W_{0}),(H, f), \mathbf{\widetilde{J}})$ is denoted by $$HF(W_{0}, W_{1};\mathbb{Z}(\pi_{1}W_{0}),[f]).$$

\subsubsection{Proof of Theorem \ref{Teorema}}

\begin{proof}[Proof of theorem \ref{Teorema}]
This theorem is a consequence of Proposition \ref{PropHF} and Proposition \ref{Prop}. Let $(W,L_{0}, L_{1})$ be an orientable, spin, monotone, Lagrangian cobordism viewed as a Lagrangian with cylindrical ends.
Let $S$ be an element in the set $\{\emptyset, L_{0}, L_{1}, L_{0}\cup L_{1}\}$. Consider $(H, f_{S}) \in \mathcal{H}(W, W)$ a perturbation where $f_{S}$ is locally given by $f_{S}(x, y) = \alpha_{\pm}x + \beta_{\pm}$ with the property that $\alpha_{+} > 0$ if $L_{1} \subset S$ and $\alpha_{+} < 0$ if not, as well as $\alpha_{-} < 0$ if $L_{0} \subset S$ and $\alpha_{-} > 0$ if not. Such a perturbation makes $W$ and $\phi^{f\circ \pi}_{1}(W)$ cylindrically distinct at infinity. Denote by $W'$ the Lagrangian $\phi^{f_{S}\circ \pi}_{1}(W)$.

Consider the Floer complex $CF(W,W'; \mathbb{Z}(\pi_{1}W),(H, \mathbf{\widetilde{J}}))$. From Proposition \ref{Prop} follows that \\$HF(W , W';\mathbb{Z}(\pi_{1}W),[f_{S}])$ depends only on $[f_{S}]$.

For $S = L_{1}$, notice that there is an horizontal isotopy $\{\psi_{t}(W')\}$ such that $W \cap \psi_{1}(W') = \emptyset$: \begin{figure}[H]
\includegraphics[scale=0.50]{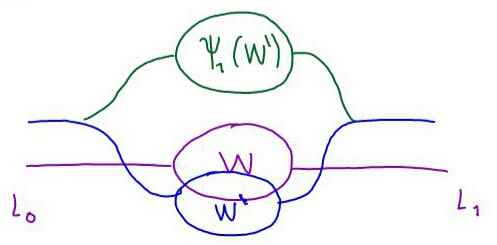}
\caption{$(W, \phi^{f_{L_{0}}\circ \pi}_{1}(W))$}
\end{figure} 

In a similar way, when $S = L_{0}$ we can also find a horizontal isotopy $\{\psi'_{t}(W')\}$ such that $W \cap \psi'_{1}(W') = \emptyset$. Thus,
 $$CF(W, \psi_{1}(W'); \mathbb{Z}(\pi_{1}W),(H,  \mathbf{\widetilde{J}})) = 0 = CF(W, \psi_{1}'(W'); \mathbb{Z}(\pi_{1}W),(H, \mathbf{\widetilde{J}})).$$

From the invariance of the homology under horizontal isotopies (Proposition 2.2), we can conclude that $$HF(W, W;\mathbb{Z}(\pi_{1}W),[f_{L_{0}}]) \cong 0 \cong HF(W,W;\mathbb{Z}(\pi_{1}W),[f_{L_{1}}]).$$

Consider now the case when $S = \emptyset$, (the case $S = \{L_{0}, L_{1}\}$ is completely analogous). Notice that there are horizontal isotopies in the same path component $\{\psi_{t}(W')\}, \{\psi'_{t}(W')\}$ (see figure), such that:
 $$CF(W, \psi_{1}(W'); \mathbb{Z}(\pi_{1}W),(H,  \mathbf{\widetilde{J}})) = CF(L_{0}, \psi_{1}(L_{0}); \mathbb{Z}(\pi_{1}W),\mathbf{\widetilde{J}}),$$ and in the same way $$CF(W, \psi_{1}'(W'); \mathbb{Z}(\pi_{1}W),(H, \mathbf{\widetilde{J}})) = CF(L_{1}, \psi_{1}'(L_{1}); \mathbb{Z}(\pi_{1}W),\mathbf{\widetilde{J}}).$$ 
\begin{figure}[H]
\includegraphics[scale=0.50]{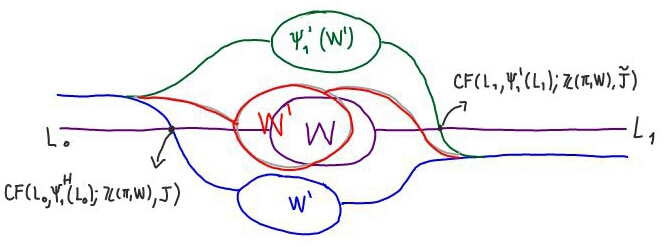}
\caption{$(W, \phi^{f_{\emptyset}\circ \pi}_{1}(W))$}
\end{figure} 
 
From the invariance of the homology under horizontal isotopies (Proposition 2.2), we can conclude that
$$H(CF(L_{0}, \psi_{1}(L_{0}); \mathbb{Z}(\pi_{1}W),\mathbf{\widetilde{J}}))\cong H(CF(L_{1}, \psi_{1}'(L_{1}); \mathbb{Z}(\pi_{1}W),\mathbf{\widetilde{J}})).$$ 
Moreover, the relative homology is independent of the choice of almost complex structure. Denoting by $HF(L_{i};\mathbb{Z}(\pi_{1}W))$ the homology of the complex $CF(L_{i}, \psi_{1}(L_{i}); \mathbb{Z}(\pi_{1}W),\mathbf{J})$, for $i= 0, 1$, we have that 
$$HF(L_{0};\mathbb{Z}(\pi_{1}W))\cong  HF(L_{1};\mathbb{Z}(\pi_{1}W)).$$  
\end{proof}
 
\subsubsection{Exact Lagrangians with cylindrical ends}

In this section we apply the previous results to the class of exact Lagrangian cobordisms.
\begin{definition}
Let $(M, \omega)$ be a symplectic manifold where the two-form $\omega  = d\lambda$ is exact. If the one-form restricted to a Lagrangian $L$ is also exact $\lambda \mid_{TL} = df$, then we say that the Lagrangian $L$ is an exact Lagrangian. 

An exact Lagrangian cobordism is a Lagrangian cobordism $(W; L_{0}, L_{1})$ where $W \subset (\tilde{M}, \tilde{\omega})$ is an exact Lagrangian.
\end{definition}
For an exact Lagrangian submanifold $L \subset (M, \omega)$ the Floer homolgy $HF(L)$ is known to be isomorphic to the Morse homology $H(CM(L,(f, g)))$. Here $CM(L,(f, g))$ denotes the Morse complex of $L$, $(f, g)$ denotes a Morse-Smale pair composed by a Morse-Smale function $f$ and a Riemannian metric $g$. We can expect the $\mathbb{Z}(\pi_{1}L)$-Floer homology  $HF(L;\mathbb{Z}(\pi_{1}L))$ to be related to the homology of the $\mathbb{Z}(\pi_{1}L)$-coefficients Morse complex $H(CM(L,;\mathbb{Z}(\pi_{1}L)(f, g))$.

\begin{definition}\cite[Definition 1.4]{CoR}The $\mathbb{Z}(\pi_{1}L)$-coefficients Morse complex $CM(L;\mathbb{Z}(\pi_{1}L),(f, g))$ is the based finite generated $\mathbb{Z}(\pi_{1}L)$-module chain complex given by 
\begin{align*}
d: CM_{*}(L;\mathbb{Z}(\pi_{1}L),(f, g))= \mathbb{Z}(\pi_{1}L)(Crit_{*}(f)) &\rightarrow CM_{*-1}(L;\mathbb{Z}(\pi_{1}L),(f, g))=\mathbb{Z}(\pi_{1}L)(Crit_{*-1}(f))\\
p &\mapsto \sum_{ q \in Crit(f)_{*-1}}( \sum_{ \alpha \in \pi_{1}L} n^{\tilde{f}, \tilde{g}}(\tilde{p}, \alpha\tilde{q}) \alpha) q.
\end{align*} 
Here $(\tilde{f}, \tilde{g})$ denotes the pullback of $(f, g): L \rightarrow \mathbb{R}$ and $n^{\tilde{f}, \tilde{g}}(\tilde{p}, \alpha\tilde{q})$ denotes the signed sum of the number of connecting flow lines of $(\tilde{f}, \tilde{g})$ between $\tilde{p}$ and $\alpha\tilde{q}$.
\end{definition}
Notice that the homology of the $\mathbb{Z}(\pi_{1}L)$-coefficients Morse complex $H(CM(L;\mathbb{Z}(\pi_{1}L)),(f, g))$ is isomorphic to the singular homology of the universal covering of $L$, denoted by $H(\tilde{L})$. This follows from the observation that the lifted Morse-Smale function defines a cellular decomposition on $\tilde{L}$.

\begin{proposition}
Let $L  \subset (M, \omega)$ be an orientable and spin exact Lagrangian submanifold. Then $HF_{*}(L;\mathbb{Z}(\pi_{1}L)) \cong H_{*}(\widetilde{L})$.
\end{proposition}
Before we start the proof of this proposition, we will define the version of the Floer complex for exact Lagrangians used here. For this section we follow the exposition in \cite{Au}.

\textit{Coefficient ring.}

In the definition of the Floer complex we have used the universal Novikov ring $\mathcal{A}$, when working with exact Lagrangians this ring is not necessary anymore.

Consider a pair of exact Lagrangians $(L_{0}, L_{1}) \subset (M, d\lambda)$. There are functions $f_{i}: C^{\infty}(L_{i}; \mathbb{R})$ such that $\lambda\mid_{L_{i}} = df_{i}$ for $i = 0,1$. The symplectic area $\omega(u)$ of any strip with boundary conditions on $L_{0}\cup L_{1}$ connecting two intersection points $x, y$, is given by $\int u^{*}d\lambda = (f_{1}(y)-f_{1}(x))-(f_{0}(y)-f_{0}(x))$. Then, there are uniform bounds on the energy in a fixed moduli space $\mathcal{M}(x,y)$, independently of the homotopy class of the strip $[u] \in \pi_{2}(M, L_{0} \cup L_{1}\cup x \cup y)$. Thus, is not necessary to consider the universal Novikov ring $\mathcal{A}$ in this case. 

The module used in the definition of the Floer complex $CF(L_{0};\mathbb{Z}(\pi_{1}(L_{0}), (H, \mathbf{J}))$ is \\$\mathbb{Z}(\pi_{1}L_{0})(L_{0}\cap L_{1})$. Here, $L_{1} = \psi^{H}_{1}(L_{0})$ is the image of $L_{0}$ under the time-one Hamiltonian diffeomorphism associated to $H$, such that $L_{0}\pitchfork L_{1}$.

\textit{Grading.}

In the module $\mathbb{Z}(\pi_{1}L_{0})(L_{0}\cap L_{1})$, there is a natural graduation defined using the Maslov-Viterbo index.

Let $u: \mathbb{R}\times [0, 1] \rightarrow M$ be a strip with boundary conditions on $L_{0}, L_{1}$ joining two intersection points $x,y \in L_{0}\cup L_{1}$. The Maslov-Viterbo index of $u$, denoted by $\mu(u)$, is the Maslov index of a loop obtained from a trivialization of the bundle $u ^{*}(TM)$. Let $\lambda: [-\infty, \infty ] \rightarrow L_{0}$, $\lambda': [-\infty, \infty ] \rightarrow L_{1}$ be two paths defined by $\lambda(s) = u(s,0)$ and $\lambda'(s) = u(s,1)$. Consider the path of Lagrangians $T_{\lambda(s)}L_{0}$ concatenated with a path of Lagrangians $\lambda''(s) \subset T_{\lambda'(s)}M$ which intersect transversely $T_{\lambda'(s)}L_{1}$ at each $s$. The Maslov-Viterbo index of $u$, is the Maslov index of the loop $T_{\lambda(s)}L_{0}\#(-\lambda''(s))$ after trivialization.

Fix a point $x_{0} \in (L_{0}\cap L_{1})$. Consider the path-connected component of $x_{0}$ in $\mathcal{P}(L_{0}, L_{1})$ denoted by $\mathcal{P}_{x_{0}}(L_{0}, L_{1})$. For any other point $y \in (L_{0}\cap L_{1})\cap \mathcal{P}_{x_{0}}(L_{0}, L_{1})$ we define $\mu(y, x_{0}) = \mu(\tilde{y})$ to be the Maslov-Viterbo index of a path $\tilde{y}$ in $\mathcal{P}_{x_{0}}(L_{0}, L_{1})$ connecting
$y$ with $x_{0}$. This is independent on the choice of path $\tilde{y}$ since the condition $\omega\mid_{\pi_{2}(M, L_{0})}=0$ implies that, for a small Hamiltonian perturbation, all Floer strips defining the Floer homology  are contained in a small tubular neighborhood of $L_{0}$. Thus they can be seen as strips on the cotagent bundle      $T^{*}L_{0}$. For the cotangent bundle the Maslov index is independent on the homotopy class of the strip. 
Assuming that $\mu(x_{0}) = k \in \mathbb{Z}$, we set $\mu(y) := \mu(y, x_{0}) - k$.
\begin{proof}[Proof of Proposition 2.2] 
By the Weinstein theorem there exists a neighborhood of $L$, $U_{L}$ symplectomorphic to a tubular neighborhood of the zero-section. Let $f: L \rightarrow \mathbb{R}$ be a Morse function and for a small $\epsilon > 0$ let $L' = \text{graph}(\epsilon df)$. Note that $L' \subset (T^{*}L, \sum dx_{i} \wedge dy_{i})$ is Hamiltonian isotopic to $L$. The set of intersection points $L\pitchfork L'$ is the set $Crit(f)$. The relation between the Floer and Morse grading is given by $\mu(x)= |x|-\mu(x_{0})$, where $|x|$ denotes the Morse index (see\cite{Fl2}). This implies that the grading of the Floer and the Morse complexes agree up to a shift.
Moreover, from \cite{Fl2} there exist a time dependent almost complex structure $\mathbf{J}$ for which the moduli spaces of Floer strips correspond to moduli spaces of connecting flow lines. 

On the other hand, in \cite[Theorem 13, Appendix B]{Sc} was shown that the orientations of the Morse complex obtained from the geometrical constructions can be extended to canonical orientations induced by the determinant bundle, the last ones are the ones used to define the Floer complex.

Then $CF(L, L'; \mathbb{Z}(\pi_{1}L), \mathbf{J}) = (\mathbb{Z}(\pi_{1}L)(Crit(f)), d) = CM(L; \mathbb{Z}(\pi_{1}L), (f,g_{\mathbf{J}}))$.

Here, $CM(L; \mathbb{Z}(\pi_{1}L))$ denotes the $\mathbb{Z}(\pi_{1}L)$-coefficients Morse complex. 
To see the last equality  $(\mathbb{Z}(\pi_{1}L)(Crit(f)), d) = CM(L; \mathbb{Z}(\pi_{1}L), (f,g_{\mathbf{J}}))$, notice that for $x \in (L \cap L')$, 
\begin{align*}
\sum_{ y, \mu(x, y) = 1 } \sum_{ \hat{u} \in \hat{\mathcal{M}}^{0}(x, y)} \text{sign}(u)[\Theta(u)]y =& \sum_{ y \in Crit(f)_{*-1}}( \sum_{ [\Theta(u)] \in\rho(\Theta(\mathcal{M}^{0}(x, y)))} n^{\tilde{f}, \tilde{g}_{\mathbf{J}}}(\tilde{x}, [\Theta(u)]\tilde{y}) [\Theta(u)]\big) y \\ =& \sum_{ y \in Crit(f)_{*-1}}( \sum_{ \alpha \in \pi_{1}L} n^{\tilde{f}, \tilde{g}_{\mathbf{J}}}(\tilde{x}, \alpha\tilde{y}) \alpha) y.
\end{align*}
Here $\rho: \Omega(L) \rightarrow \pi_{1}(L)$ denotes the projection. 
Now, from the invariance under Hamiltonian perturbations, and the independence on the choices defining the Floer homology, we can conclude that
$$H_{*}(CF(L, L'; \mathbb{Z}(\pi_{1}L), \mathbf{J}))\cong HF_{*}(L;\mathbb{Z}(\pi_{1}L)) \cong H_{*}(CM(L; \mathbb{Z}(\pi_{1}L)))\cong H_{*}(\tilde{L}).$$ 
\end{proof}
We now consider the exact Lagrangian with cylindrical ends associated to an exact Lagrangian cobordism $(W, L_{0}, L_{1})$. For such a Lagrangian the Floer homology depends on the choice of path component $[f]$, defining the perturbation on the  cylindrical ends. Different choices lead to different homologies. 

In order to relate the Floer homology $HF(W;\mathbb{Z}(\pi_{1}(L), [f])$ with the singular homology, let $S$ be as in Theorem \ref{Teorema}, an element of in the set $\{\emptyset, L_{0}, L_{1}, L_{0}\cup L_{1}\}$. Denote by $[f_{S}]$ the class of perturbation with $\alpha_{+} > 0$ if $L_{1} \subset S$ and $\alpha_{+} < 0$ if not. As well as $\alpha_{-} < 0$ if $L_{0} \subset S$ and $\alpha_{-} > 0$ if not. 

Let $\tilde{J} \in\mathcal{\widetilde{J}}_{B}$ be an autonomous almost complex structure that splits $\tilde{J} = J_{std} \oplus J$ (for some $J \in \mathcal{J}_{\omega}$) and denote by $g_{\tilde{J}}$ its induced Riemannian metric.
Let $\tilde{f}_{S}: W \rightarrow \mathbb{R}$ be a Morse function such that the negative gradient of $-\nabla_{g_{\tilde{J}}}\tilde{f}_{S}$ is transverse to $\partial W$ and it points outside of $W$ along $S$ and inside of $W$ on the complement of $S$ in $\partial W$.

\begin{proposition}\label{PropCob}
Let $(W, L_{0}, L_{1})$ be an orientable and spin exact Lagrangian cobordism equipped with a choice of spin structure. We have the following isomorphisms:$$HF(W, W;\mathbb{Z}(\pi_{1}W), [f_{S}]) \cong H(CM(W;\mathbb{Z}(\pi_{1}W),(\widetilde{f}_{S}, g_{\tilde{J}}))).$$
\end{proposition}
\begin{proof}
Extend the Morse function $\tilde{f}_{S}$ to a function $F_{S}:
 \overline{W}\rightarrow \mathbb{R}$ such that there is a compact set $B'\subset \mathbb{R}$ where $F_{S}\vert_{((\mathbb{R}^{2}\setminus
B')\times M) \cap \overline{W}}$ is a linear function.

Consider a Weinstein neighborhood $U_{\overline{W}}$ of $\overline{W}$, where we see $\overline{W}$ as the zero-section in $(T^{\*}\overline{W}, d\lambda_{std})$ with the standard exact symplectic form.

Let $\overline{W}' = \text{graph}(\epsilon dF_{S})$ be the image of $\overline{W}$ under the time one Hamiltonian diffeomorphism in the Hamiltonian isotopy $(\psi^{H}_{t})$, generated by the Hamiltonian $H =\epsilon F_{S}\circ\pi_{\overline{W}}$. Here, $\pi_{\overline{W}}: T^{*}\overline{W}\rightarrow \overline{W}$. 

The pair $(\overline{W}, \overline{W}')$is cylindrically distinct at infinity and intersects transversely, since the Hamiltonian flow $\psi^{H}$ is a translation outside some compact region.

From \cite{Fl2}, there exist a time dependent almost complex structure $\mathbf{\widetilde{J}}$, such that there is a bijection between the moduli space of Floer strips and the moduli space of flow lines connecting two critical points $x,y \in  \overline{W}\cap \overline{W}'$. 

Condition 2 on the Morse function $F_{S}$ ensures that $\mathbf{\widetilde{J}}\in \mathcal{\widetilde{J}}_{B}$. Since the Hamiltonian isotopy is linear outside the set $(\mathbb{R}^{2}\setminus
B)\times M$ then $(\psi^{H}_{t})_{*} = Id$ on this set. From the identity $\mathbf{\widetilde{J}}= (\psi^{H}_{t})_{*}\tilde{J}(\psi^{H}_{t})_{*}^{-1}$ follows the claim. 

Then, the Floer complex $CF(\overline{W}, \overline{W}'; \mathbb{Z}(\pi_{1}(W), \mathbf{\widetilde{J}})$ is well defined and 
\begin{align*}
H(CF(\overline{W}, \phi_{1}^{f_{S}\circ \pi}(\overline{W}); \mathbb{Z}(\pi_{1}(W),H, \mathbf{\widetilde{J}})) \cong & H(CF(\overline{W}, \overline{W}'; \mathbb{Z}(\pi_{1}(W), \mathbf{\widetilde{J}})) \\ \cong &
  H(CM(W;\mathbb{Z}(\pi_{1}W),(\widetilde{f}_{S}, g_{\tilde{J}}))).
 \end{align*}

\end{proof}
\begin{remark}
Notice that if $S =\emptyset$ or $L_{0}\cup L_{1}$, then the connected path $[f_{S}]$ corresponds to perturbations $f(x,y) = \alpha_{\pm}x + \beta_{\pm}$ with $\alpha_{-}$ and $\alpha_{+}$ of opposite signs.

Using the invariance of the $\mathbb{Z}(\pi_{1}W)$-Floer homology under horizontal Hamiltonian perturbations (Theorem \ref{Teorema} and Proposition 2.2), we have the following isomorphisms:
$$H_{*}(\tilde{L}_{0}) \cong HM_{*}(L_{0};\mathbb{Z}(\pi_{1}W)) \cong HF_{*}(W, W;\mathbb{Z}(\pi_{1}(W)),[f_{S}])\cong  HM_{*}(L_{1};\mathbb{Z}(\pi_{1}W))\cong H_{*}(\tilde{L}_{1}).$$

In the case $S = L_{0}, L_{1}$, the invariance of the $\mathbb{Z}(\pi_{1}W)$-Floer homology under horizontal Hamiltonian perturbations (Theorem \ref{Teorema} and the choice of Morse function on the cobordism imply:  
$$HF_{*}(W, W;\mathbb{Z}(\pi_{1}(W)),[f_{L_{i}}])\cong H_{*}(CM(W;\mathbb{Z}(\pi_{1}W),(\widetilde{f}_{L_{i}}, g_{\tilde{J}})))\cong H_{*}(\tilde{W}, \tilde{L}_{i}) = 0.$$
For the relation between singular and $\mathbb{Z}(\pi_{1}W)$-coefficients Morse homology see \cite{CoR}.
\end{remark}
\subsubsection{Proof of Corollaries \ref{Col1} and \ref{Col2}}
\begin{corollary*}\ref{Col1}
If $(W, L_{0}, L_{1})$ is an orientable and spin, exact Lagrangian cobordism and $i_{\#}:\pi_{1}(L_{i})\hookrightarrow \pi_{1}(W)$ is an isomorphism for $i= 0,1$, then $(W; L_{0}, L_{1})$ is an h-cobordism.
\end{corollary*}

\begin{corollary*}\ref{Col2}
If $(W; L_{0}, L_{1})$ is an orientable and spin, exact Lagrangian cobordism with $dim(W)> 5$ and $L_{i}$ ($i= 0,1$),  $W$ are simply connected, then $(W; L_{0}, L_{1})$ is a Lagrangian pseudo-isotopy.  
\end{corollary*}

To prove these Corollaries we first recall the $h$-cobordism theorem. 

\begin{theorem}[h-Cobordism (Smale)]
If $(W; L_{0}, L_{1})$ is a simply connected h-cobordism with dim$(W) > 5$, then the cobordism $(W; L_{0}, L_{1})$ is diffeomorphic to the trivial cobordism $([0,1] \times L_{0}; L_{0}, L_{0})$.
\end{theorem}

From the remark below, we have that if $(W; L_{0}, L_{1})$ is an exact Lagrangian cobordism, orientable and spin, equipped with a choice of spin structure then $H(\tilde{W}, \tilde{L}_{i}) = 0$ for $i= 0,1$.
Here, the homology $H(\widetilde{W}, \widetilde{L}_{i})$ is the relative homology induced from a CW-decomposition of the pair $(\widetilde{W},\widetilde{L}_{i})$.

The hypothesis on the fundamental groups $i_{\#}:\pi_{1}(L_{i})\hookrightarrow \pi_{1}(W)$ combined with $H(\tilde{W}, \tilde{L}_{i}) = 0$ for $i= 0,1$, imply that the inclusions are homotopy equivalences. This follows from one of the Whitehead Theorems \cite{Wh} Theorem 3. From this we deduce Corollary \ref{Col1}.

For the second corollary we use Corollary \ref{Col1} and the h-cobordism theorem.

\section{Simple homotopy type} \label{Chapter Simple h-theory}
This section presents the results on simple homotopy theory used in the paper. We follow the exposition of Cohen \cite{Co}.

The simple homotopy theory is the study of the relation of simple homotopy type between cellular complexes. 

\begin{definition}
Let $K$ and $L$ be finite CW-complexes. There is an elementary collapse from $K$ to $L$ (or an elementary expansion from $L$ to $K$) if $K = L \cup_{f\mid_{D^{n-1}}} D^{n-1}\cup_{f} D^{n}$, where:
\begin{enumerate}
\item $f: D^{n}\rightarrow K$ is the attaching map for $D^{n}$ and $f\mid_{D^{n-1}}: D^{n-1}\rightarrow K$ is the attaching map for $D^{n-1}$.
\item The new cells attached by $f$ are  not in $L$.
\item The closure of $\partial D^{n}- D^{n-1}$ denoted by $\overline{\partial D^{n}- D^{n-1}}$, satisfies $f(\overline{\partial D^{n}- D^{n-1}}) \subset L^{n-1}$.
\end{enumerate}
\end{definition}

Two CW-complexes $L$ and $K$ have the same \textit{simple-homotopy type} if they are related by a finite sequence of collapses and expansions, denoted $K\sim^{s} L$.

A map $f: L \rightarrow K$ is a \textit{simple-homotopy equivalence} if $f$ is homotopic to a map obtained by a finite number of compositions of the maps induced by collapses and expansions.

The relation of simple homotopy type is finer than that of homotopy type.
There are CW-complexes that are homotopy equivalent but not simple homotopy equivalent, examples can be found in \cite{Co}.

Given a fixed CW-complex $L$, consider the set of pairs $(K, L)$, where $K$ is a CW-complex homotopy equivalent to $L$, that contains $L$ as a sub-complex.
The Whitehead group of $L$, $Wh(L)$, is defined to be the set of equivalence classes of pairs $[K,L]$ under the relation $(K, L) \sim (K', L)$ if $K \sim^{s} K'$ rel $L$. This means that $K$ and $K'$ are related by a finite set of collapses and expansions for which no cell of $L$ is ever removed.
The group operation is $$[K, L] + [K', L] = [K\cup_{L} K', L].$$
Where $K\cup_{L} K'$ is the disjoint union of $K$ and $K'$ identified by the identity map on $L$.
The Whitehead group of $L$, $Wh(L)$, is an abelian group (see 6.1 of \cite{Co}).

Given a homotopy equivalence $f:K \rightarrow L$, we define its Whitehead torsion $\tau(f) := [M_{f}\cup_{K}M_{f}, L]$, where $M_{f}$ denotes the mapping cylinder of $f$.

The Whitehead torsion is the obstruction for a homotopy equivalence to be a simple-homotopy equivalence as stated in the following theorem.

\begin{theorem}[22.2 in \cite{Co}]
A homotopy equivalence $f:K \rightarrow L$, is a simple homotopy equivalence if and only if $\tau(f) = 0$.
\end{theorem}

The geometric definitions of Whitehead group is hard to work with, instead there is an algebraic version that we introduce in the following subsection.

\subsection{The Whitehead torsion} 

\subsubsection{The Whitehead torsion of a ring}

Let $\mathbb{Z}(G)$ be the integral group ring of the group $G$. The group $GL(n,\mathbb{Z}(G))$ denotes the group of non-singular $n\times n$ matrices over $\mathbb{Z}(G)$. Identifying each $M\in GL(n, \mathbb{Z}(G))$  with the matrix
\[\left[ {\begin{array}{cc}
   M & 0 \\
   0 & 1 \\
  \end{array} } \right],
\]
we obtain an injection $GL(n, \mathbb{Z}(G)) \subset GL(n+1, \mathbb{Z}(G))$. The infinite general linear group of $\mathbb{Z}(G)$ is the union over all $n$ of the groups $GL(n, \mathbb{Z}(G))$ and is denoted by $GL(\mathbb{Z}(G))$.
A matrix in $GL(\mathbb{Z}(G))$ is called elementary if it coincides with the identity matrix except for one off-diagonal element.

The set of elementary matrices is the commutator subgroup of $GL(\mathbb{Z}(G))$.

Let $E(G)$ denote the group generated by all the elementary matrices and all matrices that coincide with the identity matrix except for one diagonal element $g \in \pm G$. 

The Whitehead group of $G$ is the group:
$$Wh(G) = GL(\mathbb{Z}(G))/E(G).$$
The class of a matrix $[A] \in Wh(G)$ is called the Whitehead torsion of the matrix and is denoted by $\tau(A)$. 

\subsubsection{The Whitehead torsion of an acyclic $\mathbb{Z}(G)$-complex}

\begin{definition}
A $\mathbb{Z}(G)$-module is a free $\mathbb{Z}(G)$-module $M$ with a distinguished family of bases $B$ which satisfy:
If $b$ and $b'$ are two bases of $M$ and if $b \in B$, then $b' \in B$ if and only if $\tau([b/b']) = 0 \in Wh(G)$. Here $[b/b']$ represents the matrix that changes the base $b$ to the base $b'$.
\end{definition}
An isomorphism $f:M_{1} \rightarrow M_{2}$ is a \textit{simple isomorphism of $\mathbb{Z}(G)$-modules} if $\tau(b_{f}) = 0$, where $b_{f}$ is the matrix of $f$ with respect to any distinguished basis of $M_{1}$.

\begin{definition}
A $\mathbb{Z}(G)$-complex is a free chain complex over $\mathbb{Z}(G)$, $C=(C,d)$ such that each $C_{*}$ is a $\mathbb{Z}(G)$-module. A preferred basis of $C$ mean a basis $c = \cup_{i} c_{i}$ where $c_{i}$ is a preferred basis of $C_{i}$.
\end{definition}
A \textit{simple isomorphism} of $\mathbb{Z}(G)$-complexes, $f: C\rightarrow C'$, is a chain mapping such that $f\mid_{C_{*}}: C_{*}\rightarrow C_{*}'$ is a simple isomorphisms for all $*$.

Another useful notion is the following. Two $\mathbb{Z}(G)$-complexes, $(C, d)$ and $(C', d')$, are \textit{simple homotopy equivalent}, $(C, d)\sim^{s}(C', d')$, if there exist trivial\footnote{A $\mathbb{Z}(G)$-complex $(T,t): 0 \rightarrow T_{i+1} \rightarrow T_{i}\rightarrow 0$ is elementary trivial if $\tau(T) =0$. A trivial chain complex is a direct sum of elementary trivial chain complexes.} $\mathbb{Z}(G)$-chain complexes $(T, t)$ and $(T', t')$ such that the chain complex $(C\oplus T,d\oplus t)$ is simple isomorphic to the chain complex $(C'\oplus         T',d'\oplus t')$. 

Let $(C, d)$ be an acyclic $\mathbb{Z}(G)$-complex. Then $(C, d)$ is contractible, so there exist a chain contraction $\delta:C\rightarrow C[1]$ such that $\delta d + d\delta = 1$. If $C_{odd} = C_{1}\oplus C_{3}\oplus \cdots$ and $C_{even} = C_{0}\oplus C_{2}\oplus \cdots$, then the torsion of the complex $(C, d)$ is defined to be the torsion of the map:
$$(d+ \delta)_{odd}=(d+ \delta)\mid_{C_{odd}}: C_{odd} \rightarrow C_{even},$$
$$\tau(C) = \tau((d+ \delta)_{odd}) \in Wh(G).$$

The Whitehead torsion of a complex satisfies the following properties:
\begin{enumerate}
\item If $f:C\rightarrow C'$ is a simple isomorphism of chain $\mathbb{Z}(G)$-complexes, then $\tau(C) = \tau(C')$.
\item The torsion of a direct sum of $\mathbb{Z}(G)$-complexes $C\oplus C'$ satisfies $\tau(C\oplus C') = \tau(C)+ \tau(C')$.
\end{enumerate}
If $f:C\rightarrow C'$ is a homotopy equivalence of $\mathbb{Z}(G)$-complexes, then the Whitehead torsion of $f$ is defined by $\tau(f)= \tau(Cone(f))$.

\subsubsection{The Whitehead torsion of CW-pair}

Let $(W, L)$ be a pair of finite, connected CW-complex where $L$ is a sub-complex of $W$.
Let $C(W,L)$ denote the cellular complex, where $C(W,L)_{*}= H_{*}(W^{*}\cup L, W^{*-1}\cup L) $ is the free module generated by the $*$-cells of $W-L$ and the differential is defined by the boundary operator on the exact sequence for singular homology of the triple $(W^{*}\cup L, W^{*-1}\cup L, W^{*-2}\cup L)$.

If $L\hookrightarrow W$ is a homotopy equivalence, then the induced map $\pi_{1}(L)\rightarrow \pi_{1}(W)$ is an isomorphism. In this case, if $p:\widetilde{W} \rightarrow W$ is the universal covering of $W$, then $p^{-1}(L) = \widetilde{L}$ is the universal covering of $L$ and $\widetilde{L} \hookrightarrow \widetilde{W}$ is also a homotopy equivalence.

The projection map $p$ induces a structure of CW-complex on $\widetilde{W}, \widetilde{L}$ from the ones on $W, L$. With this cellular structure, the cellular complex $C(\widetilde{W}, \widetilde{L})$ is a $\mathbb{Z}(\pi_{1}L)$-complex.
\begin{definition}
Let $(W, L)$ be a pair of finite, connected CW-complexes, and suppose that $L\hookrightarrow W$ is a homotopy equivalence. The Whitehead torsion of the pair is defined by $$\tau(W, L)= \tau(C(\widetilde{W}, \widetilde{L})) \in Wh(\pi_{1}L),$$
where $(\widetilde{W}, \widetilde{L})$ is the universal covering of $(W,L)$. 
\end{definition}
 
\subsubsection{ The s-cobordism theorem} 
One application of the theory, briefly recalled in this section, is the s-cobordism theorem.

\begin{theorem}[s-Cobordism, Mazur \cite{Ma}, Barden \cite{Ba}, Stallings \cite{St}]
Let $(W; L_{0}, L_{1})$ be an h-cobordism of dimension
$n > 5$. Then the torsion $\tau(W, L_{0})$ vanishes if and only if $W$ is diffeomorphic to the product $L_{0}\times [0, 1]$.
\end{theorem}

Notice that if $(W; L_{0}, L_{1})$ is an h-cobordism, then we can also define the Whitehead torsion of a pair $(W,L_{0})$ using a Morse function. Recall that any nice\footnote{A self-indexing function without critical points on the boundary.} Morse function defined on an h-cobordism $(W; L_{0}, L_{1})$ induces a CW-structure on it, that lifts to the universal covering $\widetilde{W}$. The following theorem of Milnor establishes that the torsion of the complex defined by the CW-pair induced by the lift of any nice Morse function is the torsion of the pair $(W, L_{0})$.

\begin{theorem}[Theorem 9.3 \cite{Mi}] 
If $(W;L,L')$ is an h-cobordism, then $$\tau(W, L) = \tau(C(\widetilde{W}, \widetilde{L};f)$$ does not depend on the choice of nice Morse function. 
\end{theorem}

\begin{remark}\label{remarckWh}
Recall that under some conditions established in Corollary \ref{Col1}, if $(W; L_{0}, L_{1})$ is an exact Lagrangian cobordism, then it is an h-cobordism and the torsion of the pair $(W, L_{0})$ is well defined.
From Milnors's theorem follows that $ \tau(W, L_{0}) = \tau(CM(W; \mathbb{Z}(\pi_{1}W)(f_{L_{0}},g_{J}))$, for a pair $(f_{L_{0}},g_{J})$ satisfying the conditions required in the proof of Proposition 2.4.
\end{remark}

\section{Whitehead torsion and Floer complex}
\label{chapitre whitehead}

By the s-cobordism Theorem, an h-cobordism of dimension higher than 5 is a pseudo-isotopy if its Whitehead torsion vanishes. In the symplectic category, the cobordism studied has an additional structure: it is a Lagrangian submanifold with Lagrangian boundary. 

To compute the Whitehead torsion of an exact Lagrangian h-cobordism $(W; L_{0}, L_{1})$, we study the torsion of the $\mathbb{Z}(\pi_{1}W)$-Floer complex associated to the pair of exact Lagrangians with cylindrical ends $(W, W')$, where $W' = \phi^{H}_{1}(\phi_{1}^{f\circ \pi}(W))$ satisfies $W \pitchfork W'$.

In this section we show that the Whitehead torsion of the $\mathbb{Z}(\pi_{1}W)$-Floer complex is invariant under horizontal perturbations of the Lagrangian with cylindrical ends associated to the cobordism. 
 
Fukaya defined a Whitehead torsion for the Floer complex of Lagrangian intersections and he suggested the invariance of it under Hamiltonian perturbations in \cite{Fu}. He proposed the following conjecture.
\begin{conjectures}[Fukaya, 1995] If $HF(L_{1}, L_{2}) = 0$ and $\tau(CF(L_{1}, L_{2})) = 0$, then there exists an exact symplectic diffeomorphism $\phi$ such that $L_{1}\cap \phi(L_{2}) = \emptyset$.
\end{conjectures}

M. Sullivan \cite{Su} proved the invariance of the Whitehead torsion of the $\mathbb{Z}_{2}(\pi_{1}L)$-coefficients Floer complex under compactly supported Hamiltonian isotopies.
He considered the $\mathbb{Z}_{2}(\pi_{1}L)$-coefficients Floer complex associated to a pair of Hamiltonian isotopic, non-compact Lagrangians $(L, L')$, such that $\omega\vert_{\pi_{2}(M,L)}=0$. Part of his proof is based on a Gluing theorem proved by Lee.

Lee studied the Floer-Novikov complex for periodic orbits of a Hamiltonian. In \cite{Le1} and \cite{Le2}, she defined an invariant, denoted by $I_{F}$, related to the Reidemeister torsion of the Floer-Novikov complex. Based on a series of gluing theorems, she proved that $I_{F}$ is invariant under Hamiltonian perturbation. 
Her work is a generalization to the symplectic setting of the same invariant $I_{F}$, studied by Hutchings and her for the Novikov complex in \cite{HuL}.

From the theory exposed in previous section, any finite acyclic $\mathbb{Z}(\pi_{1}W)$-complex with a preferred family of bases has a well defined Whitehead torsion. The $\mathbb{Z}(\pi_{1}W)$-Floer complex associated to the pair $(W, W')$ has a preferred base given by the set $W\cap W'$, and is an acyclic complex for the class of perturbations that displace $W'$ away from $W$; therefore it has a well-defined Whitehead torsion.

From remark (\ref{remarckWh}), the Whitehead torsion of an exact Lagrangian h-cobordism, can be defined using the $\mathbb{Z}(\pi_{1}W)$-Morse complex. The $\mathbb{Z}(\pi_{1}W)$-Morse complex agrees with the $\mathbb{Z}(\pi_{1}W)$-Floer complex for a special choice of Hamiltonian. 

The invariance of the Whitehead torsion of the $\mathbb{Z}(\pi_{1}W)$-Floer complex under horizontal perturbations, implies that the Whitehead torsion of an exact Lagrangian h-cobordism is given by the Whitehead torsion of its $\mathbb{Z}(\pi_{1}W)$-Floer complex. As an application of this, we establish that an exact Lagrangian cobordism (satisfying appropriate hypotheses), is a Lagrangian pseudo-isotopy. 

\subsection{Main results of the section}

The main results of the section are the following:
\begin{theorem}
\label{theorem invariance}
Let $(W, W')$ be a pair of exact Lagrangians with cylindrical ends where $W' = \phi^{H}_{1}(\phi_{1}^{f_{S}\circ \pi}(W))$ for  $S = L_{0}, L_{1}$ and $W \pitchfork W'$. 

Consider the $\mathbb{Z}(\pi_{1}W)$-Floer complex $CF(W, W';\mathbb{Z}(\pi_{1}W),\mathbf{J}_{0})$, where $\mathbf{J}_{0} \in \mathcal{\widetilde{J}}_{B}$. If $\{\phi_{t}\}_{t\in [0, 1]}$ is a horizontal isotopy, then the Whitehead torsion of the Floer complex, satisfies: $$\tau(CF(W, W';\mathbb{Z}(\pi_{1}W),\mathbf{J}_{0})) = \tau(CF(W, \phi_{1}(W');\mathbb{Z}(\pi_{1}W),\mathbf{J}_{1})).$$
\end{theorem}

And the main result of the paper:
\begin{theorem}
\label{s-theorem}
Let $(W; L_{0}, L_{1})$ be an exact, orientable and spin Lagrangian cobordism equipped with a choice of spin structure. Assume $dim(W)>5$.

If the map $(j_{i})_{\#}:\pi_{1}(L_{i})\rightarrow \pi_{1}(W)$ induced by the inclusion $L_{i} \hookrightarrow W$ is an isomorphism for $i =0,1$, then $(W; L_{0}, L_{1})$ is a Lagrangian pseudo-isotopy.
\end{theorem}

\subsection{Bifurcation analysis}

The technique of bifurcation analysis is used for proving the invariance of Floer/Morse complexes under some parameter. It is also used to prove the independence under Hamiltonian perturbation of torsion-type invariants associated to Floer/Morse complexes. 

In this subsection we explain the bifurcation analysis technique. The idea of this technique is to consider a generic homotopy joining the two generic parameters defining two Floer complexes. At each time of the homotopy the parameters associated have some regularity. At certain times along the homotopy regularity is lost. The technique consists in describing how the complex changes after each one of these moments called bifurcations.

Floer's first proof of the independence of the Floer homology under Hamiltonian perturbations, used bifurcation analysis \cite{Fl}. Through a series of gluing theorems, he describes the changes of the Floer complex after a bifurcation occurs. The complete proofs of these gluing theorems, for the Floer complex, appeared in Lee \cite{Le1}, \cite{Le2}.  

Basic bifurcation theory appears in the Morse theory. Given two Morse-Smale pairs $(f_{0},X_{0})$ and $(f_{1},X_{1})$ ( $f_{i}$ is a Morse function and $X_{i}$ a Morse-Smale pseudo-gradient vector field for $i=0, 1$), consider a homotopy joining these parameters $\{(f_{\lambda},X_{\lambda})\}_{\Lambda}$. From Cerf's work \cite{Ce}, it is known that in a generic homotopy, for a finite number of parameters, the pair $(f_{\lambda},X_{\lambda})$ is not Morse-Smale. The two phenomena that generically occur at isolated times are:
\begin{enumerate}
\item There is a degenerated critical point, that has only one degenerate direction with a quadratic tangency, where two non-degenerated critical points are born or die. This phenomenon is called Birth-Death.
\item There are two non-degenerated critical points of same index for which the unstable and stable manifolds do not intersect transversely, yielding a degenerate flow line. This phenomenon is called Handle-Slide.
\end{enumerate}

For the Floer complex of Lagrangian intersections, when the Lagrangians are exact, the bifurcations occurring along a homotopy joining two generic set of parameters are analog to the ones occurring to the Morse complex.
\begin{definition}
Fix a pair of exact Lagrangians with cylindrical ends $(W,W')$.
The data $\mathcal{D}=(\psi^{H},\mathbf{J})$, where $\psi^{H}\in Ham(\tilde{M})$ and $\mathbf{J}$ is a time-dependent, compatible almost complex structure is called generic if the associated Floer complex  \\$CF(W, \psi^{H} (W');\mathbb{Z}(\pi_{1}W),\mathbf{J})$ is well defined, namely:
\begin{enumerate}
\item The Lagrangians intersect transversely, $W \pitchfork \psi^{H}(W')$.
\item For every $\mathbf{J}$-holomorphic strip $u$, the linearization of the time dependent Cauchy-Riemann operator $\bar{\partial}_{\mathbf{J}}$ is surjective.
\end{enumerate}
\end{definition}

Let $\Lambda=[0, 1]$ and consider a homotopy $\mathcal{D}_{\Lambda}=\{\mathcal{D}_{\lambda} =(\psi_{\lambda}, \mathbf{J}_{\lambda})\}_{\lambda \in \Lambda}$ joining two generic data $\mathcal{D}_{0}=(\psi^{H_{0}}, \mathbf{J}_{0})$ and $\mathcal{D}_{1}=(\psi^{H_{1}},\mathbf{J}_{1})$. In general, the Floer complex is not well defined at all times of the homotopy.
\begin{definition}
A bifurcation along the homotopy $\mathcal{D}_{\Lambda}$ is a time $\lambda_{0}\in \Lambda$ such that $\mathcal{D}_{\lambda_{0}}$ is not generic.
\end{definition}
To study the changes of the Whitehead torsion of the $\mathbb{Z}(\pi_{1}W)$-Floer complex, we will pick a nice homotopy which will allow as to compare the two complexes and show that they are simply homotopy equivalent. The homotopies that we will consider are defined as follows:
\begin{definition} 
A generic homotopy $\{\mathcal{D}_{\lambda}\}_{\lambda \in \Lambda}$ joining generic data $\mathcal{D}_{0}$ and $\mathcal{D}_{1}$, is a smooth homotopy such that the only bifurcations occur at isolated times and are of the following type:
\begin{enumerate}
\item Births-deaths: At some time $\lambda = \lambda_{0}$ in the homotopy there is a non-transverse intersection point in $W \cap \psi_{\lambda_{0}}^{H}(W)$, such that it has only one degenerate direction with a quadratic tangency, where two non-degenerated intersection points appear or disappear.
\item Handle slides: A degenerate strip $u$ of index $\mu(u)= 0$ between non-degenerate intersection points.
\end{enumerate}
\end{definition}
The following proposition assure us that generic homotopies exist. Let $\Phi_{\Lambda}$ denote the set of horizontal isotopies $\phi_{\lambda}$ joining $id$ with $\psi^{H}$ and denote by  $\mathcal{J}_{\Lambda}$ the space of one-parameter families of time dependent almost complex structures $\mathbf{J}_{\lambda}$ joining $\mathbf{J}_{0}$ with $\mathbf{J}_{1}$, where $\mathbf{J}_{\lambda} \in \mathcal{\widetilde{J}}_{B}$ for each $\lambda$.
\begin{proposition}
Given two generic data $\mathcal{D}_{0}$ and $\mathcal{D}_{1}$, there is a non-empty set $\mathcal{H}\subset \Phi_{\Lambda}\times \mathcal{J}_{\Lambda}$ of regular homotopies $\{\mathcal{D}_{\lambda}\}_{\lambda \in \Lambda}$ joining them.
\end{proposition}
The proof appears in \cite[Lemma 3.3]{Fl} and \cite[Proposition 3.2]{Fl}. Floer's proof still applies to this non-compact setting, since the main argument is based on the structure and regularity of the spaces of functions $\Phi_{\Lambda}$ and of almost-complex structures $\mathcal{J}_{\Lambda}$, and the regularity of these spaces is independent of the compactness of the setting. 

Consider a regular homotopy $\{\mathcal{D}_{\lambda}\}_{\lambda \in \Lambda}$ joining generic data $\mathcal{D}_{0}$ and $\mathcal{D}_{1}$. Let $\Lambda_{\text{bif}}$ be the set of parameters where a bifurcation occurs. A regular homotopy defines a one-parameter family of $\mathbb{Z}(\pi_{1}W)$-Floer complexes, $\{CF(W, \psi_{\lambda}(W');\mathbf{J}_{\lambda})\}_{\lambda\in\Lambda -\Lambda_{\text{bif}}}$.
The proof of Theorem \ref{theorem invariance} consists in the observation that through a regular homotopy, the bifurcations that the $\mathbb{Z}(\pi_{1}W)$-Floer complexes suffer are elementary (change of basis or sum with a trivial complex). Then, the $\mathbb{Z}(\pi_{1}W)$-Floer complex associated to $\mathcal{D}_{0}$ is simply homotopy equivalent to the $\mathbb{Z}(\pi_{1}W)$-Floer complex associated to $\mathcal{D}_{1}$. This implies that the Whitehead torsion of the $\mathbb{Z}(\pi_{1}W)$-Floer complex is invariant under Hamiltonian perturbations.

In the following subsections we describe the behavior of the $\mathbb{Z}(\pi_{1}W)$-Floer complex after the two types of bifurcations, Birth-Deaths and Handle-Slides. To do this, let us assume that $\{\mathcal{D}_{\lambda}=(\psi_{\lambda}, \mathbf{J})\}_{\lambda \in \Lambda}$ is a regular homotopy joining two generic sets of data $\mathcal{D}_{0}=(id, \mathbf{J}_{0})$ with $\mathcal{D}_{1}=(\psi^{H}, \mathbf{J}_{1})$. We will assume that exactly one of the two phenomena at time $\lambda =\lambda_{0}$ occurs, and we will describe the changes in $CF(W, \psi_{\lambda_{0}+\epsilon}(W');\mathbf{J}_{\lambda_{0}+\epsilon})$ with respect to $CF(W, \psi_{\lambda_{0}-\epsilon}(W');\mathbf{J}_{\lambda_{0}-\epsilon})$.

\subsubsection{Birth-death bifurcation analysis}

In this subsection we study the birth-type bifurcation following \cite{Su}; the analysis for a death type bifurcation is completely analogous.

Assume that there is a birth at $\lambda = \lambda_{0}$. Thus, there is a non-transverse intersection point $x_{0} \in W\cap \psi_{\lambda_{0}}(W')$, which gives birth to two families of non-degenerated intersection points denoted by $\{x^{-}_{\lambda}\}_{\lambda \in\Lambda}$, and $\{x^{+}_{\lambda} \}_{\lambda \in\Lambda}$ with $ x^{-}_{\lambda}, x^{+}_{\lambda}  \in W\cap \psi_{\lambda}(W')$ where $\lambda > \lambda_{0}$.

In this situation the $\mathbb{Z}(\pi_{1}W)$-Floer complex changes in the following way:
\begin{enumerate}
\item There are two new generators for the module $\mathbb{Z}(\pi_{1}W)(W \cap  W_{\lambda})$, namely $$\mathbb{Z}(\pi_{1}W)(W \cap  W_{\lambda_{0}+\epsilon}) = \mathbb{Z}(\pi_{1}W)(W \cap  W_{\lambda_{0}-\epsilon})\oplus \mathbb{Z}(\pi_{1}W)\langle x^{-}_{\lambda_{0}+\epsilon}, x^{+}_{\lambda_{0}+\epsilon}\rangle.$$ There is a new moduli space $\mathcal{M}^{1}( x^{-}_{\lambda}, x^{+}_{\lambda})$, composed by strips connecting the two new generators.
\item New strips connecting old intersection points to the new ones can appear. The spaces $\mathcal{M}^{1}(x, y)$ and $\mathcal{M}^{1}(y, x)$ where $x = x^{-}_{\lambda_{0}+\epsilon}, x^{+}_{\lambda_{0}+\epsilon}$ and $y \neq x^{-}_{\lambda_{0}+\epsilon}, x^{+}_{\lambda_{0}+\epsilon}$ could be non-empty. This will affect the differential of the $\lambda_{0}+\epsilon$ complex.
\end{enumerate}
\begin{definition}
We say that a birth or a death is independent, if the last two spaces are empty.
\end{definition}
An independent birth/death does not change the simple homotopy type of the complex, and thus the Whitehead torsion. See also Lemma 4.5.

A stabilization technique of M. Sullivan \cite{Su}, reduces the general case of a birth/death to that of independent birth/deaths as we explain below.

The crucial point here is the existence of a regular homotopy connecting the data $\mathcal{D}_{\lambda_{0}-\epsilon}$ with $\mathcal{D}_{\lambda_{0}+\epsilon}$, such that after applying the stabilization technique twice the corresponding spaces in (2) are empty. Since the stabilization operation does not change the simple homotopy type of a complex, the $\mathbb{Z}(\pi_{1}W)$-Floer complex before the birth-death is therefore simply homotopy equivalent to the $\mathbb{Z}(\pi_{1}W)$-Floer complex after the birth-death.

Due to the assumption that the birth is isolated, there exist a Darboux chart $(U,\sigma)$, where $U \subset
(\tilde{M},\tilde{\omega})$ is a Darboux neighborhood centered at $x_{\lambda_{0}}$ such that:
\begin{enumerate}
\item The map $\sigma$ is a symplectomorphism $\sigma:(U,\tilde{\omega}) \rightarrow (I^{2n}, \omega_{std})$, where $I = [-3, 3].$
\item The image of $W$ under $\sigma$ is modeled by the zero section; $\sigma(U\cap W) = I^{n} \times \{0\}.$
\item For a horizontal Hamiltonian isotopy $\{\phi_{\lambda}\}_{\lambda\in \Lambda}$, the family of Lagrangian $\{\phi_{\lambda}(W')\}_{\lambda\in \Lambda}$, is modeled by the graph of the differential of the family of smooth functions $f_{\lambda}: I^{n}\rightarrow \mathbb{R},$ defined by $f_{\lambda}(q_{1},...,q_{n}) = \dfrac{1}{9} q_{1}^{3} - (\lambda - \lambda_{0}) q_{1}  + Q(q_{2},...,q_{n}),$ where $Q$ is some non-degenerated quadratic function and $\lambda_{0}\in (0,1).$ Then the image of the family under $\sigma$ is given by the set $$\sigma(U\cap W_{\lambda}) = \{(q_{1},...,q_{n}, \dfrac{\partial f_{\lambda}}{\partial q_{1}},...,\dfrac{\partial f_{\lambda}}{\partial q_{n}}) | (q_{1},...,q_{n}) \in I^{n}\}.$$
\end{enumerate}
\begin{remark}
The critical points of the family $f_{\lambda}(q_{1},...,q_{n})$ are in correspondence with the intersection points of the family $\{W \cap \phi_{\lambda}(W')\}_{\lambda\in \Lambda}$. The family of functions $f_{\lambda}$ has a unique degenerated critical point at $\lambda_{0}$; since the equation $\dfrac{\partial f_{\lambda}}{\partial q_{1}}= \dfrac{1}{3}q_{1}^{2}-(\lambda - \lambda_{0}) = 0$ has no solution for $\lambda < \lambda_{0}$ and two non-degenerated solutions for $\lambda > \lambda_{0}$. 
\end{remark}

With the previous description we can study the birth-death bifurcation in the Lagrangian intersection setting using the classical birth-death bifurcation in a family of Morse functions.
 
\subsubsection{Stabilization technique applied to a Birth bifurcation} 

\begin{definition}
The quadratic stabilization of a smooth function $f:X\rightarrow \mathbb{R}$ is the function $f + Q:X\times \mathbb{R}^{N}\rightarrow \mathbb{R}$, where $Q$ is a non-singular quadratic function on ${R}^{N}$.
\end{definition}
\begin{remark}
Eliashberg and Gromov \cite{ElGr} used stable Morse theory to study Lagrangian intersection theory with finite dimensional methods. They remarked that the stabilization process does not change the simple homotopy type of the $\mathbb{Z}(\pi_{1}X)$-Morse complex.
In the present context we will set $N = 1$, so we will study the family of functions $f_{\lambda}\pm x^{2}$.
\end{remark}
To understand the behavior of the moduli spaces when new intersection points appear we study the stabilized setting defined in analogy with the Morse case.
\begin{definition}
The stabilization of the symplectic manifold $(\tilde{M}, \tilde{\omega})$ is defined by
 $$(S(\tilde{M}), S(\tilde{\omega})):= (\tilde{M}\oplus \mathbb{R}^{2}, \tilde{\omega} \oplus  \omega_{\text{st}}).$$
The stabilization of a pair of fixed Lagrangians $(W, W')$, with $W, W'\subset \tilde{M}$ is defined by $S(W) := W\times \{(x,0) \vert x\in \mathbb{R}\}$ and the second Lagrangian can be stabilized either in a positive or in a negative way depending on the choice of sign, $S(W') := W\times \{(x,\pm 2x) \vert x\in \mathbb{R}\}.$ When required we will denote $S_{-}(W')$ or $S_{+}(W')$ to mean a positive or a negative stabilization.
\end{definition}

In the new setting, after a negative stabilization, the Darboux chart $(U,\sigma)$ is replaced by $(U\times \mathbb{R}^{2},\sigma \times id_{\mathbb{R}^{2}})$. The family of Lagrangians $\{\phi_{\lambda}(W')\}_{\lambda\in \Lambda}$ is now described locally by the stabilized function $F_{\lambda}(q_{1},...,q_{n}, x) = f_{\lambda}(q_{1},...,q_{n}) - x^{2}$.
Thanks to the new dimension we can find a one-parameter family of smooth functions connecting $F_{0}$ to $F_{1}$, close enough to the family $F_{\lambda}$ but with a birth having a non-zero component on the fiber of the stabilization.
To do this we need to have control on the growth of the derivative in the degenerate direction, $\frac{\partial F_{\lambda}}{\partial \lambda}$. 

We will now change the family of functions $\{f_{\lambda}\}_{\lambda\in \Lambda}$ which describe locally the birth by the family of functions used by M. Sullivan.

Let $\alpha_{0}:\mathbb{R}\rightarrow\mathbb{R}$ be a smooth bump function with the following properties:
\begin{enumerate}
\item $\alpha_{0}$ is an even function supported on $[-1, 1]$ and non-vanishing on $(-1, 1)$,
\item There is a unique maximum of 1 at 0,
\item $\alpha_{0}'$ has a unique minimum of -2 at $\dfrac{1}{2}$,
\item $\alpha_{0}'(x) = \alpha_{0}'(1 - x)$ for $x \in (0, \dfrac{1}{2})$.  
\end{enumerate}
Let be $\alpha(x) = \alpha_{0}(x + \dfrac{1}{2})$.

Then there exist a Darboux chart, still denoted $(U, \sigma)$, centered at $x_{\lambda_{0}}$, such that the family $\{U \cap \phi_{\lambda}(W')\}_{\lambda\in \Lambda}$ is modeled by the graph of the differential of the family of functions:
\begin{equation}
f_{\lambda}(q_{1},...,q_{n}) = \epsilon q_{1} + \lambda\alpha(q_{1}) + Q(q_{1},...,q_{n}),
\end{equation}
where $0<\epsilon <1$ is arbitrarily small. The critical points of $f_{\lambda}$ correspond to intersection points of $W \cap \phi_{\lambda}(W')$. 
Notice that if $\dfrac{\partial f_{\lambda}}{\partial q_{1}} = \epsilon +  \lambda\alpha'(q_{1})$, then $\dfrac{\partial f_{\lambda}}{\partial q_{1}} = 0$ when $\epsilon +\lambda\alpha'(q_{1}) = 0$. 
\begin{enumerate}
\item If $0 \leq\lambda < \dfrac{\epsilon}{2}$, since $\vert \alpha'(q_{1})\vert \leq 2$, then $\lambda\vert \alpha'(q_{1})\vert < \epsilon $, so there are no critical points for any parameter $\lambda$ in this region.
\item If $\lambda = \dfrac{\epsilon}{2}$, then since $\alpha'(q_{1})$ has a unique minimum of $-2$ at $q_{1} = 0$ then there is a degenerate critical point.
\item If $\lambda > \dfrac{\epsilon}{2}$, then the equation $\epsilon +\lambda\alpha'(q_{1}) = 0$ has two solutions for each fixed $\lambda$, so there are two non-degenerate critical points with positive and negative $q_{1}$-coordinates.
\end{enumerate} 
 
The stabilized family of functions $F_{\lambda}(q_{1},...,q_{n}, x) = f_{\lambda}(q_{1},...,q_{n}) - x^{2}$ can be perturbed as the following theorem asserts:
\begin{theorem}\cite{Su}
\label{sullivan}
There exists a one-parameter family of functions $G_{\lambda}:I^{n}\times \mathbb{R}^{2}\rightarrow \mathbb{R}$ such that for each $c\in \mathbb{R}$:
\begin{enumerate}
\item $G_{0} = F_{0}$ and $G_{0} = F_{0}$,
\item $G_{\lambda}(q_{1},...,q_{n},x) = F_{0}(q_{1},...,q_{n},x) = F_{1}(q_{1},...,q_{n},x)$ near the boundary of $I^{2n}$,
\item For each $\lambda$, the graph of $\nabla G_{\lambda}$ stays inside $I^{2n}$.
\item No deaths of critical points occur and a unique birth occurs arbitrarily close to $(\lambda, \overrightarrow{q}, x) = (\epsilon, \overrightarrow{0}, c)$.
\item At the moment of the birth $supp(G_{\lambda}- F_{0}) \subset (-1,1)^{n}\times (c-1, c+1)$.
\end{enumerate}
\end{theorem}
With the help of the previous theorem and the open mapping theorem we prove the following Lemma:
\begin{lemma}\cite[Theorems 3.7, 3.12]{Su}
Let $(W,W')$ be a pair of exact Lagrangians with cylindrical ends. Consider two generic data $(id, \mathbf{J}_{0})$ and $(\phi_{1}^{H}, \mathbf{J}_{1})$ defining Floer complexes associated to the pair. Then there exist a regular homotopy $\{(\psi_{\Lambda}, \mathbf{J}_{\lambda})\}_{\lambda \in \Lambda}$ between the generic data such that, if  a birth at time $\lambda = \lambda_{0}$ occurs, then after stabilizing twice the birth is independent.
\end{lemma}
\begin{proof}
Consider the family of Lagrangian $\tilde{W}_{\lambda}$ defined locally by the functions $G_{\lambda}$ from Theorem \ref{sullivan}. This family of Lagrangians has a unique birth arbitrarily close to $(\lambda_{0}, \overrightarrow{q}, x_{1}) = (\dfrac{\epsilon}{2}, 0, C)$. Stabilize in a positive way the new setting. A second application of Theorem \ref{sullivan} to the new family of functions $G_{\lambda}$,  implies that there is another family of Lagrangians $\tilde{\tilde{W}}_{\lambda}$, defined locally by a family of functions $\tilde{G}_{\lambda}$, with a unique birth arbitrarily close to $(\lambda_{0}, \overrightarrow{q}, x_{1}, x_{2}) = (\dfrac{\epsilon}{2}, 0, C, C)$. We will show that, for any $x \in S^{2}(W)\cap \tilde{\tilde{W}}_{\lambda_{0}+ \epsilon}$ the moduli spaces of $\mathbf{J}_{\lambda_{0}+\epsilon}\oplus J_{\text{std}}\oplus J_{\text{std}}$-holomorphic strips with only one of the new intersection points involved $x^{\pm}_{\lambda_{0} + \epsilon}$, are empty:
\begin{equation}
\mathcal{M}^{1}_{\lambda_{0} + \epsilon}(x, x^{\pm}_{\lambda_{0} + \epsilon}) = \mathcal{M}^{1}_{\lambda_{0} + \epsilon}(x^{\pm}_{\lambda_{0} + \epsilon}, x) = \emptyset.
\end{equation}
This implies that after stabilizing twice and a double application of Theorem \ref{sullivan}, the birth is independent.

The spaces in (11) are empty. To see this, let $(x_{i}, y_{i}) \in \mathbb{R}^{2}$  be the local coordinates corresponding to the $\mathbb{R}^{2}$ space added after the first and the second stabilization, where $i=1,2$. Consider the projection map $\pi_{i}:\tilde{M}\times \mathbb{R}^{2}_{(x_{1},y_{1})}\times \mathbb{R}^{2}_{(x_{2},y_{2})}\rightarrow \mathbb{R}_{x_{i}}$. This map is $(\mathbf{J}_{\lambda_{0}+ \epsilon}\oplus J_{\text{std}}\oplus J_{\text{std}}, J_{\text{std}})$-holomorphic.
We have the following:
\begin{enumerate}
\item The image of $S^{2}(W)$ under the projection maps $\pi_{i}$ is the $x_{i}$-axis.
\item Assume $C-1 > 0$ in Theorem \ref{sullivan}, then $\text{supp}(\tilde{G}_{\lambda_{0}}-F_{0}) \subset (-1,1)^{n} \times (C-1, C+1)\times (C-1, C+1)$. This implies the the image of $\tilde{\tilde{W}}_{\lambda_{0}+ \epsilon}$ under the $\pi_{i}$-projection coincides with the image of $S_{+}(S_{-}(W'))$ outside $(-1,1)^{n}$\\$\times (C-1, C+1) \times (C-1, C+1)$. Thus for $i=1$ the image is the line $y_{1} = -2x_{1}$ outside $(-1,1) \times (C-1, C+1)$ union with some curve inside $(-1,1) \times (C-1, C+1)$. For $i=2$, the image is the line $y_{2} = 2x_{2}$ outside $(-1,1) \times (C-1, C+1)$ union with some curve inside $(-1,1) \times (C-1, C+1)$. Denote by $Q_{i} = \pi_{i}(S_{+}(S_{-}(W')))$.  
\end{enumerate}
Suppose there is a strip $u\in \mathcal{M}^{1}_{\lambda_{0} + \epsilon}(x, x^{\pm}_{\lambda_{0} + \epsilon})$ and consider the map $\pi_{1}\circ u: \mathbb{R}\times [0, 1] \rightarrow \mathbb{R}^{2}$. This map is holomorphic so it preserves orientations. Moreover, due to the boundary conditions, the regions where the image of the strip are mapped (at least partially), by the map $\pi_{1}\circ u$ are: $R_{1}=\{(x_{1},y_{1}) \vert -2x_{1} < y_{1} \hspace{0.2cm} \& \hspace{0.2cm} 0 < y_{1}\}$ and $R_{2}=\{(x_{1},y_{1}) \vert  y_{1} < 0 \hspace{0.2cm} \& \hspace{0.2cm} y_{1} < -2x_{1}\}$.
\begin{figure}[H]
\includegraphics[scale=0.5]{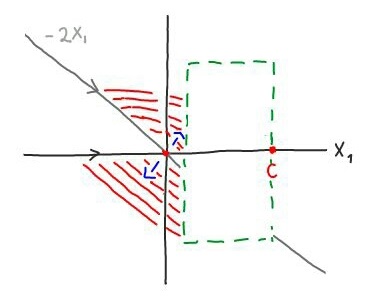}
\caption{Image of a strip in $ \mathcal{M}^{1}_{\lambda_{0} + \epsilon}(x, x^{\pm}_{\lambda_{0} + \epsilon})$.}
\end{figure}
Hence, by the open mapping theorem the image of an open subset of the domain $\mathbb{R}\times [0, 1]$ should be an open subset of one of these regions. If this is the case, then the image will contain the unbounded region $R_{1}-Q_{1}$. This contradicts the fact that the map has finite energy and thus no $\mathbf{J}_{\lambda_{0}+\epsilon}\oplus J_{\text{std}}\oplus J_{\text{std}}$-holomorphic strips from $x^{\pm}_{\lambda_{0}+\epsilon}$ to other intersection point can exist. To see that there are no strips arriving to $x^{\pm}_{\lambda_{0}+\epsilon}$, we use the same argument but using the second projection $\pi_{2}$. 
The isotopy given by $(\psi^{\tilde{G}}_{\lambda}, \mathbf{J}_{\lambda})$, where $\tilde{G}$ is the Hamiltonian generating the family in Theorem \ref{sullivan}, stabilized twice, satisfies the claim.
\end{proof}

The following lemma describes the moduli space of strips connecting the newly born intersection points between themselves.
\begin{lemma}\cite[Lemma 3.11]{Su}
Suppose there exists an $\epsilon > 0$ and a Darboux chart $(U, \phi)$ in $\tilde{ M}$ of the degenerate point $x_{0}$, such that for all $\lambda \in (\lambda_{0}, \lambda_{0} +\epsilon)$, there is a symplectomorphism $\phi: (U, \mathbf{J}_{\lambda}, \omega) \rightarrow (I^{2n}, J_{std}, dq \wedge dp)$, where $I$ is some small interval which contains $( -\delta, \delta)$ for some $\delta > 0$. Then, for any $\epsilon' \in (0, \epsilon)$, the moduli space $\mathcal{M}_{\lambda_{0}+ \epsilon'}(x^{-}_{\lambda_{0}+\epsilon'}, x^{+}_{\lambda_{0}+\epsilon'}) =
\{u\}$ contains a unique element.
\end{lemma}

\begin{proof}
To prove this lemma we show that any strip in $\mathcal{M}_{\lambda_{0}+ \epsilon}(x^{-}_{\lambda_{0}+\epsilon'}, x^{+}_{\lambda_{0}+\epsilon'})$ has image inside the neighborhood $U$, then we show that the strip should be $u$.
Suppose $w \in  \mathcal{M}_{\lambda_{0}+ \epsilon'}(x^{-}_{\lambda_{0}+\epsilon'}, x^{+}_{\lambda_{0}+\epsilon'})$, since $E(w) = \mathcal{A}_{0}(x^{-}_{\lambda_{0}+\epsilon})-\mathcal{A}_{0}(x^{-}_{\lambda_{0}+\epsilon})\rightarrow 0$\footnote{Here, $\mathcal{A}_{H}$ is the action functional,
$$\mathcal{A}_{H}: \mathcal{P}(W, W') \rightarrow \mathbb{R}, \hspace{0.1cm} \mathcal{A}_{H}(\gamma) = -\int_{[0,1]\times [0,1]}\tilde{\gamma}^{*}\omega + \int_{0}^{1} \gamma^{*}H.$$ In the transverse case we set $H = 0$.} when $\epsilon\rightarrow 0$, and $E(w) = \omega(w)$ then $Im(w) \subset U$ for $\epsilon$ small enough.
Since all the strips lie in the Darboux neighborhood, we can describe this situation using the local model. Suppose there is a $\mathbf{J}$-holomorphic strip $w$, consider the map $w$ composed with the projection $\pi_{i}:I^{2n}\rightarrow I^{n}$ mapping $(q_{1},p_{1},\ldots q_{n},p_{n}) \mapsto (q_{i},p_{i})$, denoted by $w_{i}=\pi_{i}\circ w$, the image of $W_{\lambda}$ under $w_{i}$ is a line for $i \neq 1$. Then, applying the open mapping theorem we deduce that $w_{i}$ is constant for $i \neq 1$ and $w_{1}=u$.
\end{proof}

\begin{remark}
The Floer complex associated to a pair of Lagrangian submanifolds $(W, W')$ and the regular data $(H, \mathbf{J})$ coincide with the complex associated to the stabilized data. There are no new intersection points, and by a similar argument to the one used before with involving the open mapping theorem, the moduli spaces are the same. For a detailed proof see \cite[Theorem 3.9]{Su}.
\end{remark}

The immediate consequence of the previous lemmas and the remark is the following description of the Floer complex after a birth occurs, for a regular homotopy:
\begin{align*}
\label{birth-death}
CF(W, \psi_{\lambda + \epsilon}(W'); &\mathbb{Z}(\pi_{1}W), \mathbf{J}_{\lambda - \epsilon}) = \\
&\big(CF(W, \psi_{\lambda - \epsilon}(W'); \mathbb{Z}(\pi_{1}W), \mathbf{J}_{\lambda - \epsilon})\oplus \mathbb{Z}(\pi_{1}W)\langle x^{-}_{\lambda + \epsilon}, x^{+}_{\lambda + \epsilon} \rangle, d_{\lambda - \epsilon}\oplus d_{T}\big), 
\end{align*}

where $d_{T}(x^{-}_{\lambda + \epsilon}) = \text{sign}(u)[\Theta (u)]x^{+}_{\lambda + \epsilon}.$ For definition of $\Theta$, see equation \ref{def-teta}.
\vspace{0.2cm}
\subsubsection{Handle-slide bifurcation analysis}

Let $x_{-}^{\Lambda} = \{x_{-}^{\lambda}\}_{\lambda\in \Lambda}$, $x_{+}^{\Lambda}=\{x_{+}^{\lambda}\}_{\lambda\in \Lambda}$ be two families of non-degenerate intersection points of $W \cap \psi_{\Lambda}(W)$. Let $u$ be a handle-slide from $x_{-}^{\lambda}$ to $x_{+}^{\lambda}$ occurring at $\lambda = \lambda_{0}$. 

For any two families of non-degenerate intersection points $x^{\Lambda}, y^{\Lambda}$, the $n+1$-dimensional one-parameter moduli space is given by:
$$\mathcal{M}^{n}_{\Lambda}(x^{\Lambda}, y^{\Lambda}):= \{ (\lambda, w) \hspace{0.02cm}\mid \hspace{0.02cm} \lambda \in \Lambda, \hspace{0.02cm} w \in \mathcal{M}_{\lambda}^{n}(x^{\lambda}, y^{\lambda}):= \mathcal{M}^{n}(x^{\lambda}, y^{\lambda})\}.$$ 

To describe the change in Whitehead torsion of the complex after a handle-slide bifurcation, it is required to understand the boundary components of the compactified one-parameter moduli spaces $\overline{\mathcal{\hat{M}}^{0}}_{\Lambda}(x^{\Lambda}_{-}, y^{\Lambda})$ and $\overline{\mathcal{\hat{M}}^{0}}_{\Lambda}(y^{\Lambda}, x^{\Lambda}_{+})$, for any one-parameter family of intersection points $y^{\Lambda}$ with $\mu(x_{-}^{\lambda}, y^{\lambda}) = 1$ or $\mu(y^{\lambda}, x_{+}^{\lambda}) = 1$, $\lambda \in \Lambda$.
 
The Gromov compactness theorem states that the moduli spaces can be compactified by broken trajectories. On the other hand, a Gluing theorem shows that each broken trajectory is a boundary component of the moduli space. These two results together imply the following relations:
\smaller
\begin{align}
\partial \overline{\mathcal{\hat{M}}}^{0}_{[\lambda_{0}-\epsilon, \lambda_{0}+ \epsilon]}(x^{\Lambda}_{-}, y^{\Lambda}) &= -\widehat{\mathcal{M}}_{\lambda_{0}-\epsilon}^{0}(x^{\lambda_{0}}_{-}, y^{\lambda_{0}}) \cup \widehat{\mathcal{M}}_{\lambda_{0}+\epsilon}^{0}(x^{1}_{-}, y^{1}) \cup -\{u\}\times\widehat{\mathcal{M}}_{\lambda_{0}}^{0}(x^{\lambda_{0}}_{+}, y^{\lambda_{0}})\\
\partial \overline{\mathcal{\hat{M}}}^{0}_{[\lambda_{0}-\epsilon, \lambda_{0}+ \epsilon]}(x^{\Lambda}, x^{\Lambda}_{+}) &= -\widehat{\mathcal{M}}_{\lambda_{0}-\epsilon}^{0}(x^{\lambda_{0}}, x^{\lambda_{0}}_{+}) \cup \widehat{\mathcal{M}}_{\lambda_{0}+\epsilon}^{0}(x^{1}, x^{1}_{+}) \cup \widehat{\mathcal{M}}_{\lambda_{0}}^{0}(x^{\lambda_{0}}, x^{\lambda_{0}}_{-})\times\{u\},
\end{align}
\normalsize
and out of this we can relate the Whitehead torsion of the Floer complexes before and after a handle-slide.

To show (12) and (13) we need to prove a Gluing theorem. The standard Gluing theorem concerns the gluing of non-degenerate trajectories when the linearization of the Cauchy-Riemann operator at each trajectory is surjective.

If a degenerate orbit appears, the linearization of the Cauchy-Riemann operator at the degenerate trajectory has a one-dimensional cokernel. In this situation we state the gluing theorem as follows,
 
\begin{theorem}[Gluing] \label{Gluing} Suppose there is a handle-slide $(\lambda_{0}, u) \in \mathcal{M}^{0}(x_{-}^{\Lambda}, x_{+}^{\Lambda})$ ($\mu(x_{-}^{\lambda_{0}}, x_{+}^{\lambda_{0}})=0$), for some $\lambda_{0} \in \Lambda = [0, 1]$.
Then for every intersection point $x^{\lambda_{0}}\in W \cap \psi_{\lambda_{0}}(W)$ with $\mu(x_{-}^{\lambda_{0}}, x^{\lambda_{0}})=1$, there is $\rho_{0} >0$ and a continuous map:
$$\mathcal{G}_{-}:\{\hat{u}\}\times \widehat{\mathcal{M}}^{0}_{\lambda_{0}}(x_{+}^{\lambda_{0}}, x^{\lambda_{0}})\times [\rho_{0}, \infty) \rightarrow \widehat{\mathcal{M}}^{0}_{\Lambda}(x^{\Lambda}_{-}, x^{\Lambda}),$$ 
such that for every $\hat{v}\in \widehat{\mathcal{M}}^{0}_{\lambda_{0}}(x_{+}^{\lambda_{0}},x^{\Lambda_{0}})$, if $l_{n} \in  \widehat{\mathcal{M}}^{0}_{\Lambda}(x^{\Lambda}_{-}, x^{\Lambda})$ is a sequence converging to $(\hat{u},\hat{v})$, then $l_{n} \in Im(\mathcal{G}_{-})$ for $n$ big enough. 
\end{theorem}
An analogous statement holds for $\overline{\mathcal{\hat{M}}^{0}}_{\Lambda}(x^{\Lambda}, x_{+}^{\Lambda})$, when $x^{\lambda}$ satisfies $\mu(x^{\lambda}, x_{-}^{\lambda}) = 1$ for any $\lambda \in \Lambda$.
\begin{remark}
A stronger version of Theorem 4.4 appears already as Proposition 4.2 in Floer's article \cite{Fl}, where he also suggested the proof. 

Lee gives a proof of Theorem 4.4 in \cite[Theorem 7.5]{Le0},for the Floer theory of Hamiltonian periodic orbits. 
\end{remark}
We will present Lee's proof adapted to the Lagrangian case. We use the naturality of Floer's equation, which leads to the identification of the moduli space of Floer strips with the moduli space of Floer half-tubes.

To a given $x\in W\cap \psi^{H}_{1}(W)$ associate an orbit of the Hamiltonian vector field, $\gamma_{x}(t):=(\psi_{t}^{H})^{-1}(x)$. 
In the same way a Floer's strip is transformed into a Floer's half-tube: 
\begin{align*}
\mathcal{M}(x,y)&\rightarrow\mathcal{M}(\gamma_{x}(t), \gamma_{y}(t))\\
v(s,t) &\mapsto (\psi_{t}^{H})^{-1}(v(s,t)),
\end{align*}
where
$\tilde{v}(s, t) = (\psi_{t}^{H})^{-1}(v(s,t))$ satisfies the Floer perturbed equation:
$$ \dfrac{\partial\tilde{v}}{\partial s}  + (\psi_{t}^{H})^{*}J_{t}(\tilde{v})(\dfrac{\partial\tilde{v}}{\partial t} + X_{H_{t}}(\tilde{v}))= 0,$$ an equation of the form:
\begin{equation}
\dfrac{\partial w}{\partial s}  + J_{t}(w)(\dfrac{\partial w}{\partial t} - X_{H_{t}}(w))= 0.
\end{equation}

With the above identification, we rewrite Theorem \ref{Gluing} as:
\begin{theorem}
Suppose there is a handle-slide $(\lambda_{0}, u) \in \mathcal{M}^{0}(\gamma_{-}^{\Lambda}, \gamma_{+}^{\Lambda})$ ($\mu(\gamma_{-}^{\lambda_{0}}, \gamma_{+}^{\lambda_{0}})=0$), for some $\lambda_{0} \in \Lambda = [0, 1]$.
Then for every orbit $\gamma^{\lambda_{0}}$ with $\mu(\gamma_{-}^{\lambda_{0}}, \gamma^{\lambda_{0}}) = 1$, there is $\rho_{0} >0$ and a continuous map:
$$\mathcal{G}_{-}:\{\hat{u}\} \times \mathcal{\hat{M}}^{0}_{0}(\gamma_{+}, \gamma)\times [\rho_{0}, \infty) \rightarrow \mathcal{\hat{M}}^{0}_{\Lambda}( \gamma_{-}^{\Lambda}, \gamma^{\Lambda})$$ 
such that for every $\hat{v}\in \mathcal{\hat{M}}^{0}_{\lambda_{0}}(\gamma_{+}^{\lambda_{0}}, \gamma^{\lambda_{0}})$,  if $l_{n} \in \mathcal{\hat{M}}^{0}_{\Lambda}( \gamma_{-}^{\Lambda}, \gamma^{\lambda})$ is a sequence converging to $(\hat{u},\hat{v})$, then $l_{n}\in Im(\mathcal{G}_{-})$ for $n$ big enough. 

The analogous statement holds for $\mathcal{\hat{M}}^{0}_{\Lambda}( \gamma^{\Lambda}, \gamma_{+}^{\Lambda})$, where $\gamma$ satisfies $\mu(\gamma, \gamma_{+})= 1$.
\end{theorem}  
The theorem is proven in three steps. 
\begin{itemize}
\item First step: \textit{Construction of a pre-gluing map.}
A point $(\lambda, w) \in\hat{\mathcal{M}}^{0}_{\Lambda}( \gamma_{-}^{\Lambda}, \gamma^{\Lambda})$ is a solution of the equation (14) at time $\lambda$. In this step an approximated solution $(\lambda_{0}, w_{\chi})$ depending on the data $\chi=(u,v,\rho)$ is defined.  The map assigning to each $\chi$ a pair $(\lambda_{0}, w_{\chi})$ is call the pre-gluing map.

The map $w_{\chi}$ does not satisfy equation (14), but it is close to a real solution since $\Vert \bar{\partial}_{\mathbf{J}_{\lambda_{0}}, H^{\lambda}}(w_{\chi})\Vert_{p} < \epsilon(\rho)$ is sufficiently small.  
\item Second step: \textit{Construction of $\mathcal{G}_{-}$.}
In this step we use the Newton-Picard method stated in Lemma \ref{Contraction}, to construct $\mathcal{G}_{-}$. The Newton-Picard method is applied to the operator $\bar{\partial}_{\mathbf{J}^{\Lambda}, H^{\Lambda}}$. Starting with a pre-glued solution $(\lambda_{0}, w_{\chi})$; the method gives us a honest solution. In order to apply this method the following is needed:
\begin{enumerate}
\item The existence of a uniformly bounded right inverse of the linearization of the operator 
\begin{equation}
\label{operatopr}
\mathcal{F} = \bar{\partial}_{\mathbf{J}^{\Lambda}, H^{\Lambda}}.
\end{equation}
Denote by $D_{(\lambda_{0}, w_{\chi})}$ the linearization of $\mathcal{F}$ at $(\lambda_{0}, w_{\chi})$.

In this step it is shown that $D_{(\lambda_{0}, w_{\chi})}$ has a right inverse, denoted by $P_{\chi}$, and that this right inverse is uniformly bounded independently of $\chi$. 
 
\item Obtaining a quadratic bound on the non-linear part of $\mathcal{F}$.

Near $(\lambda_{0}, w_{\chi})$, the operator $\mathcal{F}$ can be decomposed into linear part and a non-linear part. Let $(\lambda, w) = \tilde{\exp}_{(\lambda_{0}, w_{\chi})}((\lambda,\xi))$,
\begin{equation}
\mathcal{F}(\lambda, w) = \mathcal{F}((\lambda_{0}, w_{\chi})) + D_{(\lambda_{0}, w_{\chi})}(\lambda,\xi) + N_{(\lambda_{0}, w_{\chi})}(\lambda,\xi).
\end{equation}
If $(\lambda, \xi) = P_{\chi} \eta_{\chi}$, then a solution of $\mathcal{F}(\lambda, w) = 0$ is obtained as the fixed point of
\begin{equation}
\eta_{\chi} = -N(P_{\chi}(\eta_{\chi}))- \mathcal{F}(\lambda, w_{\chi}).
\end{equation} 
The existence of such a fixed point is guaranteed by the contraction mapping theorem which implies the following Lemma:
\begin{lemma}[1.2.1 in \cite{Le2}] \label{Contraction}
Let $C_{P}$ be the upper bound on $\Vert P_{\chi} \Vert$ as above, and suppose that there is a $\rho$-independent constant $k$ such that
\begin{align*}
\Vert \mathcal{F}((\lambda, w)) \Vert_{p} \leq \dfrac{1}{10kC^{2}_{P}}&\\
\forall (\lambda_{1}, \xi_{1}), (\lambda_{2},\xi_{2}), \hspace{0.2cm} \Vert N_{\chi}((\lambda_{1}, \xi_{1})) &- N_{\chi}((\lambda_{2},\xi_{2}))\Vert_{p} \leq \\
k ( \Vert (\lambda_{1}, \xi_{1})\Vert_{1,p} + \Vert &(\lambda_{2},\xi_{2})\Vert_{1,p} ) \Vert (\lambda_{1}, \xi_{1}) - (\lambda_{2},\xi_{2})\Vert_{1,p}.
\end{align*}

Then there exists a unique $\eta_{\chi}$ with $\Vert\eta_{\chi} \Vert_{p} \leq \frac{1}{5kC^{2}_{P}}$ solving (3.2.7). Moreover,
the solution $\eta_{\chi}$ varies smoothly with $\rho$, and $\Vert\eta_{\chi} \Vert_{p} \leq 2\Vert\mathcal{F}((\lambda, w_{\chi}))\Vert_{p}.$
\end{lemma}
The gluing map is defined by $\chi\mapsto \chi + P_{\chi}(\eta_{\chi})$. 
\end{enumerate}

\item Third step: \textit{Show the "surjectivity" of the gluing map:}
In this step we show that if $l_{n} \in \hat{\mathcal{M}}^{1}_{\Lambda}( \gamma_{-}, \gamma)$ is a sequence converging to $(\hat{u},\hat{v})$, then for all $n$ large enough, we have $l_{n}\in Im(\mathcal{G}_{-})$. 
\end{itemize}
Let $\gamma_{-}^{\Lambda} = (\gamma_{-}^{\lambda})_{\lambda \in \Lambda}$ and $\gamma_{+}^{\Lambda} = (\gamma_{+}^{\lambda})_{\lambda \in \Lambda}$, with $\gamma_{-}^{\lambda},\gamma_{+}^{\lambda} \in \mathcal{O}(H_{\lambda})$ for each $\lambda \in \Lambda$;  $$\mathcal{C}^{\infty}_{\Lambda}( \gamma_{-}^{\Lambda}, \gamma_{+}^{\Lambda}):=\{ (\lambda, u) \hspace{0.2cm}\vert \hspace{0.2cm} \lambda \in \Lambda,\hspace{0.2cm} u \in C^{\infty}(\mathbb{R}\times [0, 1], \tilde{M}),\hspace{0.2cm} \lim\limits_{s \rightarrow \mp  \infty} u(s,t) = \gamma_{\mp}^{\lambda}(t)\}.$$

For a vector bundle $\pi:E \rightarrow \mathbb{R}\times [0, 1]$ denote by  $W^{1,p}(E)$ the space of  $W^{1,p}$-sections $s:\mathbb{R}\times [0, 1] \rightarrow E$.
For $w \in C^{0}(\mathbb{R}\times [0, 1], \tilde{M})$ with boundary condition on a Lagrangian submanifold $W \subset \tilde{M}$, let $W^{1,p}_{L}(w^{*}(T\tilde{M})) := \{\xi \in W^{1,p}(w^{*}(T\tilde{M}))\hspace{0.2cm}\vert \hspace{0.2cm} \xi(s,i) \in T_{w(s,i)}W \hspace{0.2cm} i = 0,1  \}$ 
We now proceed to the proof:
\begin{proof}
\textbf{Step 1.} Construction of pre-gluing map.

Let $\epsilon \geq 0$ and $\beta^{-}, \beta^{+}:\mathbb{R}\rightarrow[0,1]$ be two smooth functions,
\begin{center}
$\beta^{-}(s)=\begin{cases} 
      1 & s\leq -1 \\
      0 & s\geq -\epsilon 
   \end{cases}$\hspace{0.5cm} 
   and\hspace{0.5cm} $\beta^{+}(s)=\begin{cases} 
      1 & s \geq 1\\ 0 & s\leq \epsilon
   \end{cases}.$ 
\end{center}
Pick $\rho_{0}\in \mathbb{R}^{+}$ large enough and let $\exp$ be the exponential map associated to a Riemmanian metric $\omega(-, \mathbf{J}-)$ for which $W$ is totally geodesic. Suppose that for $\vert s \vert \leq 1$ we have $u(s+\rho, t), v(s-\rho, t) \in \text{Im}(\exp_{\gamma_{+}})$. The pre-gluing map is defined as follows:
\begin{align*}
\#:\{u\}\times \mathcal{M}^{1}_{0}( \gamma_{+},\gamma ) \times [\rho_{0}, \infty) &\rightarrow \mathcal{C}^{\infty}_{\Lambda}( \gamma_{-}^{\Lambda}, \gamma^{\Lambda}),\\
 \chi = (u,v,\rho)&\mapsto (\lambda_{0}, w_{\chi}= u\#_{\rho}v),
\end{align*}
where $w_{\chi}:\mathbb{R}\times [0,1]\rightarrow M$ is given by
\begin{equation}
w_{\chi}(s,t) = \begin{cases} u(s+\rho, t)  & s \leq -1\\
\exp_{\gamma_{+}}(\beta^{-1}(s)\exp_{\gamma_{+}}^{-1}(u(s+\rho, t))\\ \hspace{1cm}+ \beta^{+}(s)\exp_{\gamma_{+}}^{-1}(v(s-\rho, t))) &  -1 \leq s \leq 1\\
v(s-\rho, t) & s \geq 1\end{cases}.
\end{equation}
The map $w_{\chi}(s,t)$, has the following properties:
\begin{itemize}
\item $\lim\limits _{\rho \rightarrow \infty} w_{\chi}(s-\rho,t)= u(s,t)$ in $\mathcal{C}^{\infty}_{loc}$,
\item $\lim\limits _{\rho \rightarrow \infty} w_{\chi}(s+\rho,t)= v(s,t)$ in $\mathcal{C}^{\infty}_{loc},$
\item $\lim\limits _{\rho \rightarrow \infty} w_{\chi}(s,t)= \gamma_{+}$ in $\mathcal{C}^{\infty}_{loc}.$
\end{itemize}
 In the same way we define the linearization of the pre-gluing map for $\rho$ fixed 
\begin{align*} 
\#_{\rho}:\text{Ker} D_{u}\times \text{Ker} D_{v} &\rightarrow W^{1,p}(w_{\chi}^{*}(T\tilde{M}))\\
\#_{\rho}(\xi_{0},\xi_{1})&\mapsto \xi_{0}\#_{\rho}\xi_{1}
\end{align*}
as follows:
$$\xi_{0}\#_{\rho}\xi_{1}(s,t) = \begin{cases} \xi_{0}(s+\rho, t)  & \\
D_{*}\text{exp}_{\gamma_{+}}(\beta^{-1}(s)D_{u(s+\rho, t)}\text{exp}_{\gamma_{+}}^{-1}(\xi_{0}(s+\rho, t)) \\\hspace{1cm}+ \beta^{+}(s)D_{v(s-\rho, t)}\text{exp}_{\gamma_{+}}^{-1}(\xi_{1}(s-\rho, t))) & \\
\xi_{1}(s-\rho, t), & \end{cases}$$
where $s$ varies in the same intervals as in the non-linear case and $$* = \text{exp}_{\gamma_{+}}(\beta^{-1}(s)\text{exp}_{\gamma_{+}}^{-1}(u(s+\rho, t)) + \beta^{+}(s)\text{exp}_{\gamma_{+}}^{-1}(v(s-\rho, t))).$$
The section $\xi_{0}\#_{\rho}\xi_{1}$ has the following properties:
\begin{itemize}
\item $\lim\limits _{\rho \rightarrow \infty} \xi_{0}\#_{\rho}\xi_{1} = 0$ in $\mathcal{C}^{0}_{loc}$,
\item $\xi_{0}\#_{\rho}\xi_{1} = 0$ for $s\in [-\epsilon, \epsilon].$
\end{itemize}
The pregluing map satisfies
$\Vert \bar{\partial}_{\mathbf{J}_{\lambda_{0}}, H^{\lambda}}(w_{\chi})\Vert_{p} < \varepsilon(\rho)$.

\textbf{Step 2.} Construction of $\mathcal{G}_{-}$.
We first show the existence of a uniformly bounded right inverse for the linearization of the operator $\mathcal{F}$ (\ref{operatopr}) at the point $(\lambda_{0}, w_{\chi})$. In order to do this we recall some definitions and standard results about the operator $\mathcal{F}$. For $p > 2$ let
$$\mathcal{P}_{loc}^{1,p} = \{ w\in  W^{1,p}_{loc}(\mathbb{R}\times [0,1], \tilde{M})  \hspace{0.2cm}|\hspace{0.2cm} w(s,i) \in W \hspace{0.2cm} i=0,1 \}.$$
For a fixed $\lambda\in \Lambda$, the space
\smaller
\begin{align*}
\mathcal{P}_{\lambda}^{1,p}(\gamma^{\lambda}_{-},\gamma^{\lambda}) = \{ w\in \mathcal{P}_{loc}^{1,p} \hspace{0.2cm}|\hspace{0.2cm} \exists \rho > 0,&  \hspace{0.1cm} \exists \xi_{-} \in W^{1,p}_{loc}((\gamma^{\lambda}_{-})^{*}(T\tilde{M})), \hspace{0.1cm} \exists \xi_{+} \in W^{1,p}_{loc}((\gamma^{\lambda})^{*}(T\tilde{M})),\\ 
\text{ such that if } &\vert s\vert > \rho, \hspace{0.2cm} w(s,t) = \exp_{\gamma^{\lambda}_{-}}(\xi_{-}(s, t)),\hspace{0.1cm} w(s,t) = \exp_{\gamma^{\lambda}}(\xi_{+}(s, t)) \hspace{0.2cm}\}, 
\end{align*}\normalsize 
is an infinite dimensional Banach manifold (see \cite[Theorem 3]{Fl1}). 

Consider the vector bundle $\mathcal{E}\rightarrow \mathcal{P}^{1,p}_{\Lambda}(\gamma_{-},\gamma)$, where the base space is defined by $$\mathcal{P}^{1,p}_{\Lambda}(\gamma_{-},\gamma)= \{(\lambda, w) \hspace{0.2cm}|\hspace{0.2cm} w\in \mathcal{P}_{\lambda}^{1,p}(\gamma^{\lambda}_{-},\gamma^{\lambda})\}$$ and the fiber is given by $\mathcal{E}_{(\lambda,w)} = L^{p}(w^{*}(T\tilde{M}))$.
 
The perturbed Cauchy-Riemann operator defines a section of the following bundle, 
\begin{equation}
\bar{\partial}_{\mathbf{J}^{\Lambda},H^{\Lambda}}:\mathcal{P}^{1,p}_{\Lambda}(\gamma_{-},\gamma)\rightarrow \mathcal{E},
\end{equation}
given by 
\begin{equation}
\bar{\partial}_{\mathbf{J}^{\Lambda},H^{\Lambda}}(\lambda,w) = \bar{\partial}_{\mathbf{J}^{\lambda}, H^{\lambda}}(w).
\end{equation}
The space $\mathcal{P}^{1,p}_{\Lambda}(\gamma_{-},\gamma)$ is a Banach manifold with local charts modeled by $T_{\lambda}\Lambda \times W^{1,p}_{L}(w^{*}(T\tilde{M})),$ and local coordinates given by exponential map:
\begin{align*}
&\tilde{\exp}_{(\lambda, w)}(\mu, \xi) := (\lambda + \mu , \exp_{w_{\chi}}(\xi + \zeta(\mu, R, R'))), \\
&\text{where }\zeta(\mu, R, R') = \beta^{-}(R+ s)\exp^{-1}_{w}(\gamma_{-}^{\lambda + \mu}) + \beta^{+}(-R'+ s)\exp^{-1}_{w}(\gamma^{\lambda + \mu}),
\end{align*}
for two large enough positive constants $R, R'$. 

With the previous definitions we have that
$\#(u,v,\rho)=(\lambda_{0},w_{\chi}) \in \mathcal{P}^{1,p}_{\Lambda}(\gamma_{-},\gamma)$. Then, in a neighborhood of $(\lambda_{0},w_{\chi})$ the section has the local expression:
\begin{equation}
\mathcal{F}\circ \tilde{\exp}_{(\lambda_{0}, w_{\chi})}:T_{\lambda_{0}}\Lambda \times W^{1,p}_{L}(w_{\chi}^{*}(T\tilde{M}))\rightarrow L^{p}(w_{\chi}^{*}(T\tilde{M})).
\end{equation}
To apply Lemma \ref{Contraction} to the operator $\tilde{\mathcal{F}}_{(\lambda_{0}, w_{\chi})} := \mathcal{F}\circ \tilde{\exp}_{(\lambda_{0}, w_{\chi})}$, we first have to show the existence of a uniformly bounded right inverse for the linearization of the operator $\tilde{\mathcal{F}}_{(\lambda_{0}, w_{\chi})}$ at $(0,0)$.

Denoting the linearized operator by $D_{(\lambda_{0}, w_{\chi})}$, this operator depends on the choice of a connection and on the constants $R, R'$ used in the definition of the local coordinates. The linearized operator has the following expression: 
\begin{equation}
D_{(\lambda_{0}, w_{\chi})}(\mu,\xi) = D_{w_{\rho}}\xi  + \mu \tilde{X}_{w_{\chi}}.
\end{equation}
The first term is $D_{w_{\chi}} = D \bar{\partial}_{ \mathbf{J}^{ \lambda_{0} }, H^{ \lambda_{0} } }(w_{\chi}) = \nabla_{s} + J_{t}^{\lambda_{0} } \nabla_{t} + S^{ \lambda_{0} }(s,t)$, where $S^{ \lambda_{0} }(s,t)$ is a zero-order operator such that $\lim \limits_{s\pm \infty} S^{ \lambda_{0} }(s,t) = S^{\pm, \lambda_{0} }$. The operators $S^{\pm}$ are symmetric operators on the solutions of the perturbed Cauchy-Riemann equation. 

The second term $\mu \tilde{X}_{\lambda_{0}}$ is the component on the $\lambda$-direction and is a cut-off version of:
\begin{equation}\label{tilde X}
\nabla_{\lambda}\nabla^{t}H^{\lambda}(t, w) + \nabla_{\lambda} J^{ \lambda }_{t} ( \partial_{t} w - X_{ H^{\lambda}_{t} }(w)).
\end{equation}

The operator $D_{(\lambda_{0}, w_{\chi})}$ is a Fredholm operator (see Theorem 4 \cite{Fl1}), with index $\text{ind}(D_{(\lambda_{0}, w_{\chi})}) = \text{ind}(D_{w_{\chi}}) + 1 = 2$. 
To see that the operator $D_{(\lambda_{0}, w_{\chi})}$ has a right inverse, observe that $\text{Ker}(D_{u}) \oplus \text{Ker}(D_{v}) \cong \text{Ker}(D_{w_{\chi}})$ (a proof of this can be found in \cite[Lemma 1.2.4]{Le2}), then, $\text{dim}(\text{Ker}(D_{\chi})) = \text{dim}(\text{Ker}(D_{u})) + \text{ dim}(\text{Ker}(D_{v})) = 2$. This combined with the computation of the index implies that the operator $D_{(\lambda_{0}, w_{\chi})}$ is surjective, since it is Fredholm then it has a right inverse. Denote its right inverse by $P_{(\lambda_{0}, w_{\chi})}$.

To see that the operator $P_{(\lambda_{0}, w_{\chi})}$ is bounded, we have to show that there exist $C \in \mathbb{R^{+}}$ such that, for $\rho > \rho_{0}$,\begin{align*}
\forall (\mu,\xi)\in T_{(\lambda_{0}, w_{\chi})}\mathcal{P}^{1,p}_{\Lambda}(\gamma_{-}, \gamma) \hspace{0.2cm}&\text{with } \xi \in (\text{Ker}D_{w_{\chi}})^{\perp}, \hspace{0.2cm}\\ C\Vert(\mu,\xi)\Vert_{W^{1,p}} &\leq  \Vert D_{(\lambda_{0}, w_{\chi})}(\mu,\xi)\Vert_{L^{p}}.
\end{align*}
We proceed by contradiction. Assume that the inequality above is not true. Then there exist a sequence $\{\rho_{n}\}_{n \geq N_{0}}$ with $\rho_{n}\rightarrow\infty$ and there exist $(\mu_{n},\xi_{n}) \in T_{(\lambda_{0}, w_{\chi_{n}})}\mathcal{P}^{1,p}_{\Lambda}(\gamma_{-},\gamma)$ such that
\begin{equation}
\Vert(\mu_{n},\xi_{n})\Vert_{W^{1,p}} = 1, \hspace{0.2cm} \text{ and }, \hspace{0.2cm}
\lim\limits_{n\rightarrow\infty} \Vert D_{(\lambda_{0}, w_{\chi})}(\mu_{n},\xi_{n})\Vert_{L^{p}} = 0.
\end{equation}

Let $\alpha(s):\mathbb{R}\rightarrow [0,1]$ be a smooth bump function such that $\alpha(s)=1$ if $|s|<\dfrac{1}{2}$ and $\alpha(s) = 0$ if $|s| > 1$. 
Consider $\alpha_{n}(s) = \alpha\big(\dfrac{s}{r_{n}}\big)=\alpha \big(\dfrac{2s}{\rho_{n}}\big)$ and 
the operator
\begin{align*}
D :T_{\lambda_{0}}\Lambda \times W^{1,p}_{L}(\tilde{\gamma}_{+}^{*}(T\tilde{M}))&\rightarrow L^{p}(\tilde{\gamma}_{+}^{*}(T\tilde{M})),\\
(\mu, \xi) &\mapsto \dfrac{\partial \xi}{\partial s} + A_{\gamma_{+}}\xi + \mu \nabla_{\lambda} \nabla H_{\lambda}(\lambda_{0}, \gamma_{+}),
\end{align*}
where $\tilde{\gamma}_{+}(s,t) = \gamma_{+}(t)$ and $A_{\gamma_{+}}$ is the linearization at $\gamma_{+}$ of the Hamiltonian flow equation
\begin{equation}
J_{t}(\gamma(t))\dfrac{d\gamma}{dt} + \nabla^{g^{t}}H_{t}(\gamma(t)).
\end{equation}
 The following estimates show that $D((\mu_{n},\alpha \xi_{n}))$ converges to zero in $L^{p}$-norm:
 \smaller
\begin{align}
\Vert &D(\alpha_{n}(\mu_{n},\xi_{n}))\Vert_{p} = \Vert \dfrac{\dot{\alpha}_{n}}{r_{n}}\xi_{n} + \alpha_{n} D(\alpha_{n}(\mu_{n},\xi_{n})) \Vert_{p}\\ =&
 \Vert \dfrac{\dot{\alpha}_{n}}{r_{n}}\xi_{n} + \alpha_{n} D((\mu_{n},\xi_{n})) -  \alpha_{n} D_{(\lambda_{0}, w_{\chi_{n}})}((\mu_{n},\xi_{n})) + \alpha_{n} D_{(\lambda_{0}, w_{\chi_{n}})}((\mu_{n}, \xi_{n}))\Vert_{p}\\
\leq & |\dfrac{\dot{\alpha}_{n}}{r_{n}}|\Vert \xi_{n}\Vert_{p} +\Vert \alpha_{n} D((\mu_{n},\xi_{n})) - \alpha_{n} D_{(\lambda_{0}, w_{\chi_{n}})}((\mu_{n},\xi_{n}))\Vert_{p} + \Vert\alpha_{n}D_{(\lambda_{0}, w_{\chi_{n}})}((\mu_{n}, \xi_{n}))\Vert_{p}\\
\leq & \dfrac{M}{r_{n}}\Vert \xi_{n}\Vert + \Vert (D - D_{(\lambda_{0}, w_{\chi_{n}})})((\mu_{n}, \xi_{n}))\Vert_{p,\Sigma_{n}}+ \Vert D_{(\lambda_{0}, w_{\chi_{n}})}((\mu_{n}, \xi_{n}))\Vert_{p,\Sigma_{n}}\\
\leq & \tau_{n} + \Vert (D - D_{(\lambda_{0}, w_{\chi_{n}})})((\mu_{n}, \xi_{n}))\Vert_{p,\Sigma_{n}},
\end{align}
\normalsize
where $M$ is the upper bound of $\dot{\alpha}$ on $\mathbb{R}$ and $\Sigma_{n}=([-r_{n}, r_{n}] \times [0,1])$. 

Notice that $\Vert \xi_{n}\Vert_{L^{p}} \leq \Vert \xi_{n}\Vert_{W^{1,p}} \leq 1$, and since
$\Vert D_{\chi_{n}}((\mu_{n}, \xi_{n}))\Vert_{L^{p}}\rightarrow 0$, this implies (3.2.20), for some $\tau_{n} \geq 0$ that goes to $0$ when $n$ goes to infinity. 

On the other hand $\Vert (D - D_{(\lambda_{0}, w_{\chi_{n}})})\Vert\rightarrow 0$ in $\mathcal{C}^{\infty}_{loc}$, which implies that $D(\alpha_{n}(\mu_{n},\xi_{n}))$ goes to zero in $L^{p}$-norm.

Notice now that $D$ is a Fredholm operator, hence there exist a subsequence of $\{\alpha_{n}(\mu_{n},\xi_{n})\}$ that converge to some element $(\mu_{\infty},\xi_{\infty})\in \text{Ker}(D)$ with $\Vert \xi_{\infty} \Vert_{1,p} + |\mu_{\infty}|$ bounded.

Consider the integral kernel $\mathcal{K}$\footnote{Given a differential operator $D$ on $\mathbb{R}$, the integral kernel of $D$ is the function $K(s,x)$ associated to the operator $\mathcal{K}(v) = \int\limits_{\mathbb{R}}K(s-x)v dx + C$ such that $\mathcal{K}(v)=u$ if $Du=v$.} of the operator $\dfrac{\partial}{\partial s} + A_{\gamma_{+}}$. Since the operator $A_{\gamma_{+}} = J_{t}(\tilde{\gamma})\dfrac{\partial}{\partial t} + S^{+}: W^{1,2}_{L}(\tilde{\gamma}_{+}^{*}(T\tilde{M}))\rightarrow L^{p}(\gamma_{+}^{*}(T\tilde{M}))$ is invertible and  $L^{2}$-self-adjoint, the integral kernel of the operator $\dfrac{\partial}{\partial s} + A_{\gamma_{+}}$ is given by the following expression:
\begin{equation}
\mathcal{K}(s)=\begin{cases} 
      e^{-A_{\gamma_{+}}^{+}s} p^{+} & s\geq 0 \\
      e^{A_{\gamma_{+}}^{-}(-s)}p^{-} & s < 0,
   \end{cases}
\end{equation}
where $p^{\pm}: L^{2}(\gamma_{+}^{*}(T\tilde{M}))\rightarrow L^{2}(\gamma_{+}^{*}(T\tilde{M}))$ are the orthogonal projections and $A^{\pm}_{\gamma_{+}}$ are the restrictions of $A_{\gamma_{+}}$ to the subspace of $L^{2}(\gamma_{+}^{*}(T\tilde{M}))$ generated by the positive or negative eigenvectors (this computation can be found in \cite[section 8.7.a]{AuDa}).

Applying the operator $\mathcal{K}$ on both sides of the equation 
\begin{equation}
\label{equ1}
(\dfrac{\partial}{\partial s} + A_{\gamma_{+}})\xi_{\infty} + \mu_{\infty} \nabla_{\lambda }\nabla H_{\lambda}(\lambda_{0}, \gamma_{+}) = 0,
\end{equation}
we obtain 
\begin{equation}
\xi_{\infty} = - \mu_{\infty}\mathcal{K} (\nabla_{\lambda}\nabla H_{\lambda}(\lambda_{0}. \gamma_{+}))
\end{equation}
However, $\nabla_{\lambda}\nabla H_{\lambda}(\lambda_{0}, \gamma_{+})$ is constant in $s$ and $K(s)$ has exponential decay on $\pm s$. This implies the exponential decay of $\xi_{\infty}$ when $s\rightarrow\pm \infty$. Therefore, the only pair that satisfies equation (\ref{equ1}) is $( \mu_{\infty}, \xi_{\infty}) = (0,0)$.

 This implies that $\lim\limits_{n \rightarrow \infty}\Vert(\mu_{n},\xi_{n})\Vert_{1,p, \Sigma_{n}}=0$.

Consider the sequences $(\beta^{-}(s+1)(\mu_{n},\xi_{n}))_{n\in \mathbb{N}}$ and $(\beta^{+}(s-1)(\mu_{n},\xi_{n}))_{n\in \mathbb{N}}$. After extracting a subsequence, we can show that these sequences converge to zero in $T_{\lambda_{0}}\Lambda \times W^{1,p}(w_{\chi}^{*}(T\tilde{M}))$. The proof of this is completely analogous to the standard case. For a detailed proof, the reader can consult \cite[Lemma 9.4.10-9.4.12]{AuDa}.
Finally, we have the following estimate:
\begin{align}
1 =& \Vert (\mu_{n},\xi_{n})) \Vert_{1,p}\\
 \leq & \Vert \beta^{-}(s+1)(\mu_{n},\xi_{n}))\Vert_{1,p} + \Vert \beta^{+}(s-1)(\mu_{n},\xi_{n}))\Vert_{1,p}\\ 
 & + \Vert(1- \beta^{-}(s+1)-\beta^{+}(s-1))(\mu_{n},\xi_{n}))\Vert_{1,p},
\end{align}
where the terms on the right tends to zero, giving a contradiction. Thus $P_{(\lambda_{0}, w_{\chi})}$
has a uniform bound.

To apply Lemma \ref{Contraction}, we need to obtain a quadratic bound on the non-linear part of $\tilde{\mathcal{F}}_{(\lambda_{0}, w_{\chi})}$. This bound follows from the following estimates:
\smaller
\begin{align}
\Vert N_{(\lambda_{0}, w_{\chi})}((\mu_{1}, \xi_{1}))-& N_{(\lambda_{0}, w_{\chi})}((\mu_{2}, \xi_{2}))\Vert_{p}\\ &\leq  \sup \limits_{(\mu, \xi) \in [(\mu_{1}, \xi_{1}), (\mu_{2}, \xi_{2})]} \Vert (dN)_{(\mu, \xi)}\Vert \Vert (\mu_{1}, \xi_{1}) - (\mu_{2}, \xi_{2})\Vert_{1,p} \\ &\leq K \sup \limits_{(\mu, \xi) \in [(\mu_{1}, \xi_{1}), (\mu_{2}, \xi_{2})]} \Vert (\mu, \xi) \Vert_{1,p} \Vert (\mu_{1}, \xi_{1}) - (\mu_{2}, \xi_{2})\Vert_{1,p} \\ &\leq K (\Vert (\mu_{1}, \xi_{1})\Vert_{1,p} + \Vert (\mu_{2}, \xi_{2})\Vert_{1,p} ) \Vert (\mu_{1}, \xi_{1}) - (\mu_{2}, \xi_{2})\Vert_{1,p}. 
\end{align}
\normalsize
Inequality (3.2.29) follows from a long computation that can be found for example in \cite[Theorem 3.a]{Fl1} or in the following Lemma:
\begin{lemma}\cite[11.4.5]{AuDa}
Let $r_{0} > 0$. There exist a constant $K > 0$ such that, for all $\rho > \rho_{0}$ and all $(\mu, \xi) \in T_{\lambda_{0}}\Lambda \times W^{1,p}(w^{*}_{\chi}(TM))$ with $(\Vert \xi \Vert \leq r_{0})$, we have $$\Vert dN_{(\lambda_{0}, w_{\chi})}(\mu, \xi)\Vert \leq K(\Vert \xi \Vert_{1,p} + \vert \mu \vert).$$
\end{lemma}
At this point, Lemma \ref{Contraction} applies, and we obtain that there exist a unique $\eta_{(\lambda_{0}, w_{\chi})}$ solving equation (17): $ \eta_{(\lambda_{0}, w_{\chi})} = -N_{(\lambda_{0}, w_{\chi})}(P_{(\lambda_{0}, w_{\chi})}(\eta_{(\lambda_{0}, w_{\chi})}))- \tilde{\mathcal{F}}_{\lambda_{0}, w_{\chi})}((0,0))$.

The gluing map is defined as follows:
\begin{align}
\mathcal{G}_{-}:\{u\} \times \hat{\mathcal{M}}^{1}_{0}(\gamma_{+}, \gamma)\times [\rho_{0}, \infty) &\rightarrow \hat{\mathcal{M}}^{1}_{\Lambda}( \gamma_{-}^{\Lambda}, \gamma^{\Lambda})\\
(\hat{u}, \hat{v}, \rho) &\mapsto \pi\circ \tilde{\exp}_{(0, w_{\chi})}(P_{(\lambda_{0}, w_{\chi})}(\eta_{(\lambda_{0}, w_{\chi})})).
\end{align}
Here, $\pi: \mathcal{M}^{1}_{\Lambda}( \gamma_{-}^{\Lambda}, \gamma^{\Lambda})\rightarrow \hat{\mathcal{M}}^{1}_{\Lambda}( \gamma_{-}^{\Lambda}, \gamma^{\Lambda}) $ is the projection to the unparametrized moduli space.
Once defined the gluing map, we arrive at the final step.
\vspace{0.5cm}

\textbf{Step 3.} In this step we show the so called "surjectivity" of the gluing map. The surjectivity in this context means that given a sequence $l_{n} \in \hat{\mathcal{M}}^{1}_{\Lambda}( \gamma_{-}, \gamma)$ converging to $(\hat{u},\hat{v})$, there is $N_{0} > 0$ such that for all $n > N_{0}$ we have $l_{n}\in Im(\mathcal{G}_{-})$.

Equivalently, we will show that any trajectory in the moduli space $\hat{\mathcal{M}}^{1}_{\Lambda}( \gamma_{-}, \gamma)$, in a small neighborhood of the the broken trajectory $(u,v)$ denoted by $U_{\epsilon}(u, v)$, is in the image of the gluing map. 

To do so, we see that for any $\epsilon > 0$ small enough there is a $\delta(\epsilon)$ such that the neighborhood of a broken trajectory $(u,v)$,
\begin{align*}
U_{\epsilon}(u,v) = \{(\mu, \nu)\in \mathcal{P}_{\Lambda}(\gamma_{-}^{\Lambda}, \gamma^{\Lambda}) \hspace{0.15cm}\vert\hspace{0.15cm} \vert\mu - \lambda_{0}\vert \leq \epsilon &\hspace{0.15cm} \forall s \hspace{0.1cm} \exists \tau(s) \hspace{0.15cm}\text{s.t. }\\ \text{d}(\nu(s), u(\tau))< \epsilon & \hspace{0.1cm} \vee \hspace{0.1cm} \text{d}(\nu(s), v(\tau))< \epsilon\},
\end{align*}
satisfies $$\hat{\mathcal{M}}^{1}_{\Lambda}( \gamma_{-}^{\Lambda}, \gamma^{\Lambda}) \cap U_{\epsilon}(u,v) \subset V_{\epsilon},$$
where $V_{\epsilon}$ is a neighborhood of the pre-glued trajectories defined as follows:
\begin{align*}
V_{\epsilon} = \{ \hspace{0.1cm} \tilde{\exp}_{(\lambda_{0}, T_{\tau}w_{\chi})}(\mu, \xi) \hspace{0.2cm} \vert\hspace{0.1cm} \chi = (u, v, &\rho), \hspace{0.2cm} v \in \hat{\mathcal{M}}_{\lambda_{0}}(\gamma_{+}, \gamma), \hspace{0.1cm}\rho \in (\rho_{0}, \infty),\\
&(\mu, \xi) \in \mathbb{R}\oplus W^{1,p}_{L}(w_{\chi}^{*}(T\tilde{M})), \hspace{0.1cm}  \vert \mu \vert + \Vert \xi \Vert_{1,p}\leq \epsilon \}.
\end{align*}
We first show that if $(\lambda,\hat{w}) \in  \hat{\mathcal{M}}^{1}_{\Lambda}( \gamma_{-}^{\Lambda}, \gamma^{\Lambda}) \cap U_{\epsilon}(u,v)$, then $\tilde{\exp}_{(\lambda_{0}, w_{\chi})}(\lambda - \lambda_{0}, \xi) = (\lambda, w)$ for some choice of $\chi = (u,v, \rho)$ and $(\lambda, \xi )$ in a small chart centered at $(\lambda_{0} , w_{\chi})$ satisfying $\Vert(\lambda, \xi )\Vert < \epsilon$ (for the norm induced by the Riemannian metric). 

To complete the proof we have to verify that $\vert \lambda - \lambda_{0}\vert + \Vert \xi \Vert_{1,p} < \delta(\epsilon)$, which implies that $(\lambda, w)\in V_{\delta(\epsilon)}$.

Let us choose $\chi = (u, v, \rho)$, such that $\mathcal{A}_{H_{\lambda_{0}}}(u(0)) = \mathcal{A}_{H_{\lambda}}(w(-\rho))$ and $\mathcal{A}_{H_{\lambda_{0}}}(v(0)) = \mathcal{A}_{H_{\lambda}}(w(\rho))$ (such a $\rho$ can be found by re-parametrization). Let be $R > 0$ with the property that outside $[-R, R] \times [0, 1]$ the value of $u, v$ is in a small ball $B_{\epsilon'}(\eta)$ for $\eta = \gamma_{-}, \gamma_{+}, \gamma$. 

Since $(\lambda, w_{\chi})$ is also close to $(u, v)$, we can suppose that $\epsilon$ is smaller than the injectivity radius of $\exp$  for a big enough $R$. Hence we can write $(\lambda, w) = \tilde{\exp}_{(\lambda_{0}, w_{\chi})}(\lambda - \lambda_{0}, \xi)$ for some $\xi$ with $\Vert \xi \Vert < \epsilon$ for the norm induced by the metric.

Observe now that $\vert \lambda -\lambda_{0} \vert < \epsilon$. To see that $\Vert \xi \Vert_{W^{1,p}}  < \delta(\epsilon)$, decompose $\Sigma:=\mathbb{R}\times [0, 1]$, the domain of $\xi$, in the following sub-domains: $\Sigma_{-}:= (-\infty, -R-\rho)\times [0, 1]$, $\Sigma_{-,0}:= ( -\rho -R, -\rho + R)\times [0, 1]$, $\Sigma_{0}:= (-\rho + R, \rho - R )\times [0, 1]$, $\Sigma_{+,0}:= (\rho - R, \rho+ R)\times [0, 1]$ and $\Sigma_{+}:= (R + \rho, \infty)\times [0, 1]$.

The weak compactness guarantees that in $\Sigma_{-,0}, \Sigma_{+,0}$, $\Vert \xi \Vert_{W^{1,p}} < C\epsilon R^{\frac{1}{p}}$, where $C > 0$ is a constant independent on $R$ and $\rho$. The exponential decay of $w_{\chi}$ and $w$ to $\gamma_{-}, \gamma$ implies that $\Vert \xi \Vert_{W^{1,p}} < C_{\pm}\epsilon $ on $\Sigma_{\pm}$.

Finally on $\Sigma_{0}$, there is $\zeta, \tilde{\zeta} \in C^{\infty}(\Sigma_{0},\tilde{\gamma}_{+}^{*}(T\tilde{M}))$ such that
$$w(s,t) = \exp_{\gamma_{+}(t)}(\zeta(s,t));\hspace{0.2cm} w_{\chi}(s,t) = \exp_{\gamma_{+}(t)}(\tilde{\zeta}(s,t)).$$
Since $w$ satisfies the perturbed Cauchy-Riemann equation, $\zeta$ satisfies the following equation:
\begin{equation}
(\frac{\partial}{\partial s} + A_{\gamma_{+}})\zeta + (\lambda - \lambda_{0}) \nabla_{\lambda} \nabla H^{\lambda_{0}}(\gamma_{+}) + N_{(\lambda_{0}, \gamma_{+})}(\lambda - \lambda_{0}, \zeta) = 0,
\end{equation}
where $N_{(0, \gamma_{+})}$ is the non-linear term left.
Applying the integral kernel to the equation below we obtain that:
\begin{equation}
\zeta + \mathcal{K}*(\mu \nabla_{\lambda} \nabla H^{\lambda_{0}})(\gamma_{+}) + N_{(\lambda_{0}, \gamma_{+})}((\lambda - \lambda_{0}), \zeta)) + \zeta_{0} = 0,
\end{equation}
where $\zeta_{0} \in \text{Ker}(\frac{\partial}{\partial s} + A_{\gamma_{+}})$ on $\Sigma_{0}$.
By definition, $\Vert \zeta \Vert_{1, \infty} \leq \epsilon$. Moreover, using the exponential growth/decay of $\zeta_{0}$, we can estimate:
\begin{equation}
\Vert \zeta_{0} \Vert_{p, \Sigma_{0}} \leq C\Vert \zeta_{0} \Vert_{\infty} \leq C_{0}(\epsilon + \Vert \mathcal{K} \Vert_{1} \Vert \mu \nabla_{\lambda} \nabla H^{0}(\gamma_{+}) + N_{(0, \gamma_{+})}(\mu, \zeta))\Vert_{\infty}), 
\end{equation}
and therefore
\begin{equation}
\Vert \zeta \Vert_{p, \Sigma_{0}} (\Vert \mathcal{K} \Vert_{p} + C')(C_{1}|\mu | + C_{2} \Vert \zeta \Vert_{W^{1,\infty}}) \leq C\epsilon.
\end{equation}
The exponential decay of $\mathcal{K}$ implies a uniform bound on $\Vert \mathcal{K} \Vert_{1}$ and $\Vert \mathcal{K} \Vert_{p}$.
On the other hand, from the exponential decay of $w_{\chi}$ to $\gamma_{+}$ over $\Sigma_{0}$, we obtain the estimate $\Vert \tilde{\zeta}\Vert_{W^{1, \infty}, \Sigma_{0}} < C_{4}e^{-C_{5}R}$, for some $C'>0$, and thus since similar estimates follow for $\frac{\partial}{\partial s}\zeta$ and $\frac{\partial}{\partial t}\zeta$, we obtain that
$$ \Vert \xi \Vert_{1,p,\Sigma_{0}} , C_{5}(\Vert \zeta \Vert_{1,p,\Sigma_{0}} + \Vert \tilde{\zeta}\Vert_{1,p,\Sigma_{0}}) \leq C_{6}\epsilon + C_{7}e^{-C_{5}R} R^{\frac{1}{p}}.$$
Summing up the estimates on each region of $\Sigma$, and putting $R = C''\epsilon^{-1}$, we obtain the bound $$|\mu| + \Vert \xi \Vert_{1,p} < C_{3} \epsilon = \delta(\epsilon).$$
\end{proof}

The expressions,
\smaller
\begin{align}
\partial \bar{\widehat{\mathcal{M}}}^{1}_{[\lambda_{0}-\epsilon, \lambda_{0}+ \epsilon]}(x^{\Lambda}_{-}, y^{\Lambda}) &= -\widehat{\mathcal{M}}_{\lambda_{0}-\epsilon}^{1}(x^{\lambda_{0}}_{-}, y^{\lambda_{0}}) \cup \widehat{\mathcal{M}}_{\lambda_{0}+\epsilon}^{1}(x^{1}_{-}, y^{1}) \cup -\{u\}\times\widehat{\mathcal{M}}_{\lambda_{0}}^{1}(x^{\lambda_{0}}_{+}, y^{\lambda_{0}})\\
\partial \bar{\widehat{\mathcal{M}}}^{1}_{[\lambda_{0}-\epsilon, \lambda_{0}+ \epsilon]}(x^{\Lambda}, x^{\Lambda}_{+}) &= -\widehat{\mathcal{M}}_{\lambda_{0}-\epsilon}^{1}(x^{\lambda_{0}}, x^{\lambda_{0}}_{+}) \cup \widehat{\mathcal{M}}_{\lambda_{0}+\epsilon}^{1}(x^{1}, x^{1}_{+}) \cup \widehat{\mathcal{M}}_{\lambda_{0}}^{1}(x^{\lambda_{0}}, x^{\lambda_{0}}_{-})\times\{u\},
\end{align}
\normalsize
follow from the previous Gluing theorem and the choice of orientations. It also implies the Corollary \ref{ColH-S} below.

To state it, consider the following linear map between the Floer complex before and after the handle-slide:
\begin{align*}
f: CF(W, \psi_{\lambda_{0}-\epsilon}(W');\mathbb{Z}(\pi_{1}W), \mathbf{J}_{\lambda_{0}-\epsilon}) &\rightarrow CF(W,\psi_{\lambda_{0}+\epsilon}; \mathbb{Z}(\pi_{1}W), \mathbf{J}_{\lambda_{0}+\epsilon}),\\
x^{\lambda_{0} - \epsilon} &\mapsto x^{\lambda_{0} + \epsilon} + \text{sign}(u)[\Theta(u)]\delta^{x^{\lambda_{0} - \epsilon}}_{x_{-}^{\lambda_{0} - \epsilon}} x_{+}^{\lambda_{0} + \epsilon}.
\end{align*}
Here, sign$(u)$\footnote{ The definition of sign$(u)$, will be object of the next section were we will discus orientations and we will fix the conventions used in the present work.} is a sign coming from the conventions on orientation.

\begin{corollary}\label{ColH-S}
The map $f$ defined before is a chain map. Thus, the Floer complex after a handle-slide $(\lambda_{0}, u) \in \mathcal{M}_{\Lambda}^{0}(x_{-}^{\Lambda},x_{+}^{\Lambda})$ occurs is described by:\\
$(CF(W,\psi_{\lambda_{0}+\epsilon}; \mathbb{Z}(\pi_{1}W), \mathbf{J}_{\lambda_{0}+\epsilon}), d_{\lambda_{0} + \epsilon}) \cong (CF(W, \psi_{\lambda_{0}-\epsilon}(W');\mathbb{Z}(\pi_{1}W), \mathbf{J}_{\lambda_{0}-\epsilon}), f d_{\lambda_{0}-\epsilon} f^{-1}).$
\end{corollary}

\begin{proof}
To simplify the notation, let us write $d_{\pm} = d_{\lambda_{0} \pm \epsilon}$. Notice that the only difference on the differential occurs when computing $d_{+}(x_{-}^{\lambda_{0} + \epsilon})$ or $d_{+}(x^{\lambda_{0} + \epsilon})$ on  an intersection point $x^{\lambda_{0} + \epsilon}$ with $\mu(x_{+}^{\lambda_{0} + \epsilon}, x^{\lambda_{0} + \epsilon}) = 1$. 

From the gluing theorem we have the relation $d_{+}x_{-}^{\lambda_{0} + \epsilon} - d_{-}x_{-}^{\lambda_{0} - \epsilon} + \text{sign}(u)[\Theta(u)]d_{-}x_{+}^{\lambda_{0} - \epsilon} = 0$, which clearly implies that $d_{+}(x_{-}^{\lambda_{0} + \epsilon}) = f d_{-} f^{-1}(x^{\lambda_{0} + \epsilon})$.  For the second case, let $x^{\lambda_{0} + \epsilon}$ be such that $\mu(x_{+}^{\lambda_{0} + \epsilon}, x^{\lambda_{0} + \epsilon}) = 1$. 

The gluing theorem implies the relation $d_{+}x^{\lambda_{0} + \epsilon} - d_{-}x^{\lambda_{0} - \epsilon} - \langle x_{-}^{\lambda_{0} - \epsilon}, d_{-}x^{\lambda_{0} - \epsilon}\rangle \text{sign}(u)[\Theta(u)]x_{+}^{\lambda_{0} + \epsilon} = 0$, thus $d_{+}(x^{\lambda_{0} + \epsilon}) = f d_{-} f^{-1}(x^{\lambda_{0} + \epsilon})$.
\end{proof}

 This completes the description of the bifurcations along the family $\{(\mathbf{J},H)\}_{\lambda \in \Lambda}$.

\subsection{Proof of Theorems \ref{theorem invariance} and \ref{s-theorem}}

We will use the following Lemma. 
\begin{lemma}[\cite{Su}, Lemma 2.13]
Let $(C, d)$ be a $\mathbb{Z}(\pi_{1}W)$-complex.
\begin{enumerate}
\item If $(T, d_{T})$ is a trivial acyclic $\mathbb{Z}(\pi_{1}W)$-complex $0\rightarrow T_{k}\rightarrow T_{k-1}\rightarrow 0$, where $T_{k} = T_{k-1} \cong \mathbb{Z}(\pi_{1}W)$ and the corresponding preferred bases are isomorphic. Then, $\tau((C, d)) = \tau( (C\oplus T, d\oplus d_{T}))$.
\item The chain map $f:C\rightarrow C$, $x\mapsto x + \delta^{x}_{x_{-}}\alpha x_{+}$ satisfies $\tau((C, d)) = \tau((C, fdf^{-1}))$.
\end{enumerate}
\end{lemma}
Finally we have all the elements to prove the results of the section.
\begin{proof}[Proof of Theorem \ref{theorem invariance}]

Choose a regular homotopy $\{\mathcal{D}_{\lambda}\}_{\lambda \in\Lambda} = \{ (\psi_{\lambda}, \mathbf{J}_{\lambda})\}$ joining $\mathcal{D}_{0}=(id,\tilde{\mathbf{J}})$ with $\mathcal{D}_{1}=(\phi^{H}_{1},\tilde{\mathbf{J}}')$ where $\{\psi_{\lambda}\}$ is a horizontal isotopy. Such a homotopy exists from Proposition 4.1.

Assume that a birth arrives at $\lambda = \lambda_{0}$. Since the birth is isolated, we can restrict to a sub-interval $[\lambda_{0} - \epsilon, \lambda_{0} + \epsilon]$ where the only bifurcation is the birth. We now apply Lemma 4.1 to find a regular homotopy that makes the birth independent after stabilizing twice. From Remark \ref{birth-death} and the Lemma 4.5, the Whitehead torsion does not change after a birth. We treat in the same way  a death bifurcation.

Suppose now that there is a handle-slide at $\lambda = \lambda_{0}$. The gluing theorem  and Corollary 4.1 and Lemma 4.5 imply that the Whitehead torsion does not change.

Finally, any intersection point $x$ generates a path $x^{\Lambda}$ along the isotopy. Then, the matrix $(d_{0} + \delta_{0})_{odd}$ associated to the complex $CF(W, W';\mathbb{Z}(\pi_{1}W),\tilde{J})$ may differ by some factors on $\mathbb{Z}(\pi_{1}W)$, from the matrix $(d_{1} + \delta_{1})_{odd}$ associated to the Floer complex $CF(W, \phi_{1}^{H}(W');\mathbb{Z}(\pi_{1}W),\tilde{J}')$. Since for any intersection point, $\psi_{\lambda}(x)$ is a loop. This difference can be expressed with a series of elementary operations. This implies that the Whitehead torsion of the two complexes coincide.
\end{proof}
\begin{proof}[Proof of Theorem \ref{s-theorem}]

By Corollary 2.1, an exact-orientable- spin Lagrangian cobordism $(W; L_{0}, L_{1})$, such that the map $i_{\#}:\pi_{1}(L_{i})\rightarrow \pi_{1}(W)$ induced by the inclusion $L_{i} \hookrightarrow W$ is an isomorphism for $i =0,1$, is an h-cobordism. Thus, it has a well defined Whitehead torsion.

The Whitehead torsion satisfies:
\begin{align}
\tau(W,L_{0}):=& \tau(C(\widetilde{W},\widetilde{L_{0}}))\\
 =& \tau(C(W,L_{0};\mathbb{Z}(\pi_{1}L_{0}))) \\
 = &\tau(CM(W;\mathbb{Z}(\pi_{1}L_{0}),(\widetilde{f}_{L_{0}}, g_{J}))) \\
 =&\tau(CF(W,\phi_{1}^{H}(W');\mathbb{Z}(\pi_{1}L_{0})),\mathbf{J})\\
 =&0.
\end{align}
The equality (51), follows from Theorem 3.3 and remark 3.1. The equality (52) follows from the proof of Proposition 2.4. For any Morse function without critical points on a collar neighborhood of the boundary, the Whitehead torsion of the cobordism coincides with the Whitehead torsion of the $\mathbb{Z}(\pi_{1}W)$-Morse complex.

Here, $(\widetilde{f}_{L_{0}}, g_{J})$ is a pair composed by a Morse function and $g_{J}$ is the Riemannian metric induced by an almost complex structure $J\in \widetilde{\mathcal{J}}_{B}$ as defined in section \ref{Chapter Floer homology}. 

Then, we can compute the Whitehead torsion of the cobordism $(W; L_{0}, L_{1})$, using the $\mathbb{Z}(\pi_{1}W)$-Morse complex $CM(W,(\widetilde{f}_{L_{0}}, g_{J});\mathbb{Z}(\pi_{1}L_{0}))$.

Since for small Hamiltonian function, $H = \epsilon\tilde{f}\circ \pi$ as in Proposition 2.4, the Morse complex and the Floer complex coincide, we obtain (52). The equality (53) follows from the Theorem \ref{theorem invariance}. Since the torsion of the Floer complex is independent on the choice of horizontal isotopy, and there is a horizontal isotopy that displaces $\phi_{1}^{H}(W')$ from $W$ (where $W' = \phi_{1}^{f_{L_{0}}\circ \pi}(W)$), for such a pair $(W, \phi_{1}^{H}(W'))$ the Floer complex is trivial so the Whitehead torsion of this complex vanishes.
\end{proof}

\section{Orientations and signs}
\label{Chapter orientations}

Let $(L_{0}, L_{1})$ be a pair of Lagrangians intersecting transversely and $(H,\mathbf{J})$ be a generic pair composed by a compactly supported Hamiltonian function and a time dependent almost complex structure $\mathbf{J}$. 

In this work we used orientation results for the following moduli spaces:

For $x, y \in L_{0}\cap L_{1}$, 
\begin{equation}\label{J-hom strip}
\mathcal{M}(x, y, \mathbf{J})= \{ u \in C^{\infty}(\mathbb{R}\times [0, 1], M) \hspace{0.2cm}\mid \hspace{0.2cm} \substack{ u(s, i) \in L_{i},\hspace{0.2cm} \bar{\partial}_{\mathbf{J}}(u)= 0\\
 \lim \limits_{s \to \mp \infty} u(s, t) = x/y, \hspace{0.2cm} E(u) < \infty }  \},
\end{equation}
\begin{equation}\label{mov bound}
\mathcal{M}_{\varphi}(x, y, \mathbf{J})= \{ u \in C^{\infty}(\mathbb{R}\times [0, 1], M) \hspace{0.2cm}\mid \hspace{0.2cm} \substack{u(s, 0) \in L_{0},\hspace{0.2cm} u(s,1) \in \varphi_{\beta(s)}(L_{1}), \\ \bar{\partial}_{\mathbf{J}}(u)= 0\\ 
\lim \limits_{s \to \mp \infty} u(s, t) = x/y, \hspace{0.2cm} E(u) < \infty } \}.
\end{equation}
For $\gamma_{-}, \gamma_{+}\in \mathcal{O}_{*}(H)$, where  $\mathcal{O}_{*}(H)$ is the space of paths connecting $L_{0}$ with $L_{1}$, that are orbits of the Hamiltonian flow $\psi^{H}_{t}$,
\begin{equation}\label{JH-tube}
\mathcal{M}(\gamma_{-}, \gamma_{+};H,\mathbf{J})= \{ u \in
C^{\infty}(\mathbb{R}\times[0,1],M)   \hspace*{0.2cm} |  \hspace*{0.2cm} \substack{u(s,i) \in L_{0}, \\ \overline{\partial}_{\mathbf{J}}u = - \nabla H(u,t),\\
  \lim \limits_{s\to \pm\infty}u(s,t) = \gamma_{\pm}(t), E(u) < \infty \\
} \}.
\end{equation}
Let $\{(\psi_{\lambda},\mathbf{J}_{\Lambda})\}_{\lambda \in \Lambda}$ be a one parameter family of regular data at each time $\lambda$. Then there are families of orbits $\gamma_{-}^{\Lambda}, \gamma_{+}^{\Lambda}$, 
\begin{equation}\label{J-parm}
\mathcal{M}_{\Lambda}(\gamma_{-}^{\Lambda}, \gamma_{+}^{\Lambda})= \{ (\lambda, w) \hspace{0.02cm}\mid \hspace{0.02cm} \lambda \in \Lambda, \hspace{0.02cm} w \in \mathcal{M}_{\lambda}^{1}(\gamma_{-}^{\lambda}, \gamma_{+}^{\lambda})= \mathcal{M}^{1}(\gamma_{-}^{\lambda}, \gamma_{+}^{\lambda};H_{\lambda}, \mathbf{J}_{\lambda})\}. 
\end{equation}
The propose of this section is to fix the conventions used to orient the previous moduli spaces. We will also recall some known results concerning the orientation problem for these moduli spaces, following \cite{FOOO} and \cite{Le2} closely. 

\subsection{Determinant line bundles and stabilization of Fredholm operators}

Let $E, F$ be Banach spaces. To any Fredholm operator $D:E\rightarrow F$ we can associate a one-dimensional real vector space called the \textit{line determinant}: $$\text{Det}(D) := \Lambda^{max}(\text{Ker}(D))\otimes \Lambda^{max}(\text{Coker}(D))^{*}.$$
Given a family of Fredholm operators $\{D_{\lambda}\}_{\lambda \in \Lambda}$, the vector bundle induced by taking the determinant line at each parameter $\bigcup \limits_{ \lambda \in \Lambda}\{\lambda\}\times\text{Det}(D_{\lambda})\rightarrow \Lambda$ is a line bundle, called the \textit{Determinant line bundle}. 

To a moduli space of $J$-holomorphic curves $\mathcal{M}$ one naturally associates a family of Fredholm operators, giving rise to a determinant line bundle on it. The problem of orientability of a moduli space translates to the study of its associated determinant line bundle. 
\begin{definition}
Let $D: E \rightarrow F$ be a Fredholm operator and $\psi:\mathbb{R}^{k}\rightarrow F$ be a linear map. A rank-k stabilization of $D$ is a finite dimensional extension, $\hat{D}_{\psi}: \mathbb{R}^{k}\oplus E \rightarrow F$ of it, given by $\hat{D}_{\psi}(\overline{r},\xi) = \psi(\overline{r}) + D(\xi)$.
\end{definition} 

\subsection{Conventions}
 
\subsubsection{Orientation for direct sums, exact sequences} 

Let $E = E_{1}\oplus\cdots\oplus E_{k}$ be the direct sum of an ordered k-tuple of oriented vector spaces. If $e_{i} \in \bigwedge^{max}E_{i}$ orients $E_{i}$, then we orient $E$ by $e_{1}\wedge \cdots \wedge e_{n} \in \bigwedge ^{max} E$.

An exact sequence of finite-dimensional vector spaces
\begin{equation}
0\rightarrow E_{1}\stackrel{i_{1}}\rightarrow F_{1}\stackrel{j_{1}}\rightarrow E_{2}\stackrel{i_{2}} \rightarrow F_{2} \cdots \stackrel{j_{n-1}}\rightarrow E_{n}\stackrel{i_{n}}\rightarrow F_{n}\rightarrow 0
\end{equation}
induces an isomorphism between $\bigotimes_{k}\bigwedge^{max}E_{k} \simeq \bigotimes_{k}\bigwedge^{max}F_{k},$ by writing $$E_{k} = B_{k}^{E}\oplus j_{k-1}B_{k-1}^{F},\hspace{0.2cm} F_{k} = i_{k}B_{k}^{E}\oplus B_{k}^{F}$$
for appropriated oriented subspaces $B_{k}^{E}, B_{k}^{F}$ where $i_{k}, j_{k}$ restrict to isomorphisms.

The exact sequences in this work have at most four non-zero terms, 
\begin{equation*}
0\rightarrow E_{1}\stackrel{i_{1}}\rightarrow F_{1}\stackrel{j_{1}}\rightarrow E_{2}\stackrel{i_{2}} \rightarrow F_{2} \rightarrow 0
\end{equation*}
in this case we have the following isomorphism:
$$\bigwedge^{max}E_{1}\otimes \bigwedge^{max}F_{2}^{*} \cong \bigwedge^{max}E_{2}\otimes \bigwedge^{max}F_{1}^{*}.$$
\subsubsection{Orientation for boundaries of manifolds}
Let $W$ be an oriented manifold with boundary. Denote by $o(W)$ its orientation and by $\overrightarrow{n}$ an exterior pointing vector on $\partial W$.
Then $\partial W$ is oriented by the following convention:
 $$o(\partial W) + \overrightarrow{n} = o(W)$$

\subsubsection{Orientation for stabilization}
Let $D: E \rightarrow F$ be a Fredholm operator with oriented kernel $Ker(D)$ and cokernel $Coker(D)$. Given a stabilization of it $\hat{D}_{\psi}$, we orient its determinant line  $Det(\hat{D}_{\psi})$ by the induced orientations from the following exact sequence:
$$0\rightarrow Ker(\hat{D}_{\psi} )\stackrel{\pi}\rightarrow (-1)^{k \text{ ind} D}\mathbb{R}^{k}\oplus Ker(D)\stackrel{\pi \circ \hat{D}_{\psi} }\rightarrow  Coker(D)\rightarrow Coker(\hat{D}_{\psi} )\rightarrow  0.$$

\subsubsection{Orientation for Glued Fredholm operators}
Given a  pair of gluable Fredholm operators $D_{1}, D_{2}$, the Gluing isomorphism 
$$\#: Det(D_{1}\# D_{2}) \simeq Det(D_{1})\otimes Det(D_{2})$$
induces an orientation on the determinant line of the  glued operator. 

Given orientations to $Ker(D_{i})$ and $Coker(D_{i})$ for $i = 0,1$, we orient the operator $D_{1} \# D_{2}$ by:
$$Ker(D_{1})\# Ker(D_{2}) :=(-1)^{dim C_{1} ind(D_{2})}Ker(D_{1})\oplus Ker(D_{2})$$ $$Coker(D_{1})\# Coker(D_{2}):= Coker(D_{1})\oplus Coker(D_{2}).$$
\begin{remark}
In particular, from the last choice we can conclude that the orientation of the moduli spaces $\mathcal{M}(x,y; \mathbf{J})$ (see \ref{J-hom strip}), appearing in section \ref{Chapter Floer homology}, section 1.1, are compatible with the gluing theorem since in this case the operators are surjective so there is no difference of sign.
\end{remark} 

\subsection{Orientability of the moduli spaces} 

The moduli space of interest here is $\mathcal{M}(x, y ; \mathbf{J})$ (see \ref{J-hom strip}), the space of Floer strips with boundary conditions along a pair of Lagrangians $(L_{0}, L_{1})$. In \cite{FOOO}, a sufficient condition for the orientability of this space is given.
The condition imposed is that the pair $(L_{0}, L_{1})$ has to be (relatively) spin. 
\begin{definition} A spin structure on an oriented vector bundle $E \rightarrow B$ is a homotopy class of a trivialization of $E$ over the 1-skeleton of $B$ that extends over the 2-skeleton of $B$. 
\end{definition} 
\begin{definition}
{A spin manifold is an oriented Riemannian manifold $M$, together with an spin structure on the tangent bundle of $M$.}
\end{definition}
\begin{definition}\cite[Definition 2.1.9]{FOOO}
Let $\Sigma$ be an oriented compact surface with boundary. A complex bundle pair $(E, \lambda)$ is a complex vector bundle $E \rightarrow \Sigma$ with a maximal totally real bundle $\lambda \rightarrow \partial \Sigma$, such that $E\mid_{\partial \Sigma}\cong \lambda \otimes \mathbb{C}$.
\end{definition}
 
We now recall the results about the orientability of moduli spaces. The following proposition is used to assign a canonical orientation at any determinant line associated to a point in the moduli space.

\begin{proposition}\cite[8.1.4]{FOOO} \label{propFo} Consider the complex bundle pair $(E, \lambda)$ over $(D^{2}, \partial D^{2})$. Suppose $\lambda$ is trivial. Then each trivialization on $\lambda$ canonically induces an orientation on $Det(\bar{\partial}_{(E,\lambda)})$.
\end{proposition}
The proof can be found in pages 677-679 \cite{FOOO}. 

The orientability of the moduli spaces of Floer strips is proved in \cite[Theorem 8.1.14]{FOOO}. The proof uses Proposition \ref{propFo} combined with a gluing argument. Additional choices of spin structures and orientations of auxiliary determinant lines are also required. 

Here, we will present the proof of an analogous statement (\ref{orientability}), for the moduli spaces of half-tubes $\mathcal{M}(\gamma_{-},\gamma_{+};H, J)$ (defined in \ref{JH-tube}) with boundary condition on a fixed Lagrangian. The idea of the proof is the following.  
 
Each orbit $\gamma \in \mathcal{O}_{*}(H)$ satisfies $[\gamma] = id \in \pi_{1}(M,L)$.
Let $D^{2}_{-} = D^{2}_{-} \cap (-\infty, 0] \times \mathbb{R}$ be the half disk. There exist a map $v_{\gamma}:(D^{2}_{-},\partial D^{2}_{-}\cap S^{1}, \partial D^{2}_{-} \cap i\mathbb{R})\longrightarrow (M,L,\gamma)$ such that $v_{\gamma}(0,t) = \gamma(\frac{t}{2}+\frac{1}{2})$. This map is a capping for $\gamma$.  

Fix the clockwise orientation on $\partial D^{2}_{-}$, then for any trajectory $u \in \mathcal{M}(\gamma_{-}, \gamma_{+})$ the map $w_{u}=v_{\gamma_{-}}\#u\#(-v_{\gamma_{+}}):(D^{2}, \partial D^{2})\longrightarrow (M,L)$ is a disk with boundary on $L$ and then the bundle pair $(w_{u}^{*}(TM), w_{u}\vert_{\partial D^{2}}^{*}(TL)) = (E_{w_{u}}, \lambda_{w_{u}})$ is a complex bundle pair.

If $\lambda_{w_{u}}$ is trivial, then by Proposition \ref{propFo} a trivialization of $\lambda$ induces a canonical orientation on $det(\bar{\partial}_{(E_{w_{u}}, \lambda_{w_{u}})})$. 

A gluing argument applied to an appropriate one-parameter family of operators relating the operator $\bar{\partial}_{(E_{w_{u}}, \lambda_{w_{u}})}$ with the operator $D_{u}$ (the linearization of the operator $\bar{\partial}_{\mathbf{J}, H}$ at $u$) induces a canonical orientation on $det(D_{u})$. After a canonical orientation is given to each $u \in \mathcal{M}(\gamma_{-},\gamma_{+};H, J)$, the spin condition guarantees a global orientation on $\mathcal{M}(\gamma_{-},\gamma_{+};H, J)$.

We now proceed to give more details and we fix some notation. Let $R\in \mathbb{R}$, and define the spaces:
\begin{itemize}
\item $\Theta_{-} = D^{2}_{-}\#\mathbb{R}\times[0,\infty),$
\item $\Theta = \mathbb{R}\times[0,1],$
\item $\Theta_{+} =\mathbb{R}\times(-\infty,0]\#D^{2}_{+}.$
\end{itemize} 
Denote by $\Omega_{R} = \Theta_{-}\#_{R}\Theta\#_{R}\Theta_{+}$ the glued domain.

To simplify notation let us write $v_{-}, v_{+}$ for the capping maps associated to $\gamma_{-}, \gamma_{+}$. On the spaces above, define the maps
\begin{itemize}
\item $\tilde{v}_{-}:\Theta_{-}\longrightarrow M$ by
$\tilde{v}_{-}\vert_{D^{2}_{-}}=v_{-}$ and $\tilde{v}_{-}\vert_{\mathbb{R}\times[0,\infty)}(s,t) = v_{-}(0,t),$ 
\item $\tilde{v}_{+}:\Theta_{+}\longrightarrow M$ by
$\tilde{v}_{+}\vert_{D^{2}_{+}}=v_{+}$ and $\tilde{v}_{+}\vert_{\mathbb{R}\times (-\infty, 0]}(s,t) = v_{+}(0,t).$
\end{itemize}
We can now define a pre-glued map
$w_{u, R}: (\Omega_{R}, \partial \Omega_{R}) \longrightarrow (M, L)$ by
\smaller
$$w_{u, R}(s, t)=\left\{ \begin{array}{lcc}
            \tilde{v}_{-}(s+2R, t) &   s \leq -R-1 
             \\ exp_{\gamma_{-}(t)}(\beta^{-}(s+R)exp^{-1}_{\gamma_{-}(t)}(\tilde{v}_{-}(s+2R, t)) 
             \\ + \beta^{+}(s+R)exp^{-1}_{\gamma_{-}(t)}(u(s, t)))&   -R-1 \leq s < -R+1 
             \\ u(s,t) &  -R+1 \leq s < R-1 
            \\exp_{\gamma_{+}(t)}(\beta^{-}(s-R)exp^{-1}_{\gamma_{+}(t)}(\tilde{v}_{+}(s-2R, t)) 
             \\ + \beta^{+}(s-R)exp^{-1}_{\gamma_{+}(t)}(u(s, t)))&   R-1 \leq s < R+1 
             \\ \tilde{v}_{+}(s-2R, t) &   s \geq 1+R, 
          
             \end{array}
   \right.$$
   \normalsize
where $\beta^{-}, \beta^{+}$ are defined as in Step 1 of the proof of Theorem 4.4.

Let $p > 2$, and consider the following two operators:   
$$\overline{\partial}_{\pm}: W^{1,p}\big(\Theta_{\pm},\partial \Theta_{\pm}; \tilde{v}_{\pm}^{*}(TM),\tilde{v}_{\pm}^{*}\vert_{\partial \Theta_{\pm} }(TL)\big)\longrightarrow  L^{p}\big(\Theta_{\pm},\tilde{ v}^{*}_{\pm}(TM)\otimes\Lambda^{0,1}(\Theta_{\pm})\big),$$ 
which is the Cauchy-Riemann operator with Lagrangian boundary condition and the operator $$\overline{\partial}_{w_{u,R}}:W^{1,p}\big(\Omega_{R},\partial \Omega_{R}; w_{u, R}^{*}(TM),w_{u, R}^{*}\vert_{\partial \Omega_{R}}(TL)\big)\longrightarrow L^{p}\big(\Omega_{R}, w_{u, R}^{*}(TM)\otimes\Lambda^{0,1}(\Omega_{R})\big),$$
is the linearization of the operator $\bar{\partial}_{\mathbf{J},H, R}$ defined as follows. 
Let $\alpha_{R}:\mathbb{R}\longrightarrow [0,1]$ be a smooth function such that 
$$\alpha_{R}(s)= \begin{cases}
             1 &  \text{ if } s   \leq -R-\dfrac{1}{2}
             \\ 0 &  \text{ if } -R \geq s \leq R 
             \\ 1 &  \text{ if } s\geq R+ \dfrac{1}{2}.      
             \end{cases}
   $$ 
The operator is given by
$$ \bar{\partial}_{\mathbf{J},H, R}u =\dfrac{\partial u}{\partial s} + J_{t}(u)\dfrac{\partial u}{\partial s} + \alpha_{R}(s) \nabla H_{t}(u).$$
For $\xi \in W^{1,p}\big(\Omega_{R},\partial \Omega_{R}; w_{u, R}^{*}(TM),w_{u, R}^{*}\vert_{\partial \Omega_{R}}(TL)\big)$ and a choice of Hermitian connection $\nabla$ for which $\mathbf{J}$ is parallel, the linearization is given by:
\begin{equation}
\overline{\partial}_{w_{u,R}}\xi = \nabla_{s}\xi + J_{u}\nabla_{t}\xi + \nabla_{\xi}(\alpha_{R}(s)\nabla H_{t}(u)).
\end{equation}
By Proposition \ref{propFo}, the determinant line $det(\overline{\partial}_{w_{u,R}})$ has a canonical orientation (since its the sum of the Cauchy-Riemann operator with a zero-order operator), induced from a trivialization of $TL$.

The Gluing theorem for Fredholm operators (see for example \cite{Sc}) implies the isomorphism between the determinant lines:
\begin{equation}
\label{detIso}
\det\overline{\partial}_{w_{u,R}} \cong \det \overline{\partial}_{-} \otimes \det(D\overline{\partial_{J}})_{u}\otimes \det \overline{\partial}_{+}.
\end{equation}
From \ref{detIso}, it follows that a choice of orientation on $\det \overline{\partial}_{\pm}$ determines an orientation on $\det(D\overline{\partial_{J}})_{u}$.

It remains to show that the previous choices lead to a trivial determinant line bundle $$\bigcup \limits_{ u\in\mathcal{M}(\gamma_{-},\gamma_{+})} \{u\}\times Det(D_{u})\rightarrow \mathcal{M}(\gamma_{-},\gamma_{+}).$$ To see this, let $u_{0}, u_{1} \in \mathcal{M}(\gamma_{-},\gamma_{+})$ be in the same connected component, orient $\det(D_{u_{0}})$ using the trivialization that comes from the spin structure on $L$ and a choice of orientation on $\det \overline{\partial}_{\pm}$.
Consider a path $\phi(t)$ joining $u_{0}$ to $u_{1}$. Since $Det(D_{\phi_{t}})$ is trivial, we can transport the orientation to $\det(D_{u_{1}})$. The orientation induced on $\det(D_{u_{1}})$ does not depend on the choice of path since for any other path $\xi(t)$, we can consider the  loop given by,
\begin{align*}
\tilde{\phi}: \Omega_{R}\times S^{1} &\rightarrow M,\\
(z,t)&\mapsto (\phi(t)\#(-\xi(t)))(z).
\end{align*}
Since $L$ is spin, for any $z \in \partial \Omega_{R}$,  the spin structure gives a trivialization of $(\omega \mid_{\{z\}\times S^{1}})^{*}(TL)$ which extends to a trivialization of $(\omega\mid_{\partial \Omega_{R} \times S^{1}})^{*}(TL)$, since this trivialization is unique up to homotopy the induced orientation is also unique.
From the discussion above we can conclude the following Theorem:

\begin{theorem}
\label{orientability} Suppose that a Lagrangian $L$ is spin. Then for any $\gamma_{0}, \gamma_{1} \in \mathcal{O}_{*}(H)$, the moduli space $\mathcal{M}(\gamma_{0}, \gamma_{1}; H, \mathbf{J})$( see \ref{JH-tube}) of connecting orbits in Lagrangian Floer theory $HF(L, L; H, \mathbf{J})$, is orientable. Furthermore, orientations on $Det(\bar{\partial}_{\pm})$ and the spin structure canonically determine an orientation on $\mathcal{M}(\gamma_{0}, \gamma_{1}; H, \mathbf{J})$.
\end{theorem}
Theorem \ref{orientability} is based on Theorem \cite[8.1.14]{FOOO}, wich establishes conditions for the orientability of the moduli space $\mathcal{M}(x, y;\mathbf{J})$ 
(defined in expression \ref{J-hom strip}), of Floer strips connecting intersection points, with boundary condition in a pair of Lagrangians $(L_{0}, L_{1})$. 

Theorem \ref{orientability} is less general, since it refers to moduli spaces of Floer half tubes (defined in \ref{JH-tube}), connecting trajectories of the Hamiltonian vector field with boundary condition on a fixed Lagrangian, these are the ones used in section \ref{chapitre whitehead} for the bifurcation analysis.

In Theorem \ref{Teorema} we use the orientability of the moduli spaces in expresions \ref{J-hom strip} and \ref{mov bound}. Theorem 8.1.14 in \cite{FOOO} implies the orientability of these moduli spaces.

The proof of Theorem 8.1.14 in \cite{FOOO} is analogous to the one of the theorem above.
Some differences are the following: instead of choosing capping disks associated to each trajectory $\gamma \in \mathcal{O}_{\eta}(H)$, a path $\lambda_{x}(t)\in \Lambda(T_{x}M)$, joining $T_{x}L_{0}$ with $T_{x}L_{1}$ is associate to each intersection point $x\in L_{0}\cap L_{1}$ in a canonical way using the spin structures on $L_{0}, L_{1}$.
Now, for each Floer strip $u \in \mathcal{M}(x, y; L_{0}, L_{1})$, consider the loop  $\lambda(t)=T_{u(s,0)}L_{0}\#\lambda_{y}\#(-T_{u(s,1)}L_{1})\-\lambda_{x}$.

Let the map $\tilde{u}:D^{2} \rightarrow M $ be $\tilde{u}\mid_{D^{2}-\{\pm 1\}}= u$ and $\tilde{u}(\pm 1) = x/y$. We obtain a bundle pair $(\tilde{u}^{*}(TM), \lambda(t))$, and applying Proposition 4.1.1 we obtain an orientation on $Det(\tilde{u})$.

To orient $Det(u)$, we use virtual operators defined on half-disks with boundary conditions involving the paths $\lambda_{x}$ and $\lambda_{y}$. From a gluing argument on these bundles we obtain an expression analogous to the one in \ref{detIso}. Finally, we use the spin structures to transport the orientation on the connected component of $u$.

When moving boundary conditions are present, \cite[Theorem 8.1.14]{FOOO} still applies: we can transform a strip with moving boundary conditions into a strip with fixed boundary conditions using the map $\psi_{t\beta(s)}$, defined in section 1.1.2.4. The new strip satisfies a perturbed $\overline{\partial}_{H, J}$-type equation. The linearization of it is a Fredholm operator. The spin structure induces an orientation on the determinant line bundle of the moduli space composed by strips with fixed boundary condition. Since this moduli space is diffeomorphic to the space with moving boundary condition, the diffeomorphism induces an orientation on it.

\subsubsection{Orientability of the moduli spaces $\mathcal{M}_{\Lambda}(\gamma_{-}^{\Lambda}, \gamma_{+}^{\Lambda})$(\ref{J-parm})}

Let $\Lambda=[0, 1]$ and consider a generic homotopy $\mathcal{D}_{\Lambda}=\{\mathcal{D}_{\lambda} =(\psi_{\lambda}, \mathbf{J}_{\lambda})\}_{\lambda \in \Lambda}$ joining two generic data $\mathcal{D}_{0}=(\psi^{H_{0}}, \mathbf{J}_{0})$ and $\mathcal{D}_{1}=(\psi^{H_{1}},\mathbf{J}_{1})$ and assume that for each $\lambda$ the data $\mathcal{D}_{\lambda}$ is regular.
The spin structure on $L$ is transported along the homotopy $\psi_{\lambda}$. Assume that $\gamma^{\Lambda}_{-}, \gamma^{\Lambda}_{+}$ are two one-parameter families of non-degenerate orbits. Let $(\lambda, u_{\lambda}) \in \mathcal{M}_{\Lambda}(\gamma^{\Lambda}_{-}, \gamma^{\Lambda}_{+})$. 

To orient the determinant line $Det(D_{(\lambda,u_{\lambda})})$ at any point we consider the operator $D_{(\lambda,u_{\lambda})}$ as stabilization of the operator $D_{u_{\lambda}}$ as follows:

Consider the map $\Psi: \mathbb{R} \rightarrow \mathbb{R}$ to be identified with the map $\tilde{X}_{u_{\lambda}}: T_{\lambda}\Lambda \rightarrow \mathbb{R}\langle \tilde{X}_{u_{\lambda}}\rangle$\footnote{Here, $\tilde{X}_{u_{\lambda}}$ represents a cut-off version of $
\nabla_{\lambda}\nabla^{t}H^{\lambda}(t, u) + \nabla_{\lambda } J^{ \lambda }_{t} ( \partial_{t} u - X_{ H^{\lambda}_{t} }(u))$,  the linear operator in the linearization of $D_{(\lambda, u_{\lambda})}.$}, defined by $\mu \mapsto \mu \tilde{X}_{u_{\lambda}}$.

Then the operator $D_{(\lambda,u_{\lambda})}$ is locally given by a stabilization as follows:
$D_{(\lambda,u_{\lambda})}(\mu, \xi) = \hat{D}_{\phi}(\mu, \xi) = \phi(\mu) + D_{u_{\lambda}}(\xi)$.

Since for a generic homotopy the operator $D_{u_{\lambda}}$ is surjective, the determinant line $Det(D_{u_{\lambda}}) = (\bigwedge^{max}Ker(D_{u_{\lambda}}))$.  Then, from section 4.2.2 the orientation on $Det(D_{(\lambda,u_{\lambda})})$ is given by $$Det(D_{(\lambda,u_{\lambda})}) \cong \bigwedge^{max}(-\mathbb{R}\oplus Ker(D_{u_{\lambda}})).$$

Following the same argument as the one given in the proof of Theorem \ref{orientability},
the spin structure grantees that the orientation can be transported on a connected component of the base $\mathcal{M}_{\Lambda}(\gamma^{\Lambda}_{-}, \gamma^{\Lambda}_{+})$. Then, $\mathcal{M}_{\Lambda}(\gamma^{\Lambda}_{-}, \gamma^{\Lambda}_{+})$ has a canonical orientation up to the choices.

\subsubsection{Orientations in Morse theory} Determinant lines also induce orientations on spaces of flow lines connecting critical points of a Morse function. The orientations induce by the determinant bundles coincide with the classical orientation of these spaces given by the intersection of stable and unstable manifold (under some choices). This construction appears in Appendix B \cite{Sc}. 

In particular, this implies that in the exact case treated in section 1.2.3, the differential of the Floer complex and Morse complex agree. Since the moduli spaces \ref{J-hom strip} are identified with the spaces of flow lines connecting critical points, and the orientations of both spaces can be defined by determinant bundles. 

\subsection{Signs} 
In this section we define the signs of a Floer strip and we compute the sign of a handle slide.

\subsubsection{Signs on Floer strips}
Let $\hat{u} \in \hat{\mathcal{M}}^{0}(\gamma_{1}, \gamma_{2})$ be a 0-dimensional moduli space of unparametrized trajectories. Let $u$ be any representative of $\hat{u}$ in $\mathcal{M}(\gamma_{1}, \gamma_{2})$. Then sign$(\hat{u})$ = sign$(u)$. is defined by the relation $sign(u)(\mathbb{R} \langle\partial_{s}u \rangle) = Det(D_{u})$, where the last space has a canonical orientation induced by the spin structure and capping.

For $u \in \mathcal{M}_{\varphi}^{0}(\gamma_{1}, \gamma_{2})$, we set $sign(u)(1\otimes 1^{*}) = Det(D_{v_{u}})$. Here $v_{u}$ is defined in section \ref{Chapter Floer homology}, (7)subsection 2.1.4, and $D_{v_{u}}$ is the linearization of the $\overline{\partial}$-type operator defined by the equation satisfied by $v_{u}$. 

If $u$ is a handle-slide occurring at $\lambda = \lambda_{0}$ then $sign(u)(\mathbb{R} \langle\partial_{s}u \rangle\otimes \mathbb{R}\langle \tilde{X}_{u}\rangle^{*}) = Det(D_{u})$. Here, $\tilde{X}_{u}$ is as defined in \ref{tilde X}.

\subsubsection{Sings of boundaries of moduli spaces involving handle-slides} 

Let $(\lambda_{0}, u_{\lambda_{0}})\in \mathcal{M}_{\Lambda}(\gamma_{-}^{\Lambda}, \gamma_{+}^{\Lambda})$ be a handle-slide, as defined in Definition 3.2.3, section \ref{chapitre whitehead}. 

In this subsection we show that the signs of the boundaries of the compactified one-parameter moduli spaces $\widehat{\mathcal{M}}^{1}_{\Lambda}(x^{\Lambda}_{-}, y^{\Lambda}),\widehat{\mathcal{M}}^{1}_{\Lambda}(x^{\Lambda}, x^{\Lambda}_{+})$ satisfy:
\smaller
\begin{align}
\partial \bar{\widehat{\mathcal{M}}}^{1}_{[\lambda_{0}-\epsilon, \lambda_{0}+ \epsilon]}(x^{\Lambda}_{-}, y^{\Lambda}) &= -\widehat{\mathcal{M}}_{\lambda_{0}-\epsilon}^{1}(x^{\lambda_{0}}_{-}, y^{\lambda_{0}}) \cup \widehat{\mathcal{M}}_{\lambda_{0}+\epsilon}^{1}(x^{1}_{-}, y^{1}) \cup -\{u\}\times\widehat{\mathcal{M}}_{\lambda_{0}}^{1}(x^{\lambda_{0}}_{+}, y^{\lambda_{0}})\\
\partial \bar{\widehat{\mathcal{M}}}^{1}_{[\lambda_{0}-\epsilon, \lambda_{0}+ \epsilon]}(x^{\Lambda}, x^{\Lambda}_{+}) &= -\widehat{\mathcal{M}}_{\lambda_{0}-\epsilon}^{1}(x^{\lambda_{0}}, x^{\lambda_{0}}_{+}) \cup \widehat{\mathcal{M}}_{\lambda_{0}+\epsilon}^{1}(x^{1}, x^{1}_{+}) \cup \widehat{\mathcal{M}}_{\lambda_{0}}^{1}(x^{\lambda_{0}}, x^{\lambda_{0}}_{-})\times\{u\}.
\end{align}
\normalsize
A similar statement appears as \cite[Section 7.3.2]{Le1}.
See subsection 4.2, section \ref{chapitre whitehead} to recall the definitions. 
 
Consider  the relation (61). Denote by $(\lambda, u)$ the handle-slide. Let $\chi_{-} = (u, v, \rho)$ be such that $\mathcal{G}_{-}(\hat{\chi}_{-}) = (\lambda_{-}, w_{\lambda_{-}}) \in \widehat{\mathcal{M}}^{1}_{\Lambda}(x^{\Lambda}_{-}, y^{\Lambda})$.

The determinant lines $Det(D_{u})$ and $Det(D_{v})$ are oriented as Theorem \ref{orientability} explains. Then we have the following isomorphisms:
\smaller
\begin{align}
Det(D_{u})\otimes Det(D_{v}) = & Ker(D_{u})\otimes (Coker(D_{u}))^{*}\otimes Ker(D_{v})\\ \cong& 
Ker(D_{u})\otimes Ker(D_{v})\otimes (Coker(D_{u}))^{*}\\ \cong & -(\bigwedge ^{max} Ker D_{u}\#_{\rho} Ker D_{v}) \otimes (Coker(D_{u}))^{*}\\ \cong &
-Det(D_{w_{\chi_{-}}}) \cong -Det(D_{w_{-}}).
\end{align}
\normalsize
The isomorphism (65) follows from the conventions for the gluing map in subsection 5.2.4.

The last isomorphisms in (66) follows from the closeness between the preglued trajectory and the glued trajectory. 

We conclude that sign$(u)$sign$(v)$ = -sign$(w_{-})$. Notice that the orientation of $Det(D_{(\lambda_{-}, w_{-})})$ is given by $-\mathbb{R}\oplus Ker(D_{u})$. Then, the base $\{sign(\lambda_{-})\frac{\partial}{\partial \lambda}, sign(w) \frac{\partial w_{-}}{\partial s})\}$ is a compatible base when  $sign(\lambda_{-})= -1$. From this we can conclude that $sign(u)sign(v) = sign(\lambda_{-})sign(w_{-})$.
In the same way we deduce (62). This concludes the verification.


\end{document}